\documentclass[11pt]{amsart}

\usepackage[english]{babel}
\usepackage[useregional]{datetime2}
\usepackage[colorinlistoftodos,bordercolor=orange,backgroundcolor=orange!20,linecolor=orange,textsize=small]{todonotes}

\usepackage{fullpage}
\usepackage{caption}
\usepackage{subcaption}
\usepackage{amsmath}
\usepackage{amsthm}
\usepackage{hyperref}
\usepackage{cleveref}
\usepackage{centernot}
\usepackage{blkarray}
\usepackage{float}
\usepackage[font=footnotesize,labelfont=bf]{caption}
\definecolor{burntorange}{cmyk}{0,0.52,1,0}
\usepackage{graphicx}
\usepackage{stackengine}
\stackMath
\usepackage{booktabs}
\newcommand{\tropprod}{\mathop{\odot}}
\newtheorem{theorem}{Theorem}[section]
\newtheorem{lemma}[theorem]{Lemma}
\newtheorem{proposition}[theorem]{Proposition}
\newtheorem{property}[theorem]{Property}
\newtheorem{corollary}[theorem]{Corollary}
\theoremstyle{definition}
\newtheorem{definition}[theorem]{Definition}

\theoremstyle{remark}
\newtheorem{remark}[theorem]{Remark}
\newtheorem{example}[theorem]{Example}

\usepackage[foot]{amsaddr}
\makeatletter
\renewcommand{\email}[2][]{%
  \ifx\emails\@empty\relax\else{\g@addto@macro\emails{,\space}}\fi%
  \@ifnotempty{#1}{\g@addto@macro\emails{\textrm{(#1)}\space}}%
  \g@addto@macro\emails{#2}%
}
\makeatother

\usepackage{graphicx}

\newcommand{\theory}{\mathrm{Th}}
\newcommand{\thVrcf}{\theory_{\mathrm{rcvf}}}
\newcommand{\new}[1]{{\em #1}}

\DeclareMathAlphabet{\mathbbold}{U}{bbold}{m}{n}
\newcommand{\zero}{\mathbbold{0}}
\newcommand{\unit}{\mathbbold{1}}
\newcommand{\zeror}{\mathbbold{0}} %
\newcommand{\unitr}{\mathbbold{1}} %
\newcommand{\mv}{m} %
\newcommand{\support}{\operatorname{supp}}
\newcommand{\C}{\mathbb{C}}

\newcommand{\R}{\mathbb R}
\newcommand{\vall}{\mathrm v}
\newcommand{\sval}{\mathrm{sv}}
\newcommand{\sign}{\mathrm{sgn}}

\newcommand{\pbool}{\mathrm{Res}}
\newcommand{\Val}{\mathrm V}
\newcommand{\dc}{\mathrm{dc}}

\newcommand{\smax}{\mathbb{S}_{\max}}
\newcommand{\rmax}{\mathbb{R}_{\max}}
\newcommand{\tmax}{\mathbb{T}_{\max}}
\newcommand{\bmax}{\mathbb{B}_{\max}}
\newcommand{\bmaxs}{{\mathbb B}_{{\mathrm s}}}

\newcommand{\PF}{\mathcal{P}_{\!\mathrm{f}}}%

\newcommand{\Y}{\mathsf{Y}}

\newcommand{\mult}{\mathrm{mult}}
\newcommand{\card}{\mathrm{card}}
\newcommand{\sat}{\mathrm{sat}}
\newcommand{\balance}{\,\nabla\,}
\newcommand{\Pn}{\normalize{P}}

\newcommand{\normalize}[1]{\,\overline{\!{#1}}} %
\newcommand{\surpass}{\trianglelefteq}

\newcommand{\formE}{\mathfrak{E}}
\newcommand{\formF}{\mathfrak{F}}

\newcommand{\ballext}{\mathcal{B}} %
\newcommand{\notbalance}{\!\centernot{\,\balance}}
\newcommand{\bp}{\mathbf{P}}

\usepackage[scr=boondox,scrscaled=1.05]{mathalfa}

\newcommand{\trop}[1][]{\ifthenelse{\equal{#1}{}}{ \mathbb{T} }{ \mathbb{T}(#1) }}

\usepackage{amssymb}

\renewcommand{\geq}{\geqslant}
\renewcommand{\leq}{\leqslant}
\renewcommand{\succeq}{\succcurlyeq}
\renewcommand{\le}{\leq}
\renewcommand{\ge}{\geq}

\newcommand{\botelt}{\bot}
\newcommand{\topelt}{\top}

\newcommand{\morphism}{\mu}
\newcommand{\Q}{\mathbb{Q}}
\newcommand{\N}{\mathbb{N}}
\newcommand{\Z}{\mathbb{Z}}

\DeclareMathOperator{\uval}{ldeg}
\newcommand{\coloneqq}{:=}

\newcommand{\morphismsys}{\varphi}
\newcommand{\resfield}{\mathscr{k}}
\newcommand{\hahnseries}[2]{#1[[t^{#2}]]}
\newcommand{\puiseuxseries}[1]{#1\{\{t\}\}}
\newcommand{\semiring}{\mathcal{A}} %
\newcommand{\extension}{\mathcal{E}}
\newcommand{\semiringvee}{\mathcal{A}^{\vee}} %
\newcommand{\tangible}{\mathcal{T}} %
\newcommand{\vfield}{\mathcal{K}}%
\newcommand{\rfield}{\mathcal{L}}%
\newcommand{\elf}{b} %
\newcommand{\rsmax}{r} %
\newcommand{\subgroup}{\mathcal G}
\newcommand{\vgroup}{\Gamma}
\newcommand{\vring}{\mathscr{O}}
\newcommand{\videal}{\mathscr{M}}
\newcommand{\res}{\operatorname{res}}
\newcommand{\hyper}{\mathcal{H}}

\newcommand{\angular}{\mathrm{ac}}
\newcommand{\xsec}{\mathrm{cs}}
\newcommand{\bifench}{{\mathrm c}}

\newcommand{\skewproductstar}[2]{#1{\rtimes}{#2}}

\begin{document}
\title{Factorization of polynomials over the symmetrized tropical semiring and Descartes' rule of sign over ordered valued fields}
\author{Marianne Akian$^{\, 1}$}
\author{Stephane Gaubert$^{\, 2}$}
\author{Hanieh Tavakolipour$^{\, 3}$}
\address[$1,2$]{Inria and CMAP, Ecole polytechnique, CNRS, Institut Polytechnique de Paris}
\address[$3$]{Amirkabir University of Technology, Department of Mathematics and Computer Science}
\email[$1$]{marianne.akian@inria.fr}
\email[$2$]{stephane.gaubert@inria.fr}
\email[$3$]{h.tavakolipour@aut.ac.ir}
\thanks{This work started during the postdoc time of the third author at Inria and CMAP, Ecole polytechnique, CNRS, Institut Polytechnique de Paris.
The study of the third author has been funded by an Inria postdoc position and by Iran National Science Foundation (INSF) (Grant No. 99023636)}
\date{\today}



\maketitle

\begin{abstract}
  The symmetrized tropical semiring is an extension of the tropical semifield,
  initially introduced to solve tropical linear systems using Cramer's rule.
  It is equivalent to the signed tropical hyperfield
  which has been used in the study of tropicalizations of semialgebraic sets. Polynomials
  over the symmetrized tropical semiring, and their factorizations,
  were considered by Quadrat. Recently, Baker and Lorscheid introduced a notion of multiplicity
  for the roots of univariate polynomials over hyperfields.
  In the special case of the hyperfield of signs, they related
  multiplicities with Descarte's rule of sign for real polynomials.
More recently, Gunn extended these multiplicity definition and characterization
to the setting of ``whole idylls''.
    We investigate here the factorizations of univariate polynomial functions over symmetrized
    tropical semirings, and relate them with the multiplicities of roots
    over these semirings. We deduce a Descartes' rule for ``signs and valuations'',
which applies to polynomials over a real closed field with a convex valuation and 
an arbitrary (divisible) value group.
We show in particular that the inequality of the Descartes' rule
is tight when the value group is non-trivial.
This extends to arbitrary value groups a characterization of Gunn in the rank one case,
answering also to the tightness question.
Our results are obtained using the 
framework of semiring systems introduced by Rowen,
together with model theory of valued fields.
\end{abstract}

\section{Introduction}
\subsection{Motivation}
The {\em max-plus} or {\em tropical} semifield, denoted 
$\rmax$,  is the set of real
numbers $\R$ with $-\infty$ endowed with $\max$ operation as addition
and usual addition as multiplication.
This algebra was introduced by several authors as a tool to solve 
problems of optimization, optimal control, discrete event systems
or automata theory, which motivated the study of basic questions of algebra. %
In particular, elementary notions like polynomials and factorization of polynomial functions in one variable goes back to Cuninghame--Green~\cite{cuninghame1980algebra}, who proved that any polynomial function can be factored uniquely as a product of linear factors. One of the difficulties of the above tropical algebra is that there are no opposites, in particular linear systems have rarely a solution. In an attempt to bypass this difficulty and in particular to solve linear systems by using Cramer formula, the symmetrized tropical semiring $\smax$ was introduced in~\cite{maxplus90b}, as an extension of $\rmax$. The construction of this extension is similar to the construction of $\mathbb{Z}$ as an extension of $\N$, except that $\smax$ contains not only the positive, negative, and zero numbers, called signed elements, which form the set  $\smax^\vee$,  but also ``differences'' of the same non-zero elements which play the role of quasi-zeros, and are called balance elements. Linear systems over $\smax$ have been studied in~\cite{maxplus90b} and also in \cite[Sec.\ 3.5]{baccelli1992synchronization} and in \cite{gaubert1992theorie,cramer-guterman}. Polynomials  over $\smax$, or rather over $\smax^\vee$, were considered by Quadrat in  \cite[Sec.\ 3.6]{baccelli1992synchronization}, where he showed that some polynomial functions may not be factored (as a product of linear factors), and established some necessary and/or sufficient conditions for factorization.
In the tropical semifield, the existence and unicity of factorization
of polynomial functions allowed one to define the multiplicity 
of roots as the number of factors with this root in the factorization.
In the  symmetrized tropical semiring, 
the multiplicity cannot be defined in this way in general, since the 
factorization may not hold or may not be unique.
We are thus interested in characterizing the situations in which the
factorization holds and is unique, 
and this is one of the aims of the present paper.

The tropical semifield is also related to the notion of non-Archimedean valuation over a field, since a valuation (or rather its opposite) on a field $\vfield$
can be viewed as a map from $\vfield$ to $\rmax$, that is almost a morphism.
The images by such a non-Archimedean valuation of algebraic subsets of $\vfield^n$ are called {\em non-Archimedean amoebas}, and are studied in tropical 
geometry~\cite{itenberg2009tropical,maclagan2015introduction}.
For instance, Kapranov's theorem
shows that the closure of the image by a non-Archimedean valuation of an algebraic hypersurface over an algebraically closed field is a tropical hypersurface, see~{\cite{kapranov}}.
When $n=1$, this means that for every polynomial over an algebraically closed
valued field, the valuations of its roots (counted with multiplicities)
are the ``tropical roots'' (counted with multiplicities) 
of its valuation (which is a tropical polynomial).
A basic example of an algebraically closed field with a non-Archimedean 
valuation is the field of complex Puiseux series, with the valuation which
takes the leading (smallest) exponent of a series.
In that case, the above result is nothing but
the characterization of the leading exponents of the  Puiseux series or 
asymptotic expansions of the different roots 
in terms of the slopes of the Newton polygon, 
which is part of the classical Newton-Puiseux theorem.

Hyperfields satisfy many of the properties of semifields, but with a
multivalued addition. The typical example of an hyperfield is obtained by taking the quotient a field by a subgroup (see for instance \cite{krasner,connesconsani,viro2010hyperfields} for an 
introduction to hyperfields and for the above construction).
In this case, the canonical projection from the field to its quotient hyperfield is a morphism of hyperfields (see for instance \cite{viro2010hyperfields,baker2018descartes}).
Moreover, for a valued field $\vfield$, considering the subgroup of elements with valuation $0$, one obtains the tropical hyperfield which is related to the tropical semifield, or rather the supertropical semiring of Izhakian (see for instance~\cite{IR,AGRowen}). Then, the valuation is equivalent to the canonical projection and thus is a morphism of hyperfields. 
These properties have motivated the study of hyperfields in relation
with tropical geometry. Moreover, several generalizations have been introduced
by Baker and Lorscheid, like the notions of ordered blueprints and idylls,
see for instance~\cite{Lor1,BL}.

In \cite{Rowen,Rowen2}, Rowen introduced the notion of
semiring system which involves a subset of tangible elements
and a surpassing relation, and in \cite{AGRowen},
it is shown that under some conditions, 
the set of tangible or zero elements of a semiring system 
 is isomorphic to a hyperfield, and that any hyperfield can 
be written as the set of tangible or zero elements of a semiring system.
In particular, the semiring system associated to
the supertropical semiring \cite{IR} is such that 
the set of tangible or zero elements coincides with $\rmax$,
and is isomorphic to 
the tropical hyperfield \cite{krasner,viro2010hyperfields,CC}.
Moreover, the semiring system associated to $\smax$ is such that 
the set of tangible or zero elements coincides with $\smax^\vee$,
and is isomorphic to 
the tropical real hyperfield of \cite{viro2010hyperfields}, also called
real tropical hyperfield in \cite{jell_scheiderer_yu},
and signed tropical hyperfield in \cite{gunn}.

General non-Archimedean valuations over fields 
take their values in an ordered abelian group $\vgroup$, and
similarly, the tropical semifield can be constructed over
any ordered abelian group (and shall be denoted $\tmax$ in the sequel),
so does the symmetrized  tropical semiring (for which we use 
the same notation $\smax$ or the notation $\smax(\vgroup)$).
In general, a basic example of a field with a general
valuation is the field of Hahn series with complex coefficients.
 Then, Kapranov theorem
may be obtained for a general non-trivial valuation over an algebraically
closed field, and tropical geometry
has also been developped in this generality, see for instance the works
of Aroca \cite{aroca,aroca2}, Joswig and Smith \cite{joswig_smith},
Kaveh and Manon \cite{Kaveh-Manon},
or  Iriarte and Amini \cite{iriarte2022polyhedral,amini_iriarte202}.
In particular, when $k\geq 2$, the additive group $\vgroup=\R^k$ with 
lexicographic order
is a value group with rank $k$, which allows one to study asymptotic 
expansions with $k$ infinitesimal parameters of different orders
of magnitude. 
Kapranov theorem 
over an algebraically closed field has also been extended recently 
by Maxwell and Smith \cite{maxwell} for some morphisms of hyperfields.

When considering the field of real Puiseux series,
or the field of real Hahn series,  one may be interested 
in characterizing not only,
as in Newton-Puiseux and Kapranov theorems, 
the leading exponents of the roots, but also their 
signs. Such parameters are gathered into the signed valuation, already
considered in \cite{allamigeon2020tropical,jell_scheiderer_yu}
to study tropicalizations of semi-algebraic sets.
The signed valuation is 
almost a morphism from
the real closed field of real Puiseux series, or Hahn series, to 
the semiring $\smax$ or more generally $\smax(\vgroup)$, 
which takes its values in the 
subset $\smax^\vee$ of signed elements.
This is also a morphism of hyperfields or of semiring systems.
This leads to the question of an analogue of Newton-Puiseux 
and Kapranov theorems
for signed valuations instead of valuations. We shall address 
here the case $n=1$. 
To answer to this question, one necessarily needs to define a 
notion of multiplicity of roots of one variable polynomials
over the symmetrized tropical semiring,
and we shall see that we also need to characterize the
factorization of polynomial functions.

When considering the trivial group $\vgroup=\{0\}$, we obtain that
$\smax$ is the symmetrized Boolean semiring,
and $\smax^\vee$ is isomorphic to the hyperfield of signs, for which some of the
above questions have been studied in \cite{baker2018descartes,lorscheidfac}.
Indeed, a notion of multiplicity of polynomials
has been introduced and studied by Baker and Lorscheid \cite{baker2018descartes}
in the context of hyperfields. 
In particular, they proved inequalities on multiplicities 
for morphisms of hyperfields, and %
applied their results to
the particular cases of the hyperfield of signs and of the tropical hyperfield.
In \cite{lorscheidfac}, Agudelo and Lorscheid also
studied  factorization of formal polynomials over the hyperfield of signs.
Moreover, in~\cite{gunn}, Gunn extended 
the results of \cite{baker2018descartes} to
the case of the signed tropical hyperfield.
Since morphisms of hyperfields are particular cases of morphisms 
of semiring systems \cite{AGRowen}, all the results of
\cite{baker2018descartes}, \cite{lorscheidfac} and \cite{gunn}
can be rewritten in the framework of semiring systems.
In the present paper, we
shall go further on these notions and results in the context of any
symmetrized tropical semiring constructed over any ordered group $\vgroup$,
and also in the context of more general semiring systems.
It is clear that the multiplicity notion defined in  \cite{baker2018descartes} 
can be applied to the roots in $\smax^\vee$ 
of polynomials with coefficients in $\smax^\vee$,
using the isomorphism of $\smax^\vee$ with an hyperfield
(for instance the signed tropical hyperfield when $\vgroup=\R$).
In \Cref{sec-multsym}, we also extend it to any semiring system. 
Moreover, the signed valuation can be seen as a morphism of hyperfields, 
then the inequalities on multiplicities for such morphisms 
shown in \cite[Prop.~B]{baker2018descartes} are leading to a weak analogue
of Kapranov theorem for signed valuation of univariate polynomials 
in which equalities of multiplicities are replaced by inequalities:
for every polynomial over a real closed field with a convex valuation,
the signed valuations of its roots in this field (that is its real roots),
are signed tropical roots of its signed valuation (which is a polynomial
over $\smax^\vee$), but the multiplicity of a signed tropical root $r$
of a polynomial's signed valuation may be greater than the sum of the
multiplicities of all the possible roots of the polynomial, with a
signed valuation equal to $r$.
In \cite{baker2018descartes}, it is also shown that the equality
holds when the given polynomial over the field is factored as a product of 
linear factors. In the case of a real closed field, this 
is equivalent to the property that the roots of the 
polynomial are all real (in the real closed field).
Another question however is to check that the inequalities on multiplicities
are tight, a question raised by Baker and Lorscheid \cite[Remark 1.14]{baker2018descartes} for general morphisms of hyperfields. 
For $\smax^\vee$, tightness means that for any polynomial $P$ over $\smax^\vee$, 
there exists a polynomial $\bp$ 
over the real closed field whose signed valuation
is $P$ and the sum of the
multiplicities of all the possible roots of $\bp$,
with same signed valuation $r$,
coincides with the multiplicity of $r$ as a root of $P$.
This is true when the group $\vgroup$ is trivial,  due to \cite{grabiner},
see  \cite[Remark 1.14]{baker2018descartes}.
Gunn gave a partial answer  for 
the real ordered group $\vgroup=\R$:
in \cite[Th.~5.4]{gunn} is is shown that 
there exists a polynomial over the real closed field  of Hahn series,
with for each root $r$ of $P$,
the expected number of roots %
in the algebraically closed complex field of Hahn series,
that have a real dominant term with a signed valuation equal to $r$.

In \cite{baker2018descartes}, Baker and Lorscheid also characterized
the  multiplicities of roots of 
polynomials over the hyperfield of signs, %
then specializing the inequalities on multiplicities of polynomials 
to the sign map from the real field to the  hyperfield of signs,
they  deduce the inequality part of 
Descartes' rule of signs of real polynomials:
the number of positive roots of a real polynomial is greater or equal to the
number of sign changes of its coefficients.
Gunn extended these results to 
the case of polynomials over the signed tropical hyperfield \cite{gunn}.
In particular, 
\cite[Th.\ A]{gunn} characterizes the multiplicities of 
the roots of such polynomials, by using the notion of initial forms,
with a proof similar to the one done in 
\cite{baker2018descartes} for the case of the hyperfield of signs.
This characterization is also related with the Descartes' rule of signs 
for polynomials over a real closed field with a valuation on a real
value group shown in \cite[Th.\ B]{gunn}, although the latter result is proved independently from \cite[Th.\ A]{gunn}.

All these results concern the above signed valuation, either when the group is trivial,
or the group is the real group.
One may then ask if similar results hold for a general ordered group, and also if 
characterizations of multiplicities can be obtained for more general hyperfields or
semiring systems.
In particular, such characterizations may be useful in further studies or in
other applications, where Descartes' rule of signs is involved.
For instance, Itenberg and Roy~\cite{itenbergroy} extended
Descartes' rule of signs to systems of polynomials with several variables, 
and such a 
result may be related to a weak analogue of 
Kapranov theorem for polynomial systems
and the signed valuation. 
Moreover,
Tonelli-Cueto
pointed out the analogy between Descartes' rule of 
signs and Strassman's theorem for polynomials with $p$-adic coefficients
in  \cite[Section 2]{Tonelli-Cueto}. This analogy may become clearer if
we consider the valuation map on a general valued field, with a general value group
$\vgroup$, which is a morphism of hyperfields from the valued field to
the tropical semifield $\tmax$ associated to $\vgroup$.
Indeed,  Strassman's theorem can be seen as an application of the inequalities on multiplicities of \cite[Prop.~B]{baker2018descartes} to the field of $p$-adic numbers with a value group
equal to the group of all integers $\Z$.

After having posted the initial version of this manuscript on arxiv, we learnt
from Trevor Gunn that some of the results
of \cite{baker2018descartes} and \cite{gunn} 
were  extended in~\cite{gunn2} to
more general algebraic structures called ``whole idylls''.
In particular, \cite[Th.\ A and B]{gunn2} extends the characterization of
roots multiplicities to the case of polynomials over a ``whole idyll'',
an algebraic structure which includes hyperfields as a special case.
More generally, idylls were introduced in~\cite{BL}, as particular cases of ordered blueprints.
They can also be realized as particular cases of semiring systems~\cite{AGRowen}. However, the correspondance is somehow formal, since it is obtained by taking the quotient of a blueprint by its group.
We will compare the different approaches and results in \Cref{sec-multsym},
see in particular \Cref{rk:mult-smax}.

\subsection{Main results}
In order to study the above problems, 
we shall use the {\em semiring systems} framework. %
Semiring systems avoid the use of a multivalued addition compared to
hyperfields. They also avoid to keep a large set of formal expressions like blueprints. They are also more
adapted to usual algebraic constructions, like polynomials and matrices~\cite{AGRowen}, and also allow one to handle equations in a more symmetric way than 
hyperfields.

We consider in particular layered semiring systems \cite{AGRowen},
which are extensions of a ``base'' semiring system by a tropical semifield.
These include in particular
the  supertropical semiring and the symmetrized tropical semiring.
These semirings arise as the extensions of their Boolean versions,
which correspond to the Krasner hyperfield and the hyperfield of signs, respectively.
We obtain in \Cref{sec-multsym}, especially in~\Cref{tho-mult-smax},
a general characterization of multiplicities of roots of polynomials in 
a layered semiring system using multiplicities over the base semiring system.

\Cref{tho-mult-smax} allows one to
deduce the multiplicities over the symmetrized tropical semiring
from the ones over the hyperfield of signs characterized by Baker and Lorscheid~\cite{baker2018descartes}.
This is stated in \Cref{cor-mult-smax}, a result which generalizes
 the characterization of such multiplicities by Gunn~\cite{gunn}.
However, \Cref{tho-mult-smax} is of a more general interest, and indeed 
a similar result was obtained recently
in \cite[Th.\ A and B]{gunn2}, working with ``whole idylls''.
We provide a detailed comparison with the results of \cite{gunn2} in~\Cref{sec-multsym}, see in particular~\Cref{rk:mult-smax}.

The remaining main results of the paper concern the case of 
the symmetrized tropical semiring.
One part concerns the relation between multiplicities of roots and factorization. We first
give in \Cref{suf_cond} sufficient conditions for factorization which improve
the ones of \cite[Sec.\ 3.6]{baccelli1992synchronization}, and also
give a sufficient condition for unique factorization.
The proofs only use the properties of formal
polynomials over the tropical semifield $\tmax$.
These conditions are refined in \Cref{sec-multtact},
by applying \Cref{cor-mult-smax}.
We also determine the multiplicities of roots in case of unique factorization in \Cref{coro2-uniquefact}.
In particular \Cref{prop-mult-fact} shows that a polynomial function over the symmetrized
tropical semiring whose roots have multiplicities which sum to $n$ can be factored as a product of linear factors. The converse property does not hold,
in stark contrast with usual real algebra.

The second part concerns the characterization of
 the  signed valuations of polynomials over an
ordered field with a convex valuation, see \Cref{sec-valroots}.
Such a characterization is 
obtained as an application of the results of \Cref{sec-multtact},
and generalizes the Descartes's rule of signs for a
real closed field with an arbitrary non-trivial convex valuation.
Moreover, as said before, it is a weak analogue of Kapranov theorem.
The inequality part of Descartes's rule of signs
follows from the inequalities on multiplicities 
obtained by Baker and Lorscheid in \cite{baker2018descartes}.
Then, 
in \Cref{th-descartes}, %
we obtain the ``modulo 2'' part and 
the tightness of Descartes's rule of signs. %
These properties are obtained first by showing
them in a special case, and then, by exploiting the completeness of
the first order theory of real closed valued fields,
which follows from Denef-Pas quantifier elimination~\cite{denef_p-adic_semialgebraic,pas_cell_decomposition,pas_on_angular_component}.
The tightness result refines the one of Gunn \cite[Th.~5.4]{gunn}, and answers to a question raised there.
\bigskip

The paper is organized as follows.
First, in \Cref{sec-max}, we review some basic definitions, notations and results of max-plus or tropical algebra and their polynomials and also the related non-Archimedean valuation theory. These constructions are done as much as possible over any ordered group, or at least any divisible ordered group. Then, in \Cref{sec-sym}, we review some preliminaries on the symmetrized tropical semiring $\smax$, the results of Quadrat on the factorization of polynomial functions over $\smax$ in \cite[Sec.\ 3.6]{baccelli1992synchronization} (presented here for an arbitrary value group), the signed valuation over ordered valued fields, and the notions of hyperfields,
semiring systems and their extensions (layered semiring systems).
In \Cref{sec-factsym}, we give new sufficient conditions for a polynomial
function over  $\smax$ to be factored, and for this factorization to be unique.
In \Cref{sec-multsym}, we generalize and study 
the notion of multiplicity of a root of
a formal polynomial of~\cite{baker2018descartes} to the 
setting of a semiring systems and layered semiring systems.
In \Cref{sec-multtact}, we revisit factorization of polynomials over $\smax$ at the light of the multiplicity results of  \Cref{sec-multsym},
while using the results of \Cref{sec-factsym}.
In \Cref{sec-valroots}, for an ordered field with a convex valuation, we relate roots of the signed valuation of a polynomial with the signed valuations of its roots.

\section{Preliminaries on tropical algebra and non-Archimedean valuation}\label{sec-max}
In this section, we recall some basic definitions, notation and results of tropical algebra, and discuss their relation with non-Archimedean valuation theory.

\subsection{Tropical semifields}\label{subsec-def-trop}
A {\em commutative semiring} is a set $\semiring$ equipped with an addition $(a,b)\mapsto a\oplus b$ that is associative, commutative, and has a neutral element $\zero$, together with a multiplication $(a,b)\mapsto a\odot b$ that is
associative, commutative, and has a neutral element $\unit$.
We require that the multiplication distributes over the addition
and that $\zero$ is an absorbing element for multiplication.
A semiring is {\em idempotent} if its addition is idempotent,
meaning that $a\oplus a =a$.
A (commutative) {\em semifield} is a semiring in which every element
different from $\zero$ has a multiplicative inverse. 
An idempotent semiring
is equipped with the natural (partial) order $\preceq$, defined by
\begin{equation}\label{order_max}
a \preceq b \quad\text{if}\quad a \oplus b = b \enspace .
\end{equation}
We refer the reader to~\cite{baccelli1992synchronization} for background
on \textit{semirings} and \textit{semifields}. %

A general family of idempotent semifields arises as follows.
Given a (totally) ordered abelian group $(\vgroup,+,0,\leq)$,
the {\em tropical semifield} (over $\vgroup$) is defined as
\begin{align}
  \label{tmax}
  \tmax(\vgroup): = (\vgroup \cup\{\botelt\}, \max, +) \enspace,
  \end{align}
where $\botelt$ is the %
bottom element of $\tmax(\vgroup)$ ($\botelt \leq a$ for all $a\in\vgroup$)
-- which does not belong to $\vgroup$.
The addition of $\tmax(\vgroup)$ is given by
$(a,b) \mapsto a\oplus b:= \max(a,b)$ and has zero element $\botelt$ which will be also denoted by $\zero$.
The multiplication is given by $(a,b)\mapsto a\odot b:= a+b$,
and $\botelt \odot a=a\odot \botelt= \botelt$,  for all $a,b\in \vgroup$,
so that it has unit $0$, which will be also denoted by $\unit$, and the zero
element $\botelt$ is absorbing.
The natural order $\preceq$ defined in \eqref{order_max}
coincides with the order $\leq$ of $\vgroup$. 
In particular, we shall use the notation $\leq$ instead of $\preceq$ in
$\tmax(\vgroup)$.
The $n$th-power of an element $a\in\vgroup$ for the multiplicative law $\odot$,
$a^{\odot n}:=a\odot \dots \odot a$  ($n$-times), coincides with the
sum  $a+ \dots + a$ ($n$-times), also denoted by $na$.

The special case in which $\vgroup=\R$ and $\botelt=-\infty$
is the most standard one in tropical geometry.
Then, $\tmax(\R)$ is the so-called
{\em max-plus semifield}, denoted by $\rmax$. 
The case in which $\vgroup$ is trivial
is also of interest, since $\tmax(\{0\})=\{\zero,\unit\}$ is nothing
but the {\em Boolean semifield}, denoted by $\bmax$ in the sequel.
We say that the group $\vgroup$ is {\em divisible},
if for all $a\in \vgroup$ and for all positive integers $n$,
there exists $b$ such that $nb=a$. %
In this case, $b$ is unique (since $\vgroup$ is ordered).
Moreover, the action $(n,b) \mapsto nb$ of natural numbers on $\vgroup$
canonically extends to an action of rational numbers on
$\Gamma$, and for this action, $\Gamma$ becomes
a module.
Several standard notions and properties stated in the litterature
for the  max-plus semifield can be extended to any tropical semifield 
$\tmax(\vgroup)$, at least when $\vgroup$ is divisible or
non-trivial and divisible.
In the sequel, we shall consider as much as possible
a general ordered abelian group $\vgroup$, and will specify the assumptions 
on $\vgroup$ when needed.
Then, when the particular choice of the group $\vgroup$ will be irrelevant,
or clear from the context,
we will write $\tmax$ instead of $\tmax(\vgroup)$, for brevity.

\subsection{Formal polynomials and polynomial functions in $\tmax$}\label{sec-form-rmax}

{\em We consider here the semifield $\tmax=\tmax(\vgroup)$ associated
to a general ordered abelian group $\vgroup$. Often $\vgroup$ will need to be
divisible and sometimes non-trivial in order to obtain the same 
properties as in $\rmax$.}

A \textit{formal polynomial} $P$ over $\tmax$ is defined as usual: it is a sequence $(P_k)_{k\in \N} \in \tmax$, where $\N $ is the set of natural numbers (including $0$),  such that $P_k=\zeror$ for all but finitely many values of $k$. We denote by $\tmax[\Y]$, the set of formal polynomials. This set is endowed with the following two internal operations, which make it an idempotent semiring:
coefficient-wise sum, $(P \oplus Q)_k=P_k \oplus Q_k$; and 
Cauchy product, $(P Q)_k= \bigoplus_{0 \leq i \leq k}P_i\odot Q_{k-i}$.
Formal polynomials are equipped with the entry-wise partial order $\leq$, such that $P \leq Q$ if $P_k \leq Q_k$, for all $k \in \N$. We denote a formal polynomial $P$ as a formal sum, $P = \bigoplus_{k=0}^{\infty} P_{k}  \Y^{k}$. The \textit{degree} and \textit{lower degree} of $P$ are respectively defined by
\begin{equation}\label{deg}\deg(P):=\sup\{k \in \N \mid P_k \neq \zeror\},
  \;
  \uval (P) := \inf\{k \in \N\;|\;P_k \neq \zeror\}.
\end{equation}
(The term ``valuation'' is often used for the ``lower degree'',  but we
prefer the latter terminoligy since valuation will be used in a broader sense below.). When $P = \zeror$, we have $\deg(P)= -\infty$ and $\uval(P) = +\infty$. The \textit{support} $\mathrm{supp}(P)$ is the set of indices of the non-zero elements of $P$, that is, 
\begin{equation}\label{supp}\mathrm{supp}(P):=\{k\in \N \mid P_k \neq \zeror\}.\end{equation}
A formal polynomial consisting of a sequence with only one non-zero element is called a \textit{monomial}. We say that a formal polynomial has a \textit{full support} if 
\[P_k\neq \zeror, \quad \forall k: \mv \leq k \leq n,\; \text{with}\;\deg(P)=n,\; \uval(P) =\mv\enspace .\]
To any $P \in \tmax[\Y]$, with degree $n$ and lower degree $\mv$,
we associate a \textit{polynomial function} 
\begin{equation}\label{widehat_p}\widehat{P}: \tmax \rightarrow \tmax, \;y \mapsto \widehat{P}(y)= \bigoplus_{k=\mv}^{n}P_{k}\odot y^{\odot k}=
\max_{m\leq k \leq n}(P_k+ky)\enspace , 
\end{equation}
where the last expression uses the standard notation of the order group $\vgroup$.  We denote by $\PF(\tmax)$, the set of polynomial functions $\widehat{P}$. 
When $\vgroup=\R$, this is a (conventional) piecewise linear function. 
 
One crucial difference between tropical polynomials and polynomials over the
field of real numbers is that the surjective mapping $\mathcal{F}:\tmax[\Y] \mapsto\PF(\tmax)$, $P \rightarrow \widehat{P}$ is not injective. In other words, we can find polynomials $P$ and $Q$ with different coefficients such that
$\hat P(y)=\hat Q(y)$ for all $y \in \tmax$. For example,
\[\forall y\in \tmax , \;y^{\odot 2}\oplus \unit = (y \oplus\unit)^{\odot 2}= y^{\odot 2} \oplus y \oplus \unit.\]This is why it is important to distinguish between formal polynomials and polynomial functions. See also 
 \Cref{ex_poly1}.

Let us now consider the fundamental theorem of tropical algebra related to the unique factorization of polynomial functions, after stating two definitions.
\begin{definition}[Corners] \label{def_corners}
Given a tropical formal polynomial $P$ over $\tmax$,
 the non-$\zeror$ corners are the points at which the maximum 
in the definition \eqref{widehat_p} of the associated polynomial function as a supremum of monomial functions,
is attained at least twice (i.e.\ by at least two different monomials). If $P$ has no constant term, then $P$ is said to have a corner at $\zeror$ ($\botelt$ in $\vgroup$).
\end{definition}
When $\vgroup=\R$, the non-zero corners of $P$ are equivalently the points of non-differentiability of the associated polynomial function $\widehat{P}$ restricted to $\R$.

\begin{definition}[Multiplicity of a corner]
The \textit{multiplicity} of a corner $c\neq \zeror$ of the formal polynomial $P$ is the difference between the largest and the smallest exponent of the monomials of $P$ which attain the maximum at $c$. If $P$ has a corner at $\zeror$, the multiplicity of $\zeror$ is the lower degree $\uval{P}$ of $P$, defined in \eqref{deg}.
\end{definition}
When $\vgroup=\R$, the multiplicity of a corner $c\neq \zeror$ is equivalently the change of slope of the graph of the  associated polynomial function $\widehat{P}$ at $c$.

\begin{remark}
In tropical geometry, the set of corners of a polynomial $P$ in one variable
is the tropical variety associated to $P$. Corners are also called 
tropical roots. We use the name ``corner'' introduced by 
Cuninghame--Green and Meijer in \cite{cuninghame1980algebra}, to avoid confusion
with the roots in the symmetrized semiring as in  \Cref{sec-polysym}.
\end{remark}

We have in $\tmax[\Y]$ the following analogue of the fundamental theorem of algebra which was shown by Cuninghame--Green and Meijer for $\rmax$, that is in the case where $\Gamma=\R$.
\begin{theorem}[Corollary of~\protect{\cite{cuninghame1980algebra}}]\label{factor}
  Suppose that $\vgroup$ is divisible, and let $\tmax=\tmax(\vgroup)$.
Then, every formal polynomial  $P \in \tmax[\Y]$ of degree $n$ has exactly $n$ corners $c_1\geq \cdots \geq c_n$, and the associated polynomial function $\widehat{P}$ can be factored in a unique way as 
\[\widehat{P}(y)=  P_n\odot (y \oplus c_1) \odot \cdots \odot (y \oplus c_n)
\enspace. \]
\end{theorem}
The result is established in~\cite{cuninghame1980algebra} when $\vgroup=\R$, i.e., for polynomials over $\rmax$. A way to prove~\Cref{factor} would be to check that the arguments of~\cite{cuninghame1980algebra} carry over to $\tmax(\vgroup)$ where $\vgroup$ is divisible.
However, we shall
repeatedly need to extend results of this kind, so we next present the general
model theory argument, allowing one to deduce a specific class of results over $\tmax(\vgroup)$ with $\vgroup$ non-trivial and divisible from their analogues over $\rmax$.
Sometimes, as for \Cref{factor}, these results are also extended to the
trivial group $\vgroup$.
We rely on a classical theorem, showing that the first order theory
of non-trivial totally ordered divisible groups is complete~\cite{robinson}.
Actually, we shall need
a straightforward extension of this result,
which applies to non-trivial totally ordered
divisible groups completed with a bottom element $\botelt$.
The language of this theory consists of the symbols $(0,\botelt, +,\leq)$.
Its axioms are the non-triviality axiom 
$\exists y (y \neq 0 \land y \neq \botelt)$,
the axioms that describe the behavior of the bottom element $\botelt$ with
respect to the addition and order, i.e.,
$\forall y ({\botelt} + y = \botelt)$ and $\forall y (y \ge \botelt)$,
and the axioms that specify that $\vgroup$ is a totally ordered abelian group
with neutral element $0$. %
\begin{proposition}[{See~\cite[Prop.~8]{allamigeon2020tropical}}]
  \label{prop-botiscomplete}
  The first order theory of non-trivial totally ordered
  divisible groups with a bottom element is complete.
  \end{proposition}
\begin{proof}[Proof of~\Cref{factor}]
  Let us fix a positive integer $n$. Consider
  the following formula in the variables $P_0,\dots,P_n$
  and in the variables $c_1,\dots,c_n$:
  \[
  \formF (P,c): \qquad \forall y (\widehat{P}(y)=P_n \odot (y \oplus c_1) \odot \cdots \odot (y \oplus c_n) )\enspace .
  \]
  This is a first order formula in the theory of non-trivial totally ordered
  divisible groups with a bottom element.
  Let $\mathfrak{S}_n$ denote the symmetric group on $\{1,\dots,n\}$,
  and consider the formula in the variables $c_1,\dots,c_n,d_1,\dots,d_n$,
  which states that $c$ and $d$ coincide up to a permutation:
  \[
  \formE (c,d):\qquad \bigvee_{\sigma\in \mathfrak{S}_n}\bigwedge_{i=1,\ldots, n}
(c_i=d_{\sigma(i)} )\enspace .
  \]
  Then, the conclusion of~\Cref{factor} can be expressed
  as the first order sentence
  \[
 \forall P_0,\dots, \forall P_{n-1},
\forall P_n\neq \botelt,  (\exists c, \formF (P,c))  \wedge (\forall c, \forall d, \formF (P,c) \wedge \formF (P,d) \Rightarrow \formE (c,d)) \enspace ,
  \]
  where $\forall c$ is a shorthand for $\forall c_1,\dots, \forall c_n$,
  and similarly for $\exists c$ and $\forall d$. 
  By~\cite{cuninghame1980algebra}, this statement is valid when
  the underlying group is $\R$. So, by~\Cref{prop-botiscomplete},
  it is valid fo an arbitrary non-trivial totally ordered divisible group $\vgroup$. This shows that~\Cref{factor} holds in $\tmax(\vgroup)$, when
$\vgroup$ is a non-trivial totally ordered divisible group.
When $\vgroup$ is trivial, the result follows from the one in $\rmax$, since $\tmax(\{0\})$ is a subsemifield of $\tmax(\R)$, and the only possible corners are $\botelt$ and $0$.
  \end{proof}
In the sequel, we will use, without further
comments, in the setting of $\tmax(\vgroup)$,
properties already proved in the special case of $\tmax(\R)$,
when the validity of these properties
follows from a straighforward completeness argument, as in the proof
of~\Cref{factor}.

Convex analysis offers an alternative way of visualizing formal polynomials and corners. This is done usually when $\vgroup=\R$. To the formal polynomial $P \in \tmax[\Y]$, where $P = \bigoplus_{k=0}^{\infty} P_{k}  \Y^{k}$, let us associate the map
\[\mathrm{coef} P: \mathbb{R} \mapsto \tmax, \; (\mathrm{coef} P)(x)= \begin{cases}P_k& \text{if} \;x=k \in \N\\ \zeror& \text{otherwise}.
\end{cases} \]
When $\vgroup=\R$, then the map $\mathrm{coef}P$ is an extended real
valued function, and 
 the polynomial function $\widehat{P}(y)$ corresponds to the
{\em Legendre-Fenchel transform}~\cite{rockafellar1970convex} of the map $-\mathrm{coef}P$ (notice the minus sign). %
Let us generalize the usual definition of the Newton polygon as follows,
in order to handle any {\em ordered divisible group $\vgroup$}.

\begin{definition}[Newton polygon]
Assume $\vgroup$ is divisible.
Let $P= \bigoplus_{k=0}^{\infty} P_{k} \Y^{k} \in \tmax[\Y]$. The \textit{rational Newton polygon} of $P$ is the upper boundary of the rationally convex hull in $\Q\times \vgroup$ of the set of points $\{(k, P_k)\;| k\in \N, \; P_k \neq \zeror\}$. Here we say that a subset $S$ of $\Q\times \vgroup$ is rationally convex if for all $x,y\in S$ and $t\in [0,1]\cap \Q$, we have $(1-t)x + ty\in S$.
\end{definition}
When $\vgroup=\R$, the (usual) Newton polygon is the closure in $\R^2$ of the
rational Newton polygon,  and it is the upper concave hull of the map
$\mathrm{coef}P$.
In general, it is the upper concave hull of the restriction of the map
$\mathrm{coef}P$ to $\Q$.

The following standard result follows from the properties of Legendre-Fenchel transform~\cite{rockafellar1970convex}.
\begin{proposition}[See for instance~\protect{\cite[Prop.~2.4]{akian2016non}}]
\label{corners_polygon}
Let $P \in \rmax[\Y]$ be a formal polynomial. The corners of $P$ coincide with the opposite of the slopes of the Newton polygon of $P$. The multiplicity of a corner $c$ of $P$ coincides with the length of the interval where the Newton polygon has slope $-c$.
\end{proposition}

To discuss about the possibility of factorization for formal polynomial $P \in \tmax[\Y]$, we shall adapt convex analysis techniques to the case of divisible groups.
\begin{definition} 
Assume $\vgroup$ is a divisible group.
With a given $P \in \tmax[\Y]$, we associate its {\em concave hull}, which is the element of $\tmax[\Y]$ denoted by $P^{\sharp}$, such that, for all $k \in \N$, 
\begin{equation}\label{p_sharp}P^{\sharp}_k:= \max \{(1-t) P_i+t P_j\mid i,j\in\N,\; t\in \Q\cap [0,1],\; \text{s.t.}\; (1-t)i+tj=k\}\in \Gamma\cup\{\bot\}\enspace.\end{equation}
Equivalently, the map $k\in \N\mapsto P^{\sharp}_k$ is the restriction to $\N$ of the upper rationally concave hull of the map $\mathrm{coef} P$ (which graph is the rational Newton polygon of $P$).
\end{definition}
Observe that, for all $k\in \N$, the set of values $t\in \Q\cap [0,1]$ such that
$(1-t)i+tj=k$ from some $i,j\in\N$ such that $P_i\neq \bot$ and $P_j\neq \bot$
is finite, so the maximum in~\eqref{p_sharp} is well defined.
Using equivalent definitions of concave functions, we get the following result.
\begin{lemma}[See~\protect{\cite[p.\ 123]{baccelli1992synchronization}} for $\vgroup=\R$] \label{roots_poly}
  Assume $\vgroup$ is a divisible group.
For any formal polynomial $P= P_n \Y^n \oplus  \cdots \oplus P_{\mv} \Y^\mv$,
 of lower degree  $\mv$ and degree $n$, we have 
\begin{enumerate}
\item $P = P^\sharp$, if and only if $P$ has full support (that is $P_k\neq \zeror$, for $k=\mv, \ldots, n$), and  
\begin{equation}\label{concave}
P_{n-1}-P_n \geq P_{n-2}-P_{n-1} \geq \cdots \geq P_{\mv}-P_{\mv +1}.\end{equation}
\item If $P=P_n (\Y \oplus c_1)\cdots (\Y \oplus c_n)$ (where $c_i$ may be equal to $\zeror$), then $P=P^\sharp$.
\item
Conversely, if $P=P^\sharp$, then the numbers $c_i \in \tmax$ defined by 
\[c_i := \begin{cases}
P_{n-i} - P_{n-i+1}& 1 \leq i \leq n-\mv;\\
\zeror & n-\mv <i \leq n.
\end{cases}
\]
are such that 
$c_1  \geq \cdots \geq c_n$
and $P$ can be factored as 
\[P=P_n (\Y \oplus c_1)\cdots (\Y \oplus c_n).\]
\end{enumerate}
If $P$ satisfies one of the above conditions, we shall say that
$P$ is {\em factored}.
\end{lemma}

\begin{definition} 
Assume $\vgroup$ is a non-trivial divisible group.
With a given $P \in \tmax[\Y]$, we associate the \new{canonical form} of the polynomial function $\widehat{P}$, which is the element of $\tmax[\Y]$, denoted by $P^\bifench$, such that, for all $k \in \N$, 
\begin{equation}\label{p-sharp2}
 P^\bifench_k:=\inf_{y\in \vgroup} \widehat{P}(y)-k y \enspace ,
\quad \text{for }\; k\in \N\enspace.\end{equation}
\end{definition}
We use here the name {\em canonical form} as in \cite[Def.\ 3.39]{baccelli1992synchronization}.
\begin{lemma}[See~\protect{\cite[Th.\ 3.38, Point 1]{baccelli1992synchronization}} for $\vgroup=\R$]\label{max-lemma} 
  Assume $\vgroup$ is a non-trivial divisible group.
  Then, for all $k\in \N$, the infimum in~\eqref{p-sharp2} is achieved
  or equal to $\bot$. 
  Moreover, for any non-zero
  $P \in \tmax[\Y]$, with lower degree $m$ and degree $n$,
the \new{canonical form} $P^\bifench$ of $\widehat{P}$ coincides with
$P^\sharp$, it has same lower degree and degree than $P$, and the infimum in \eqref{p-sharp2} is achieved 
when 
$k\in  \{m,\ldots ,n\}$.
Moreover, $P^\bifench$ is the maximum of the formal polynomials $Q$ such that $\widehat{Q}=\widehat{P}$ for the order $\leq$.
\end{lemma}
\begin{proof}
When $\Gamma=\R$, the result follows from Legendre-Fenchel duality in classical convex analysis theory, in the particular case of polyhedral functions, see \cite[Th.\ 3.38]{baccelli1992synchronization}. 
  Then, it carries over to any non-trivial divisible group using~\Cref{prop-botiscomplete}. Observe in this respect that, again fixing the degree of $P$,
  for each $k\in \N$, the property
  that either the infimum~\eqref{p-sharp2} is achieved or the family
  $(\widehat{P}(y)-k y)_{y\in \Gamma}$ is unbounded from below can
  be expressed by a first order formula in the theory of non-trivial
  divisible ordered groups, in which the coefficients of $P$ are
  variables. 

  However, the arguments of the proof of \cite[Th.\ 3.38, Point 1]{baccelli1992synchronization} are sketchy. So, for the convenience of the reader, we include an alternative, detailed, proof
 in Appendix (which also avoids the recourse to model theory).
\end{proof}  

In \Cref{ex_poly1} we demonstrate a pathological example for the notion of Newton polygon, when $\vgroup=\R$, and we will discuss the result of \Cref{corners_polygon}.
 \begin{example}\label{ex_poly1}
Take $\vgroup=\R$, and   
consider the tropical formal polynomial $P = \Y^5\oplus 4 \Y^3\oplus \Y \oplus 1$. We illustrate below the above definitions and properties on the formal polynomial $P$, in particular, the graphs of $P$ and of the polynomial function $\widehat{P}$ in \eqref{ex1}, the corners and their multiplicities in \eqref{ex2} and \eqref{ex3},
factorization of $\widehat{P}$ in \eqref{ex4}, and computation of $P^{\sharp}$ in
\eqref{ex5}.
 \begin{enumerate}
 \item\label{ex1}
  In \Cref{a}, we show the graph of $P$ and its associated polynomial function $\widehat{P}$ in $\tmax^{2}$ plane.  It is immediate to see that the graph of the polynomial function is the supremum of its monoimials $1$, $y$, $3y+4$ and $5y$. One can check that the line with slope $1$ does not appear in $\widehat{P}$. In other words, if $Q = \Y^5\oplus 4 \Y^3 \oplus 1$, the supremum of the monimials of $P$ and $Q$ are the same, so their associated functions $\widehat{P}$ and $\widehat{Q}$ coincide. 

\item\label{ex2}
In order to compute the corners and multiplicities of $P$, we compute the variation of slope of $\widehat{P}$ in \Cref{a} at non-differentiable points $-1$ and $2$ which are $3$ and $2$, respectively. 

\item\label{ex3}
In \Cref{b}, we have the plot of Newton polygon of $P$. The blue dots represent the coefficients of $P$. $(k,P_k)\; k=0, 1,3,5$. Note that $(1,0)$ is related to the the line with slop $1$ in \Cref{a}. Using the Newton polygon of $P$, the length of the intervals with slopes $-2$ and $1$ are equal to $2$ and $3$, respectively. Thus the corners of $P$ are $c_1=2$ and $c_2=-1$ with respective multiplicities $2$ and $3$.
\item\label{ex4}
Using the results \Cref{ex1} and \Cref{ex2} together with \Cref{factor} one can get the following factorization for $\widehat{P}$
   \[\widehat{P}(y) =   (y \oplus 2)^{\odot 2}\odot (y \oplus -1)^{\odot 3}.\]

   \item\label{ex5}
 To compute $P^\sharp$ by \Cref{p_sharp}, we have
   \[P^{\sharp}_0=1,\; P^{\sharp}_1=2, \;P^{\sharp}_2=3, \;P^{\sharp}_3=4,\; P^{\sharp}_4=2, \;P^{\sharp}_5=0.\]
   So 
   \[P^{\sharp} = \Y^5 \oplus 2  \Y^{4} \oplus 4 \Y^{3} \oplus 3  \Y^2 \oplus 2\Y  \oplus 1.\]
   Note that as we mentioned before, $P^{\sharp}$ has a full support. In \Cref{c} we demonstrate the graph of $P^{\sharp}$ and its associated polynomial function $\widehat{P^{\sharp}}$. It is clear that $\widehat{P^{\sharp}}=\widehat{P}$. Also \Cref{d}, represents the Newton polygon of $P^{\sharp}$.
  
 \end{enumerate}
\begin{figure}
     \centering
    \begin{subfigure}[b]{0.48\textwidth}
         \centering
         \includegraphics[width=\textwidth]{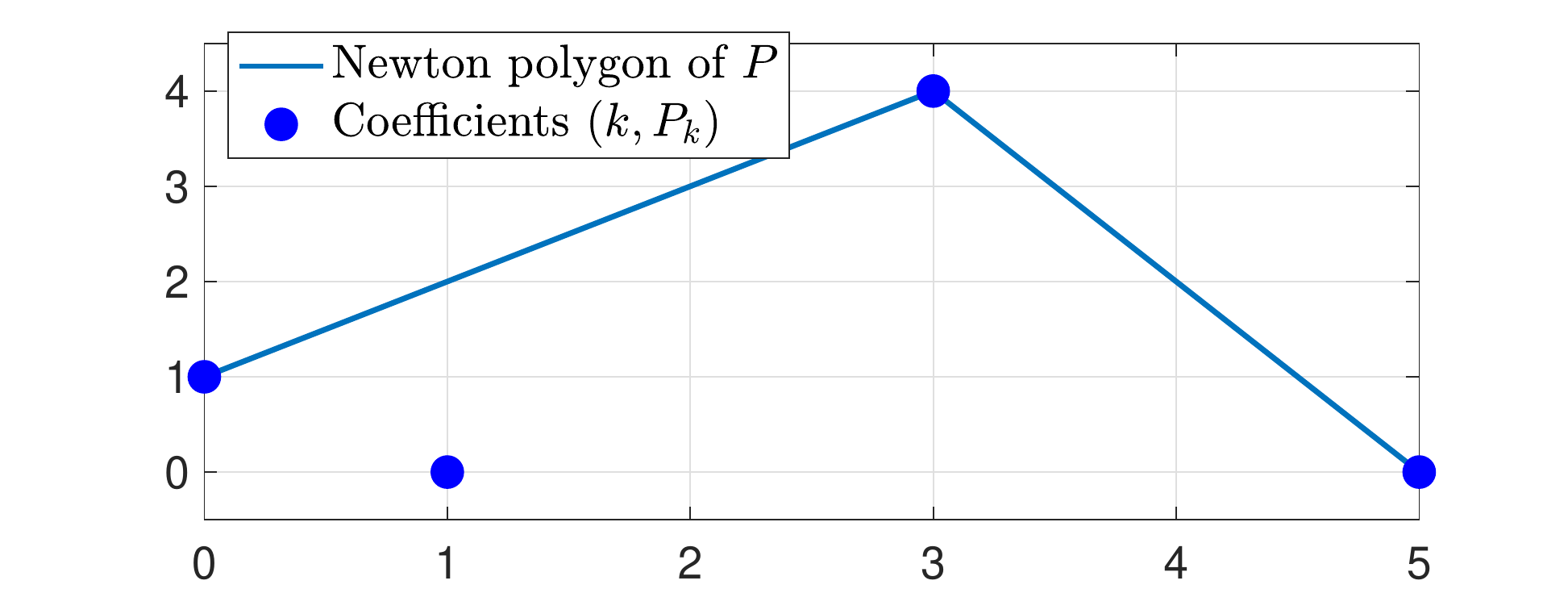}
         \caption{Configuration of points of Newton polygon of $P$.}
         \label{b}
     \end{subfigure}
     \hfill
         \begin{subfigure}[b]{0.48\textwidth}
         \centering
         \includegraphics[width=\textwidth]{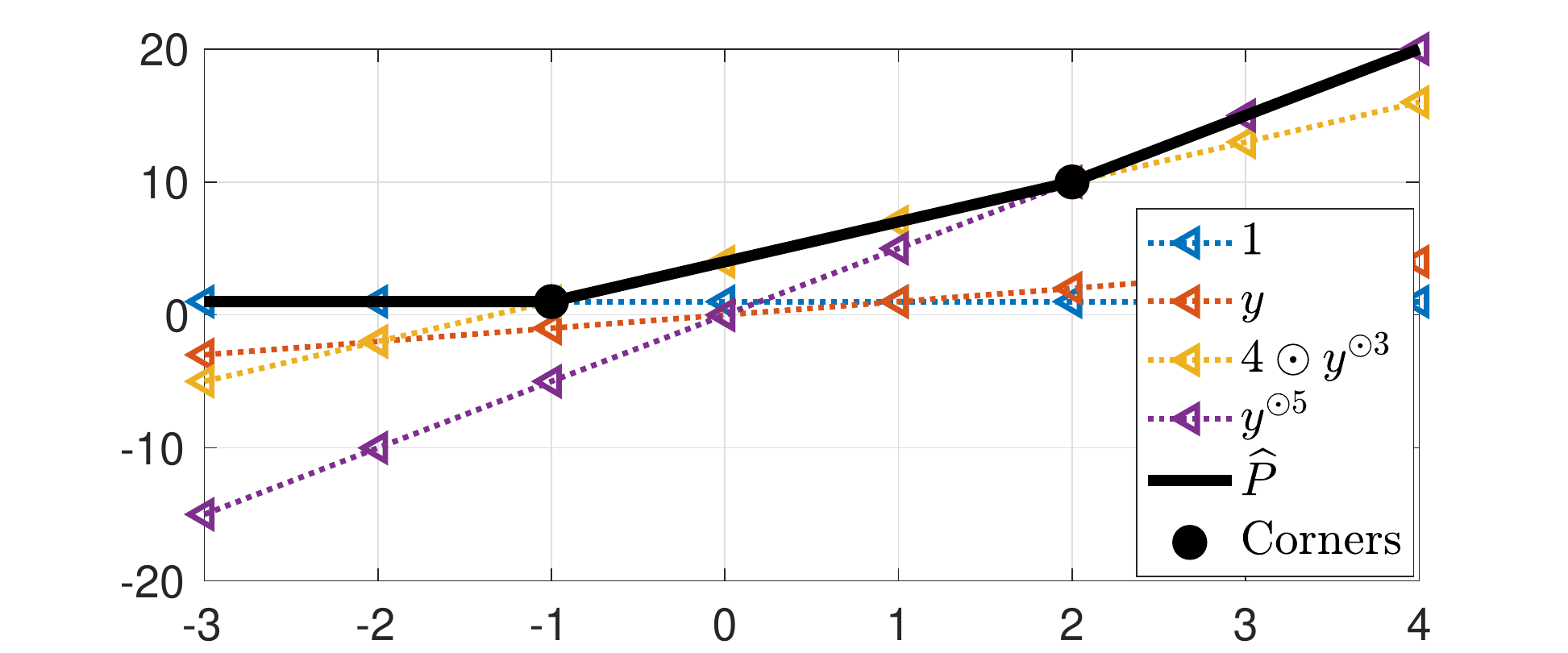}
         \caption{Polynomial function $\widehat{P}(y) = y^{\odot 5}\oplus 4 \odot y^{\odot 3} \oplus y\oplus 1$.}\label{a}
     \end{subfigure}
     \newline
          \begin{subfigure}[b]{0.48\textwidth}
         \centering
         \includegraphics[width=\textwidth]{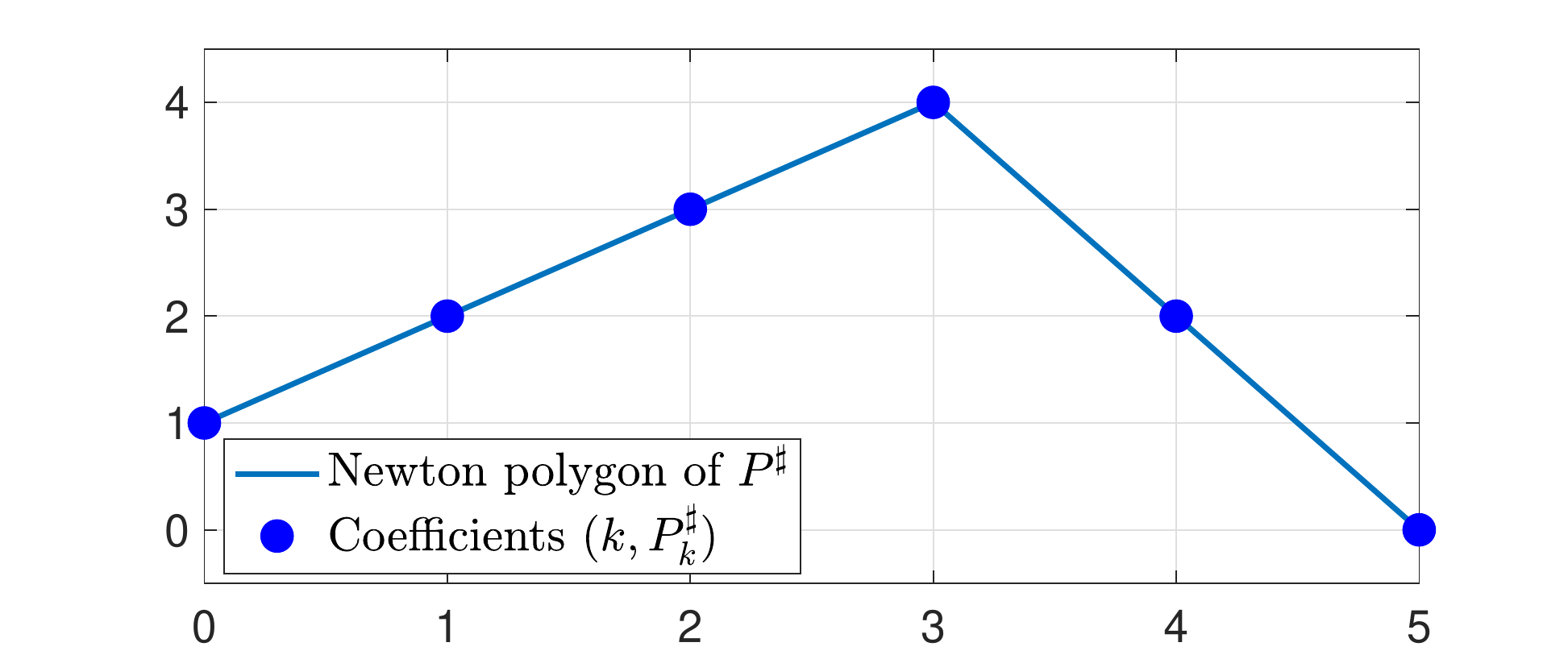}
         \caption{Configuration of points of Newton polygon of $P^{\sharp}$. }
         \label{d}
     \end{subfigure}
     \hfill
       \begin{subfigure}[b]{0.48\textwidth}
         \centering
         \includegraphics[width=\textwidth]{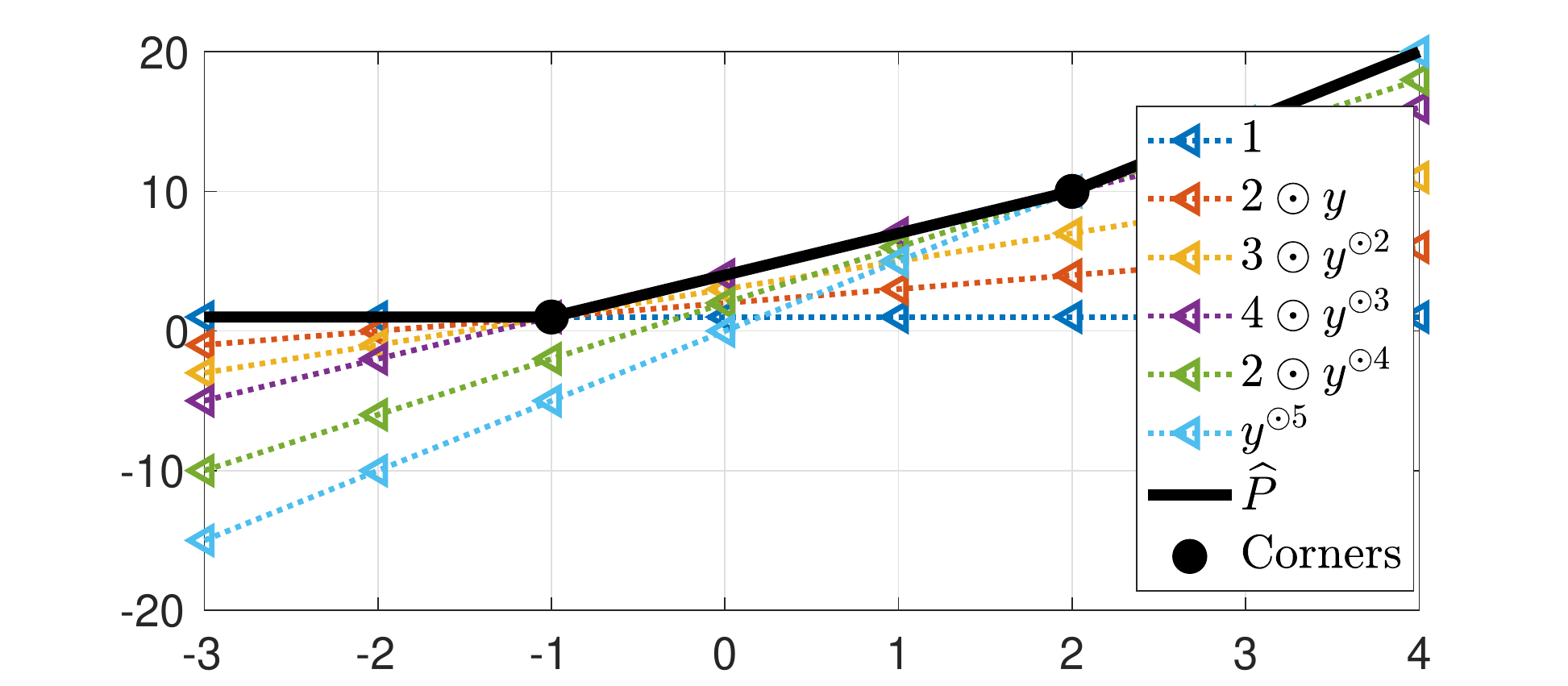}
         \caption{Illustration of $\widehat{P}=\widehat{P^{\sharp}}$.}
         \label{c}
     \end{subfigure}
     \hfill
        \caption{$P$, $P^{\sharp}$, $\widehat{P}$ and $\widehat{P}^{\sharp}$ of \Cref{ex_poly1}.}
        \label{fig_poly}
\end{figure}

\end{example}

 \subsection{Non-Archimedean valuations and tropical algebra}\label{sec-valuation}
 
 In this section, we recall some basic results about valued fields. We refer to~\cite[Chapter~2]{engler_prestel} for more information. Let $\vfield$ be a field, let $\vgroup$ be a (totally) ordered abelian group, and let $\botelt$ be added to $\vgroup$ so that it is the bottom element of $\vgroup\cup\{\botelt\}$, and thus satisfying the same properties as in \Cref{subsec-def-trop}.
A surjective function $\vall \colon \vfield \to \vgroup \cup \{\botelt \}$ is a \emph{non-Archimedean valuation} if
\begin{equation}
\begin{aligned}
 \vall(b) = \botelt &\iff b = 0 \, , \\
\forall b_{1}, b_{2} \in \vfield, \ \vall(b_{1}b_{2}) &= \vall(b_{1}) + \vall(b_{2}) \, , \\
\forall b_{1}, b_{2} \in \vfield, \ \vall(b_{1} + b_{2}) &\le \max(\vall(b_{1}),\vall(b_{2})) \, . \label[equation]{eq:valuation}
\end{aligned}
\end{equation}
The triple $(\vfield, \vgroup, \vall)$ is called a \emph{(non-Archimedean) valued field},
and $\vgroup$ is called the \emph{value group}.
The {\em valuation ring} is defined as $\vring \coloneqq \vring_\vfield \coloneqq \{b \in \vfield \colon \vall(b) \le 0 \}$, it is a subring of $\vfield$,
and $\videal \coloneqq \videal_\vfield \coloneqq \{ b \in \vfield \colon \vall(b) < 0\}$ is its maximal ideal. The quotient field $\resfield \coloneqq \resfield_\vfield \coloneqq \vring/\videal$ is called the \emph{residue field}. We denote by $\res$ the canonical projection from $\vring$ to $\resfield$, it is called the \emph{residue map}. The valuation is called \emph{trivial} if $\vgroup = \{0\}$. Otherwise, it is called \emph{nontrivial}. %

Our definition is opposite of the usual one which
deals with a map $\vall:\mathbb{F} \mapsto \vgroup \cup \{\topelt\}$,
where $\topelt$ is a ``top'' element, that satisfies 
\[\vall(b_1+b_2) \geq \min\{\vall(b_1), \vall(b_2)\} \enspace .\]
The latter definition would lead us
to work with the {\em min-plus semifield} $(\vgroup \cup\{\topelt\},\min,+)$,
instead of the max-plus semifield $\tmax(\vgroup)$. We prefer to adopt the ``max-plus''
setting here, it has the advantage that the natural order defined by~\eqref{order_max} coincides with the ordinary one.

There is a close relation between non-Archimedean valuations and tropical algebra, since~\eqref{eq:valuation} indicates
the map $\vall$ is almost a morphism from $\vfield$ to $\tmax$.

\begin{example}\label{ex-hahn}
A classical example of non-Archimedean valued field $\vfield$,
with residue field $\resfield$ and value group $\vgroup$,
is the field of {\em Hahn series}
$\hahnseries{\resfield}{\vgroup}$, whose elements
are formal sums 
\begin{align}
f= \sum_{\alpha \in \vgroup} f_\alpha t^\alpha\label{e-def-hahn}
\end{align}
where $(f_\alpha)_{\alpha\in \vgroup}$ is a family 
of elements of $\resfield$ such that the opposite
of the support $\support f:= \{\alpha\in \vgroup\mid f_\alpha \neq 0\}$
is a well ordered subset of $\vgroup$.
The {\em valuation}
of $f$ is defined as the maximal element of the support.
The sum and product
of Hahn series are defined by the usual rules, understanding
that $t^\alpha$ is a formal term and that $t^\alpha t^\beta=t^{\alpha + \beta}$. 
When $\resfield$ is ordered, $\hahnseries{\resfield}{\vgroup}$
is canonically ordered in the following way: $f>0$ if
its leading coefficient, i.e., the
coefficient $f_\alpha$ of the monomial of $f$ with exponent $\alpha$
equal to the valuation of $f$, be positive.
The field $\hahnseries{\resfield}{\vgroup}$
is real (resp.\ algebraically) closed as soon as $\resfield$ is real
(resp.\ algebraically) closed and $\vgroup$ is divisible~\cite{ribenboim}.
\end{example}
\begin{example}
When $\vgroup=\R$ and $\resfield=\R$ or $\resfield=\C$, one may
consider absolutely convergent series instead
of formal series. In this way, we get
the subfield $\hahnseries{\resfield}{\R}_{\textrm{cvg}} \subset
\hahnseries{\resfield}{\R} $ consisting of series of the form~\eqref{e-def-hahn}
that converge absolutely when $t$ is substituted by a sufficiently
small positive real number. The field
$\hahnseries{\R}{\R}_{\textrm{cvg}}$ 
is real closed~\cite{van_dries_power_series},
and so, $\hahnseries{\C}{\R}_{\textrm{cvg}}=(\hahnseries{\R}{\R}_{\textrm{cvg}})[\sqrt{-1}]$ is algebraically closed.
\end{example}
\begin{example}\label{ex-Puiseux}
The case in which $\vgroup=\Q$ is somehow simpler. Then,
one may consider the field of ordinary formal Puiseux series,
denoted by $\puiseuxseries{\resfield}$ here,
whose elements are series $f$ of the form~\eqref{e-def-hahn},
such that the support of $f$ is included in an arithmetic progression
of rational numbers. Again, $\puiseuxseries{\resfield}$
is real (resp.\ algebraically) closed as soon
as $\resfield$ is real (resp.\ algebraically) closed.
When $\resfield=\R$ or $\resfield=\C$, one can consider as well
the subfield of absolutely convergent Puiseux series 
$\resfield\{\{t\}\}_{\textrm{cvg}}$.
\end{example}

The next proposition shows that the valuations of the classical roots of a polynomial over $\vfield$ coincide with the tropical roots of the valuation of the polynomial.
For a polynomial $P = \sum_{k=0}^n P_k \Y^k \in \vfield[\Y]$, we define $\Val(P)=\bigoplus_{k=0}^n \vall(P_{k}) \Y^k \in \tmax[\Y]$.
\Cref{prop-kap} is a dual formulation of a classical result, stated in terms of Newton polygons ~\cite[Exer.\ VI.4.11]{bourbakicomm};
a proof of this proposition can be found in~\cite{akian2016non}.
A partial generalization to the multivariate case is known
as Kapranov's theorem in tropical geometry~{\cite[Th.~2.1.1]{kapranov}}.
\begin{proposition}\label{prop-kap}
Let $\vfield$ be an algebraically closed field with a non-trivial non-Archimedean valuation $\vall$ and let $P=\sum_{k=0}^nP_k Y^k\in \vfield[Y]$ with $P_n\neq 0$. Then, the 
valuations of the roots of $P$ (counted with multiplicities)
coincide with the corners or tropical roots of the tropical polynomial 
$\Val(P):=\bigoplus_{k=0}^n\vall(P_k)\Y^k\in \tmax[\vgroup]$.\hfill\qed
\end{proposition}

For our last results, see \Cref{sec-valroots}, 
we shall need a further definition, originating from model theory of valued
fields~\cite{denef_p-adic_semialgebraic,pas_cell_decomposition,pas_on_angular_component}.
A map $\xsec \colon \vgroup \to \vfield^{*}$ is called a \emph{cross-section}
if it is a multiplicative morphism such that $\vall \circ \xsec$ is the identity map. A map $\angular \colon \vfield \to \resfield$ is called an \emph{angular component} if it fulfills the following conditions:
\begin{enumerate}
\item $\angular(0) = 0$;
\item $\angular$ is a multiplicative morphism from $\vfield^{*}$ to $\resfield^{*}$;
\item the function $\morphism$ from $\vring$ to $\resfield$, mapping $b$ to $\angular(b)$ if $\vall(b) = 0$, and to $0$ otherwise, is a surjective morphism of rings whose kernel is equal to $\videal$.
\end{enumerate}
A valued field which admits a cross section admits the angular
component $\angular(b) \coloneqq \res(\xsec(-\vall(b))b)$ for $b \neq 0$.
For example, if $\vfield$ is the field of Puiseux series defined above,
then $\xsec(\alpha) = t^{\alpha}$ is a cross-section, and the map
which associates to a series $f$ its leading coefficient $f_{\vall(f)}$
is an angular
component. More generally, every real closed valued field has a cross-section, see e.g.~\cite[Lemma 3]{allamigeon2020tropical}, and thus an angular component. It also has a divisible value group $\vgroup$ (see e.g. the proof of~\cite[Lemma 3]{allamigeon2020tropical}).
We finally recall a completeness theorem from the theory of valued
field, established in~\cite{allamigeon2020tropical}, as
a consequence of a quantifier elimination method
in valued fields~\cite{denef_p-adic_semialgebraic}
developed by Denef and Pas~\cite{pas_cell_decomposition}.
\begin{theorem}[{\cite[Th.~10]{allamigeon2020tropical}}]\label{theorem:qe_vrcf}%
  The first order theory  $\thVrcf$ of valued fields with angular components
  that are real closed and have a non-trivial
  and convex valuation is complete.
\end{theorem}
We refer the reader to~\cite{allamigeon2020tropical}
for the formal definition of the theory $\thVrcf$.
For the present purposes, it suffices to say that this theory
involves a three sorted logic, allowing one to quantify
three kinds of variables: elements of the valued field, elements of the residue
field, and elements of $\vgroup\cup\{\botelt\}$. In addition
to the usual symbols from the theory of ordered fields and ordered groups with bottom, formulas in this theory
can make use the symbols $\angular$ for the angular component
and $\vall$ for the valuation. In that way, formulas
like $\alpha = \vall (b)$ or $a=\angular(b)$, where $\alpha, a,b$ are interpreted
as variables in $\vgroup \cup\{\botelt\}$, $\resfield$, or $\vfield$,
respectively, can be expressed. 
Note that, although every real closed valued field has an angular component, the condition ``with angular components'' in the above result was necessary to define what are the admissible expressions. Recall also that the value group of a real
closed field, that have a convex valuation,
is always divisible, a property which will often be used in sequel.

\section{Preliminaries on the Symmetrized tropical algebra and signed valuation}\label{sec-sym}
 We next recall the construction and basic properties of the symmetrized
 tropical semiring, as well as polynomials over this semiring.
We also discuss its relation with signed valuation, hyperfields and 
semiring systems. 

\subsection{The symmetrized tropical algebra}
The symmetrized tropical semiring $\smax$ 
initially appeared in~\cite{maxplus90b} (in the case $\vgroup=\R$),
 as a tool to solve systems of linear equations with signs. %
 We refer the reader to \cite{baccelli1992synchronization,gaubert1992theorie,cramer-guterman} for more information.

Because of the idempotency of $\oplus$, the theory of linear system of equations in tropical algebra is not satisfactory. The idea used in \cite{maxplus90b} is to extend $\tmax$ to a larger set $\smax$ for which $\tmax$ can be considered as the positive part. This extension is similar to the construction of $\mathbb{Z}$ as an extension of $\N$ in usual algebra but with more complications. In this new algebraic structure, every non-trivial scalar linear equation such as $a \oplus y = \zeror$, has at least one solution. The purpose of this subsection is recall this approach and the general properties. {\em All the definitions and properties below are generalized to the case of any ordered group $\vgroup$, and we shall write again $\tmax$ instead of $\tmax(\vgroup)$.}

Let us consider the set $\tmax^2:=\tmax\times \tmax$ endowed with operations $\oplus$ and $\odot$:
\[(a_1,a_2) \oplus (b_1,b_2) =(a_1\oplus b_1, a_2 \oplus b_2),\]
\[(a_1,a_2) \odot (b_1,b_2) = (a_1\odot b_1 \oplus a_2 \odot b_2, a_1 \odot b_2 \oplus a_2 \odot b_1),\]
with $(\zeror, \zeror)$ as the zero element and $(\unitr, \zeror)$ as the unit element. 
Define the following three operators on  $a= (a_1, a_2)\in \tmax^2$:
\begin{center}
\begin{tabular}{ll}
$\ominus a = (a_2, a_1)$ & minus operator $\tmax^2\to \tmax^2$;\\ 
$|a| = a_1 \oplus a_2$ & absolute value $\tmax^2\to \tmax$;\\  
$a^{\circ} = a\ominus a = (|a|, |a|)$&  balance operator $\tmax^2\to \tmax^2$.
\end{tabular}
\end{center}
\begin{property}\label{property-minus}
The operators $\ominus$ and $^{\circ}$ satisfy
\begin{enumerate}
\item $\ominus (\ominus a )= a$;
\item $\ominus(a\oplus b) = (\ominus a) \oplus (\ominus b)$;
\item $\ominus(a \odot b) = (\ominus a) \odot b$.
\item $a^{\circ} = (\ominus a)^{\circ}$;
\item $a^{\circ \circ} = a^{\circ}$;
\item $a\odot b^{\circ}= (a\odot b)^{\circ}$;
\end{enumerate}
\end{property}
These properties allow us to write $a \oplus (\ominus b) = a \ominus b$ as usual. 
\begin{definition}[Balance relation]\label{balance_def}
Let $a = (a_1, a_2)$ and $b = (b_1, b_2)$. We say that $a$ balances $b$, $a \balance b$, if $a_1 \oplus b_2 = a_2 \oplus b_1 $.
\end{definition}

  The balance relation is not transitive.
  For example we have $(1,2) \balance (3,3)$, $(3,3) \balance (1,1)$, but $(1,2) \notbalance  (1,1)$. The following relation $\mathcal{R}$ on $\tmax^2$
  is an equivalence relation, and it refines the balance relation:
\[(a_1,a_2) \mathcal{R} (b_1,b_2) \Leftrightarrow
\begin{cases}
a_1 \oplus b_2 = a_2 \oplus b_1& \;\text{if}\; a_1 \neq a_2, \;b_1 \neq b_2,\\
(a_1,a_2)=(b_1,b_2)& \text{otherwise.}
\end{cases}
\]  
One can check that $\mathcal{R}$ is compatible with $\oplus$ and $\odot$ of $\tmax^2$, $\balance$, $\ominus$, $|.|$ and $^{\circ}$ operators,
which then can be defined on the quotient $\tmax^2 / \mathcal{R}$.
\begin{definition}\label{sym_def}
The \textit{symmetrized tropical semiring} is the quotient semiring $(\tmax^2 / \mathcal{R},\oplus,\odot)$ and is denoted by $\smax$. 
We denote by $\zero:=\overline{(\botelt, \botelt)}$ the zero element 
and by $\unit:=\overline{(0, \botelt )}$ the unit element.
\end{definition}
We distinguish three kinds of equivalence classes:
\begin{center}
\begin{tabular}{ll}
$\overline{(c, \botelt)} = \{(c,a_2)\mid a_2<c\}, \; c\in \vgroup$ & positive elements  \\ 
$\overline{(\botelt,c)}=\{(a_1, c)\mid a_1<c\}, \; c\in \vgroup$ & negative elements  \\  
$\overline{(c,c)}=\{(c,c)\}, \; c\in \vgroup\cup\{\botelt\}$ & balance elements.
\end{tabular}
\end{center}
Recall that the zero element of $\smax$ is $\zero=\overline{(\botelt,\botelt)}$.
Then, we denote by $\smax^{\oplus}$,
$\smax^{\ominus}$ and  $\smax^{\circ}$ 
the set of positive or zero elements,
the set of negative or zero elements, and the set of balance elements,
respectively.
Therefore, we have:
\[\smax^{\oplus}\cup \smax^{\ominus}\cup \smax^{\circ}=\smax, \]
where the pairwise intersection of any two of these three sets
is reduced to $\{\zero\}$.
The subsemiring $\smax^{\oplus} $ of $\smax$ can be
identified to $\tmax$, by the morphism $c\mapsto \overline{(c, \botelt)}$.

\begin{property}\label{prop-modulus}
Using the above identification, the absolute value map $a\in \smax \mapsto |a|\in \smax^\oplus$ is a morphism of semirings.
\end{property}
The identification of $\tmax$ with $\smax^{\oplus} $, or the embedding of
$\tmax$ in $\smax$ allows one to consider for instance the normalization
$|a|^{\odot -1} \odot a$.

\begin{definition}[Signed tropical numbers]
The elements of $\smax^\vee:=\smax^{\oplus} \cup \smax^{\ominus}$ are called \new{signed tropical numbers}, or simply \new{signed}. They are either positive, negative or zero.
\end{definition}

We shall sometimes write $\smax(\vgroup)$ instead
of $\smax$ to emphasize the dependence on the choice
of $\vgroup$. In particular, by taking the trivial
$\vgroup=\{0\}$, one obtains $\bmaxs:= \smax(\{0\})$,
which is isomorphic to the subsemiring of
$\smax$ composed of $\{\zero,\unit,\ominus \unit,\unit^\circ \}$.
It is called the {\em symmetrized Boolean algebra}.
We denote $\bmaxs^\vee:=\bmaxs\cap \smax^\vee
=\{\zero,\unit,\ominus \unit\}$.

\begin{remark}
For each element $a \in \smax^{\vee}$, the element $\ominus a$ is the unique signed element such that $b = a \ominus a \in \smax^{\circ}$. In general $b \neq \zero$. So, the whole set $\smax^{\circ}$ plays the role of the usual zero element. 
Note that $\smax^{\circ}$ is an ideal, thus a subsemiring of $\smax$,
and that the map  $a\in \smax \mapsto a^\circ\in \smax^\circ$  is
a morphism of semirings.
\end{remark}

\begin{property}\label{prop-nabla}
 The relation $\balance$ satisfies the following properties:
\begin{enumerate}
\item $a\balance a$;
\item $a \balance b \Leftrightarrow b \balance a$;
\item $a \balance b \Leftrightarrow a \ominus b\balance \zero$;
\item $\{a\balance  b , c\balance d\} \Rightarrow a \oplus c \balance b \oplus d$;
\item $a \balance  b \Rightarrow a\odot c\balance  b\odot c$;
\item $x \balance a,\;c\odot x \balance  b$ and $x \in \smax^{\vee} \Rightarrow c \odot a\balance  b$ (weak substitution);
\item $a \balance  x,\;x\balance b$ and $x \in \smax^{\vee} \Rightarrow a \balance b $ (weak transitivity);
\item $a \balance  b$ and $a,b \in \smax^{\vee} \Rightarrow a=b$ (reduction of balances);
\item \label{prop-signed}
$a_1\oplus\cdots \oplus a_n\balance  \zero$ and $a_i \in \smax^{\vee},\; i=1,\ldots, n$ $ \Rightarrow a_i=\ominus a_j$ for some $i\neq j$,
such that $|a_i|\geq |a_k|$ for all $k\in [n]$.
\end{enumerate}
\end{property}
\begin{definition}\label{partial_order}
We define the following relations, for $a,b \in \smax$:
\begin{gather*}
a \preceq b \iff b = a \oplus c \;\text{for some}\;c \in \smax
\iff b=a\oplus b,\\
a \preceq^{\circ} b \iff b = a \oplus c \;\text{for some}\;c \in \smax^{\circ} 
\enspace . \end{gather*}
\end{definition}

\begin{property}
The relation $\preceq$ is the natural partial order  \eqref{order_max} 
of $\smax$.
The relation $\preceq^\circ$  is a partial order on $\smax$.
\end{property}

\begin{example}
We have 
\[\zero \preceq \ominus 2 \preceq \ominus 3,\;\zero \preceq 2 \preceq 3,\; 2 \preceq \ominus 3.\]
However, $3$ and $\ominus 3$ are not comparable with $\preceq$.
\end{example}

\begin{property}\label{pro_preceq}
 We have the following properties for all $a,b \in \smax$:
\begin{enumerate}
\item\label{pro_preceq1}  $a \preceq b$ implies $|a| \leq |b|$;
\item\label{pro_preceq2} If $|a| \leq |b|$ and $b \in \smax^{\circ}$, then $a \preceq^\circ b$ and so $a \preceq b$;
\item\label{pro_preceq3} If  $a\in\smax^\vee$, then $a \preceq^\circ b$ is equivalent to $a\balance b$ and it implies $|a| \leq |b|$.
\item \label{pro_preceq4}If  $b\in\smax^\vee$, then $a \preceq^\circ b$ is equivalent to $a= b$.
\end{enumerate} 
\end{property}
\subsection{Polynomials over $\smax$}\label{sec-polysym}
Most of the definitions and results of this section are taken from \cite[Sec.\ 3.6]{baccelli1992synchronization}, up to the replacement of $\R$ by $\vgroup$. 
For the results, we shall need to %
assume that $\vgroup$ is a non-trivial divisible ordered group.
However, the definitions can be given for a general  ordered group,
and the questions of factorization of polynomial functions can be considered in any divisible group.
In particular, when $\vgroup$ is trivial, the factorization of polynomial functions over $\smax$ is related to the factorization of polynomials over the hyperfield of signs which was adressed already in \cite{lorscheidfac}.

The definition of \textit{formal polynomials}, as well as their degree $\deg$ and lower degree $\uval$, and of \textit{polynomial functions}, can be done for any semiring, so in particular for $\smax=\smax(\vgroup)$ (see \Cref{sec-form-rmax}), and lead to the semirings $\smax[\Y]$ and $\PF(\smax)$ of formal polynomials and polynomial functions respectively.
We can define also on $\smax[\Y]$ the partial orders or pre-orders $\preceq$ and $\preceq^\circ$, 
by applying the corresponding relations coefficient-wise.
For convenience of the reader we provide \Cref{notations_poly} of notations related to polynomials over $\smax$.

Let $\smax^{\vee}[\Y]\subset\smax[\Y]$ be the subset of formal polynomials $P$ with coefficients in $\smax^{\vee}$ and let $\PF (\smax^{\vee})\subset\PF (\smax)$ be the set of associated polynomial functions $\widehat{P}:\smax\to\smax$.
Suppose that $P\in \smax^{\vee}[\Y]$. Then, 
the \textit{signed roots} (or \textit{roots} for short) of $P$ or $\widehat{P}$  are the signed elements $y \in \smax^{\vee}$ for which $\widehat{P}(y) \balance  \zero$. 
We say that the polynomial function $\widehat{P}\in \PF (\smax^{\vee})$ can be factored into linear factors if there exist $r_i \in \smax^{\vee}$, for $i=1, \ldots, n$, such that 
\begin{equation}\label{s-factorization}
\widehat{P}(y)=  P_n\odot (y \ominus r_1) \odot \cdots \odot (y \ominus r_n)\enspace .
\end{equation}
Note that for $y\in \smax^{\vee}$, we have 
\begin{equation}\label{fact-roots}
\tropprod_i (y \ominus r_i) \balance \zero \iff (y \ominus r_i)\balance \zero,\quad \text{for some}\; i \iff y=r_i\quad \text{for some}\; i\enspace,\end{equation}
and therefore, the computation of linear factors of a polynomial function gives its roots. However, we will see later that this factorization does not always exist and is not unique 
when the roots $r_i$ are not simple, see \Cref{rem-nonunique}.
This means that, unlike $\tmax$, the semiring $\smax$ is not algebraically closed. However, some polynomial functions called \textit{closed polynomial functions} can be factored into linear factors. %

 In order to study closed polynomial functions in $\smax$, let  
 \[\mathcal{F}: \smax[\Y] \mapsto 
\PF (\smax), \;P \rightarrow \widehat{P}.\]
 As for $\tmax$, the mapping $\mathcal{F}$ is not injective. We will use the concept of closed polynomial function defined in \cite{baccelli1992synchronization}.
\begin{definition}[Closed polynomial function, \protect{\cite[Def.\ 3.87]{baccelli1992synchronization}}  for $\vgroup=\R$]\label{closed_def}
We say that the polynomial function $\widehat{P} \in \PF (\smax^{\vee}) $ is closed if $\mathcal{F}^{-1}(\widehat{P})$ admits a maximum element for the partial order $\preceq$, denoted $P^{\sharp} \in \smax[\Y]$. This element is called the maximum representative of $\widehat{P}$.
\end{definition}
If $\vgroup$ is non-trivial and divisible, then all elements of $\mathcal{F}^{-1}(\widehat{P})$ have necessarily same degree and lower degree as $P$ as in the case of $\tmax$ and are upper bounded (by the same arguments as in the proof of \Cref{max-lemma}), and so the supremum $P^{\sharp}$ of $\mathcal{F}^{-1}(\widehat{P})$ always exist, since any upper bounded subset of $\smax$ admits a supremum. 
However, the supremum of $\mathcal{F}^{-1}(\widehat{P})$ may not be an element of $\mathcal{F}^{-1}(\widehat{P})$ and so a maximum. This is the case for instance for $P=\Y^2\oplus\unit$ which is such that $P^\sharp=\Y^2\oplus \unit^\circ \Y\oplus\unit$. Also, $P^{\sharp}$ is in general not an element of $\smax^\vee[\Y]$.
 See \Cref{closed_poly} for an illustration of these properties.

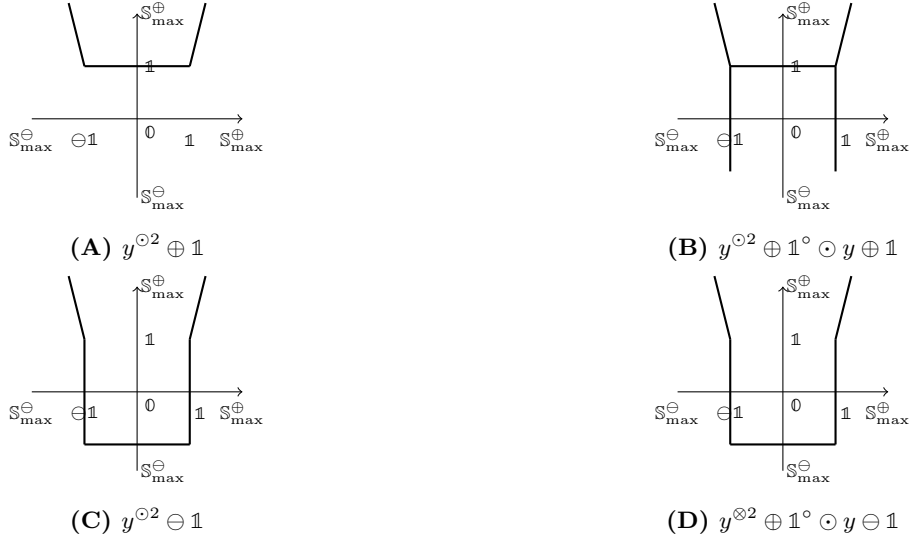
\begin{figure}
 \centering
  \begin{subfigure}[b]{0.48\textwidth}
  \centering
   \begin{tikzpicture}[scale=0.7]
\draw[->] (-2,0) -- (2,0);
\draw[->] (0,-1.5) -- (0,2);
\draw[thick](-1,1) -- (1,1);
\draw[thick] (1,1) -- (1.3,2.2);
\draw[thick] (-1,1) -- (-1.3,2.2);
\fill (0.25,-0.25) node {\tiny$\zero$};
\fill (-2,-0.4) node {\tiny$\smax^{\ominus}$};
\fill (2,-0.4) node {\tiny$\smax^{\oplus}$};
\fill (0.5,2) node {\tiny$\smax^{\oplus}$};
\fill (0.5,-1.5) node {\tiny$\smax^{\ominus}$};
\fill (-1,-0.4) node {\tiny$\ominus \unit$};
\fill (1,-0.4) node {\tiny$\unit$};
\fill (0.25,1) node {\tiny$ \unit$};
\end{tikzpicture}\caption{$y^{\odot 2} \oplus \unit$}%
     \end{subfigure}
     \hfill
         \begin{subfigure}[b]{0.48\textwidth}
         \centering
\begin{tikzpicture}[scale=0.7]
\draw[->] (-2,0) -- (2,0);
\draw[->] (0,-1.5) -- (0,2);
\draw[thick](-1,1) -- (1,1);
\draw[thick] (1,1) -- (1.3,2.2);
\draw[thick] (-1,1) -- (-1.3,2.2);
\draw[thick] (1,1) -- (1,-1);
\draw[thick] (-1,1) -- (-1,-1);
\fill (0.25,-0.25) node {\tiny$\zero$};
\fill (-2,-0.4) node {\tiny$\smax^{\ominus}$};
\fill (2,-0.4) node {\tiny$\smax^{\oplus}$};
\fill (0.5,2) node {\tiny$\smax^{\oplus}$};
\fill (0.5,-1.5) node {\tiny$\smax^{\ominus}$};
\fill (-1,-0.4) node {\tiny$\ominus \unit$};
\fill (1.2,-0.4) node {\tiny$\unit$};
\fill (0.25,1) node {\tiny$ \unit$};
\end{tikzpicture}\caption{$y^{\odot 2} \oplus \unit^{\circ} \odot y \oplus \unit$}

 \end{subfigure}
   \\
    \begin{subfigure}[b]{0.48\textwidth}
  \centering

\begin{tikzpicture}[scale=0.7]
\draw[->] (-2,0) -- (2,0);
\draw[->] (0,-1.5) -- (0,2);
\draw[thick] (1,1) -- (1.3,2.2);
\draw[thick] (-1,1) -- (-1.3,2.2);
\draw[thick] (1,1) -- (1,-1);
\draw[thick] (-1,1) -- (-1,-1);
\draw[thick] (-1,-1) -- (1,-1);
\fill (0.25,-0.25) node {\tiny$\zero$};
\fill (-2,-0.4) node {\tiny$\smax^{\ominus}$};
\fill (2,-0.4) node {\tiny$\smax^{\oplus}$};
\fill (0.5,2) node {\tiny$\smax^{\oplus}$};
\fill (0.5,-1.5) node {\tiny$\smax^{\ominus}$};
\fill (-1,-0.4) node {\tiny$\ominus \unit$};
\fill (1.2,-0.4) node {\tiny$\unit$};
\fill (0.25,1) node {\tiny$ \unit$};
\end{tikzpicture}\caption{$y^{\odot  2} \ominus \unit$}

     \end{subfigure}
     \hfill
      \begin{subfigure}[b]{0.48\textwidth}
  \centering
 
\begin{tikzpicture}[scale=0.7]
\draw[->] (-2,0) -- (2,0);
\draw[->] (0,-1.5) -- (0,2);
\draw[thick] (1,1) -- (1.3,2.2);
\draw[thick] (-1,1) -- (-1.3,2.2);
\draw[thick] (1,1) -- (1,-1);
\draw[thick] (-1,1) -- (-1,-1);
\draw[thick] (-1,-1) -- (1,-1);
\fill (0.25,-0.25) node {\tiny$\zero$};
\fill (-2,-0.4) node {\tiny$\smax^{\ominus}$};
\fill (2,-0.4) node {\tiny$\smax^{\oplus}$};
\fill (0.5,2) node {\tiny$\smax^{\oplus}$};
\fill (0.5,-1.5) node {\tiny$\smax^{\ominus}$};
\fill (-1,-0.4) node {\tiny$\ominus \unit$};
\fill (1.2,-0.4) node {\tiny$\unit$};
\fill (0.25,1) node {\tiny$ \unit$};
\end{tikzpicture}\caption{$y^{\otimes 2} \oplus \unit^{\circ} \odot y \ominus \unit$}

     \end{subfigure}\caption{The polynomial function $y^{\odot 2} \oplus \unit$ is not closed since it does not have the same polynomial function as $y^{\odot 2} \oplus \unit^{\circ}\odot y \oplus  \unit$. However the polynomial function $y^{\odot 2} \ominus \unit$ is closed since it has the same polynomial function as  $y^{\odot 2} \oplus \unit^{\circ}\odot y \ominus  \unit$. }
         \label{closed_poly}
\end{figure}

In $\tmax$, if $\vgroup$ is divisible, the formal polynomial $P^\sharp$ is factored and has the same polynomial function. In $\smax$, this property holds when the polynomial is closed, as we will see in \Cref{lemma-psharp}, but may not hold otherwise as for the previous example $P^\sharp=\Y^2\oplus \unit^\circ \Y\oplus\unit$.
These are the underlying properties of the following characterization.

\begin{theorem}[See~\protect{\cite[Th.\ 3.89]{baccelli1992synchronization}} for $\vgroup=\R$]
\label{theo-factored} Assume that $\vgroup$ is a non-trivial divisible group.
A polynomial function $\widehat{P} \in \PF (\smax^{\vee})$ can be factored into linear factors  if and only if it is closed.
\end{theorem}

\begin{remark} If $\vgroup$ is trivial, then the above arguments fail.
Indeed, first the elements of $\mathcal{F}^{-1}(\widehat{P})$ do not have necessarily same degree and lower degree as $P$. So one would need to replace in 
\Cref{closed_def}, $\mathcal{F}^{-1}(\widehat{P})$ by the subset of polynomials with same degree and lower degree as $P$.
However, if we do so, the closeness property is not sufficient for factorization.
Take the example of $P=\Y^2\oplus\unit$. When $\vgroup$ is trivial, the polynomial function $\widehat{P}$ is equal to $\unit$ on $\smax^\vee$, so $P$ is the only element of $\mathcal{F}^{-1}(\widehat{P})$ with same degree and lower degree than $P$, and is thus closed with  $P^\sharp=P$ in the above sense, but $\widehat{P}$ is not factored.
\end{remark}

\begin{example}\label{ex_sym_plot}
We consider again the formal polynomal $P = \Y^5\oplus 4 \Y^3\oplus \Y \oplus 1$ of \Cref{ex_poly1} over $\smax(\R)$. To plot $P$ we must consider all of three components $\smax^{\ominus}$, $\smax^{\oplus}$ and $\smax^{\circ}$ and also the value of the function itself, is an element of $\smax$ which also belongs to one of these components.  Therefore, this case is more complicated than the case over $\tmax$. To simplify this, it is convenient to only consider $P$ over $\smax^{\vee}$. Then the points of discontinuity show the points in which the polynomial functions take balanced values. These discontinuities always are symmetric to the corners. To plot $P$, consider the polynomial function $\widehat{P}$ as following: 
 \[\widehat{P}(y)=\begin{cases}
 y^{\odot 5}&y \succeq 2\;, \\
 4\odot y^{\odot 3} & -1 \preceq y \preceq 2\;,\\
 1&  \ominus -1 \prec y \preceq -1\;,\\
 \ominus 4\odot y^{\odot 3} &  \ominus 2 \preceq y \preceq \ominus -1\;,\\
 \ominus y^{\odot 5}& y \preceq\ominus 2\;.
 \end{cases}\]
 It is not hard to check that $\ominus -1$ is the only point of discontinuity and the only root of $P$. \Cref{sym_poly} illustrats the plot of $P$ and $\widehat{P}$.

  \begin{figure}[htbp]\small
\begin{tikzpicture}[scale=0.5]
\draw[->] (-5.5,0) -- (5.5,0);
\draw[->] (0,-5.5) -- (0,5.5);

\draw[thick] (-3,-3) -- (-3,3);
\draw[thick](-3,3) -- (3,3);
\draw[thick] (3,3) -- (4,3.5);
\draw[thick] (-3,-3) -- (-4,-3.5);
\draw[thick] (4,3.5) -- (5,5);
\draw[thick] (-4,-3.5) -- (-5,-5);

\draw[dashed] (-4,-3.5) -- (-4,0);
\draw[dashed] (3,3) -- (3,0);
\draw[dashed] (4,3.3) -- (4,0);

\fill (4,3.5) circle (3pt);
\fill (-4,-3.5) circle (3pt);
\fill (-3,3) circle (3pt);
\fill (3,3) circle (3pt);
\fill (-3,-3) circle (3pt);

\fill (0.25,-0.45) node {$\zero$};
\fill (-5.5,-0.6) node {$\smax^{\ominus}$};
\fill (5.8,-0.6) node {$\smax^{\oplus}$};
\fill (1,5.5) node {$\smax^{\oplus}$};
\fill (1,-5.6) node {$\smax^{\ominus}$};
\fill (-4.15,0.45) node {$\ominus\, 2$};
\fill (-2,0.45) node {$\ominus\! -\! 1$};
\fill (4,-0.45) node {$ 2$};
\fill (3,-0.45) node {$-1$};
\end{tikzpicture}

 \caption{Plot of $P$ in \Cref{ex_sym_plot}. The solid black line illustrates $\widehat{P}$. The only point of discontinuity of $\widehat{P}$ is $\ominus -1$ which is the only root of $P$.}
         \label{sym_poly}

     \end{figure}
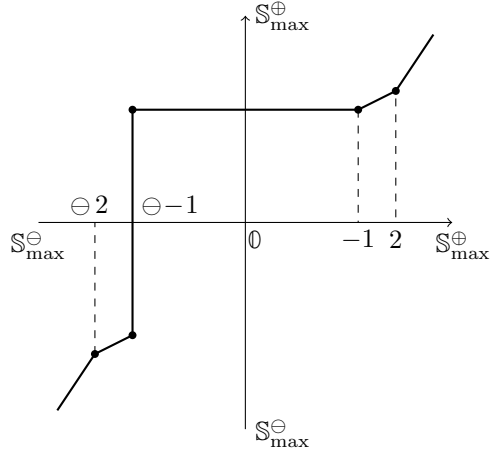
    
\end{example}
The following results were also shown in \cite{baccelli1992synchronization}.
The first one relates roots of polynomials 
over $\smax$ with corners over $\tmax$.  It follows from the definitions of these notions and is true for any ordered group $\vgroup$. 
The second one gives a sufficient condition for a polynomial to be closed.
Let us apply the absolute value map $|\cdot|$ to formal polynomials coefficient-wise.
 \begin{lemma}[See~\protect{\cite[Lemma 3.86]{baccelli1992synchronization}}  for $\vgroup=\R$]\label{abs_roots} %
Let ${P} \in \smax^\vee[\Y]$. If $r$ is a root of $P$, then $|r|$ is a corner of $|P|$. 
\end{lemma}
\begin{lemma}[See~\protect{\cite[Cor.\ 3.90]{baccelli1992synchronization}} for $\vgroup=\R$]\label{suf_closed} Assume that $\vgroup$ is a non-trivial divisible group.
Let ${P} \in \smax^\vee[\Y]$.
A sufficient condition for $\widehat{P}$  to be closed is that $|{P}|$ has distinct non $\zeror$ corners.
\end{lemma}

\begin{table}[ht]
\caption{Frequently used notation for polynomials over $\smax$}\label{notations_poly}
\begin{center}
\footnotesize
  \begin{tabular}{@{}c@{}|@{}c@{}||@{}c@{}|@{}c@{}}
& Notion& & Notion\\
\hline
$P, Q$&Formal polynomial&$P \oplus Q$&Coefficient-wise sum, \Cref{sec-form-rmax} \\
$\smax[\Y]$&Semiring of formal polynomials&$PQ$&Cauchy product, \Cref{sec-form-rmax}\\
$\PF(\smax)$&Semiring of polynomial functions&$P \preceq Q \; (\preceq^\circ%
)$&Coefficient-wise orders, \Cref{partial_order}\\
$P_k$&Coefficients of $P$, \Cref{sec-form-rmax}&$P \balance Q$&Coefficient-wise balance relation\\
$\deg(P)$& Degree of $P$, \eqref{deg}&$\Val(P)$&Coefficient-wise valuation, \Cref{sec-valuation}\\
$\uval(P)$&Lower degree of $P$, \eqref{deg}&$\sval(P)$&Coefficient-wise signed valuation, \Cref{sec-valroots}\\
$\mathrm{supp}(P)$&Support of $P$, \eqref{supp}&$\mathrm{mult}_a(P)$&Multiplicity of root $a$ of $P$, \eqref{mult}\\
$\widehat{P}$&Polynomial function of $P$, \eqref{widehat_p}&$\Pn^r$&Normalized polynomial, \Cref{tho-mult-smax}\\
$P^{\sharp}$& Max representative of $\widehat{P}$, \Cref{closed_def}&$\sat(r,P)$&Saturation set, \Cref{tho-mult-smax}\\
$|P|$&Coefficient-wise absolute value map, \Cref{sec-polysym}&$P^{\sat,r}(\Y)$&Saturation polynomial, \Cref{tho-mult-smax}\\
  \end{tabular}
\end{center}
\end{table}

\subsection{The signed valuation over an ordered non-Archimedean valued field}
\label{sec-signval}

We now consider a valued field $(\rfield,\vgroup, \vall)$
{\em equipped with a total order} $\leq$ compatible with the operations
of the field so that $\rfield$ is an ordered field.
We require also the non-Archimedean valuation $\vall$ to be \emph{convex}, meaning
that the following property is satisfied:
\[
\elf_{1} \in \vring_\rfield\, , \; \elf_{2} \in \rfield\, ,\;\text{and}\; 
 0 \le \elf_{2} \le \elf_{1} \implies \elf_{2} \in \vring_\rfield \, .
\]
It is easy to see, using the properties of the order and the valuation that 
the convexity of $\vall$ is equivalent to the following
monotonicity property of $\vall$:
\begin{equation} \elf_{1} , \; \elf_{2} \in \rfield\, ,\;\text{and}\; 
 0 \le \elf_{2} \le \elf_{1} \implies \vall(\elf_{2})\leq  \vall(\elf_{1})\, .
\end{equation}
(Note that we are using the same notation for the order of $\rfield$ and the
one of $\vgroup$).
We shall call such a field $\rfield$ an \textit{ordered valued field}.

If $\rfield$ is a real closed field, it has a unique total order. Then, if
the non-Archimedean valuation $\vall$ is convex, the residue field
$\resfield_\rfield$ is real closed (see \cite[Theorem~4.3.7]{engler_prestel}).
An example of real closed field with a convex valuation
is provided by the field of real Hahn series $\hahnseries{\R}{\vgroup}$.

\begin{definition}[See also \protect{\cite{allamigeon2020tropical}}]
\label{def-sign}
For any ordered field $\rfield$, we define the {\em sign map}
$\sign:\rfield \to \bmaxs^\vee=\{\zero, \ominus \unit, \unit\}$ as follows, for all $\elf\in \rfield$,
\[ \sign(\elf):=\begin{cases}
\unit&\elf>0,\\
\ominus \unit& \elf<0,\\
\zero& \elf=0. 
\end{cases}\]
If $\rfield$ is also a valued field, with a convex valuation, and ordered value group $\vgroup$, the \textit{signed valuation} on $\rfield$ is the  map $\sval:\rfield \to \smax=\smax(\vgroup)$ which associates the element $\sign(\elf) \odot \vall(\elf)\in \smax^\vee$, with an element $\elf$ of $\rfield$,
where here $\vall(\elf)$ is seen as an element of $\smax^{\oplus}$.
\end{definition}

The map $\sval$ is a ``valuation'' in a generalized sense, in that it
can be identified to a {\em morphism of hyperfields} or a
{\em morphism of systems}, as explained below.

First, recall that a hyperfield $\hyper$ is a multiplicative group, 
endowed with a multivalued addition map,
$(a,b)\in \hyper\times \hyper \to a\boxplus b \in 2^\hyper$, 
with zero element $0$,
satisfying single distributivity, and a negation map
$\hyper\to\hyper$, $a\mapsto - a$, satisfying unique negation
($0\in a\boxplus b$ iff $a=- b$) and reversibility
($a\in b\boxplus c$ iff $c\in a\boxplus (-b)$).
The typical example of an hyperfield is obtained by taking the quotient 
$\vfield/\subgroup$ of
a field $\vfield$  by a subgroup $\subgroup$ of  $\vfield^*:=\vfield\setminus \{0\}$ (see for instance \cite{krasner,connesconsani,viro2010hyperfields} for an 
introduction to hyperfields and for the construction $\vfield/\subgroup$).
In this case, the canonical projection $\vfield\to \vfield/\subgroup$ is
a morphism of hyperfields (see for instance \cite{viro2010hyperfields,baker2018descartes}).

Algebraic constructions such as polynomials and matrices over hyperfields do not lead to hyperfields, so several structures extending this notion have been introduced in the litterature. For instance, in~\cite{BL}, Baker and Lorscheid introduced the notions of blueprints, ordered blueprints and idylls, which generalize the notion of hyperfields.
Also, in \cite{Rowen,Rowen2}, Rowen introduced the notion of systems and in particular of semiring systems, which appeared to be 
more general than hyperfields and idylls as shown in \cite{Rowen2,AGRowen}.
We shall use later the notion of semiring systems,
that we next define.
\begin{definition}[\protect{\cite{Rowen,Rowen2,AGRowen}}]\label{defi-system}
A (commutative) {\em semiring system} is composed of a commutative semiring $(\semiring,\oplus,\zero,\odot,\unit)$, \sloppy
a submonoid $\tangible$ of $(\semiring^*,\odot,\unit)$, additively generating $\semiring$, and being cancelative ($b\odot a=b\odot a'$ with $b\in \tangible$
and $a,a'\in\semiring$ implies $a=a'$),
and called the set of {\em tangible elements}, 
a negation, that is a map $\semiring\to \semiring$, $a\mapsto \ominus a$, 
satisfying~\Cref{property-minus}, (1)--(3),
and a {\em surpassing relation} $\surpass$, defined as follows.
Denote $\semiringvee:=\tangible\cup\{\zero\}$,
a {\em surpassing relation} $\surpass$  on $\semiring$ 
is a pre-order on $\semiring$ compatible
with the operations and such that, for all $a\in \semiring$, $\zero\surpass a\ominus a$ and,
for all $b,b'\in \semiringvee$, $b\surpass b'$ implies $b=b'$.
The set $\semiringvee$ is also required to satisfy
\new{unique negation}, that is 
$b,b'\in \semiringvee$ and $\zero\surpass b\ominus b'$ imply $b=b'$.
\end{definition}
A hyperfield gives rise to a system, called {\em hyperfield system}, by 
considering the set $\semiring$ of finite sums of elements of $\hyper$,
the set of tangible elements $\tangible=\hyper\setminus\{0\}$, and the inclusion relation as the
surpassing relation \cite{Rowen2,AGRowen}.
When $\hyper$ is a field, then the surpassing relation is
the equality relation.
Consider the ordered valued field  of real Hahn series 
$\rfield:=\hahnseries{\R}{\vgroup}$,
and  the hyperfield $\hyper(\vgroup)$ defined as the quotient
$\rfield/\subgroup$,
where $\subgroup$ is the group of elements $\elf\in \rfield$ which satisfy $\elf>0$ and $\vall(\elf)=0$. 
When, $\vgroup=\R$, $\hyper(\vgroup)$ is called sometimes the {\em signed tropical hyperfield}
\cite{gunn}, or the {\em tropical real hyperfield} \cite{viro2010hyperfields}, or the {\em real tropical hyperfield} \cite{jell_scheiderer_yu}. 
When $\vgroup=\{0\}$, the  hyperfield  $\hyper(\vgroup)$ is nothing but the
{\em hyperfield  of signs} (considered in~\cite{connesconsani,viro2010hyperfields,baker2018descartes}).
It can also be obtained by taking the quotient $\rfield/\subgroup$ of any ordered field $\rfield$ by the group $\subgroup$ of its positive elements.
Moreover, as explained in \cite{AGRowen},
idylls can be realized as a subclass of semiring systems,
in which the set $\tangible$ is a group $G$, the semiring $\semiring$ is
the monoid semiring $\N[G]$ (including $\zero$) and the surpassing relation
is defined as $a\preceq b$ if $b=a+c$ with $c\in N_G$, where $N_G$ is
a proper $G$-sub-bimodule (or equivalently a proper ideal) of $\N[G]$.

Morphisms of hyperfields, in particular canonical projections
 $\vfield\to \vfield/\subgroup$, and similarly  morphisms of idylls,
 are necessarily morphisms of 
the corresponding systems, which can be defined as follows (note that 
the conditions below
are equivalent to the ones of \cite[Def.\ 7.1]{Rowen2}, since
the tangible set contains the unit of a semiring system).

\begin{definition}[Specialization of \protect{\cite[Def.\ 7.1]{Rowen2}}]
\label{def-morphism-systems}
Let $(\semiring,\tangible,\ominus,\surpass)$ and $(\semiring',\tangible',\ominus,\surpass)$ 
be two semiring systems.
 We say that the map
$\morphismsys:\semiring\to\semiring'$ is a {\em morphism of semiring systems} if
for all $a,a'\in\semiring$ and $b\in  \tangible$, we have 
\begin{align*}
&\morphismsys(\zero)=\zero \\ %
& \morphismsys(\tangible)\subset \tangible'\\
&\morphismsys(\ominus a)=\ominus \morphismsys(a)\\
&\morphismsys(a\oplus a')\surpass \morphismsys(a)\oplus \morphismsys(a')\\
&\morphismsys(b\odot a)=\morphismsys(b)\odot \morphismsys(a)\\
& a\surpass a'\implies \morphismsys(a)\surpass \morphismsys(a').
\end{align*}
\end{definition}

For any ordered valued field $(\rfield,\vgroup,\vall,\leq)$,
we can consider, as for the field of real Hahn series, 
the hyperfield $\rfield/\subgroup$,
where $\subgroup$ is the group of elements $\elf\in \rfield$ which satisfy $\elf>0$ and $\vall(\elf)=0$. 
Then, the projection $\rfield\to\rfield/\subgroup$ is necessarily a morphism of hyperfields or of semiring systems.
Consider the semiring system composed of $\smax$ with the surpassing
relation $\preceq^{\circ}$ and
with the set of tangible elements $\smax^\vee\setminus\{\zero\}$.
Note that in that case, unique negation corresponds to
reduction of balances in \Cref{prop-nabla}.
Then, $\smax$ is isomorphic to the hyperfield system associated
to $\hyper(\vgroup)$, see \cite{AGRowen}.
It is also isomorphic to the hyperfield system associated
to $\rfield/\subgroup$. The isomorphism maps an element of 
$\rfield/\subgroup$ to the signed valuation 
of any of its elements. It identifies the element $a\in \smax^{\circ}$
with the set $\{b\in \smax^\vee \mid b\preceq^\circ a\}$.
Indeed, in the hyperfield  $\rfield/\subgroup$, the addition of
$\elf \subgroup$ with $\elf'\subgroup$ is singly valued equal to $\elf\subgroup$ or $\elf'\subgroup$, except
when $\sval(\elf)=\ominus \sval(\elf')$, in which case
the sum is the set $\{\elf''\in \rfield\mid \vall(\elf'')\leq \vall(\elf)=\vall(\elf')\}$.
Then, the signed valuation defined in \Cref{def-sign} is a morphism of semiring systems from the field 
$\rfield$ (with equality as surpassing relation, and usual negation)
to the semiring system $\smax$.
It has the additional property:
\begin{align}\label{posval}
\sval(\elf)\in \smax^\oplus\setminus\{\zero\} \iff \elf>0\enspace .
\end{align}
This is in particular the case for 
the map $\operatorname{sgn}:\rfield\to\bmaxs$.

\begin{remark}\label{rem-morphism-systems}
When $\rfield$ is a field, seen as a semiring system,
so with $\tangible=\rfield^*$, the usual negation and the equality surpassing
relation, then a
morphism $\morphismsys$ of semiring systems from $\rfield$ to $\smax$ satisfies necessarily
$\morphismsys(0)=\zero$ and $\morphismsys(\rfield^*)\subset \smax^\vee\setminus \{\zero\}$,
hence 
\begin{align*}
\elf=0 \iff \morphismsys(\elf)=\zero\enspace .
\end{align*}
Moreover,
\begin{align*}
&\morphismsys(-\elf)=\ominus \morphismsys(\elf)\\
&\morphismsys(\elf+ \elf')\surpass \morphismsys(\elf)\oplus \morphismsys(\elf')\\
&\morphismsys(\elf \elf')=\morphismsys(\elf)\odot \morphismsys(\elf').
\end{align*}
In particular, $\morphismsys(1)=\unit$, $\morphismsys(-1)=\ominus \unit$ and
$\morphismsys(\elf^2)\in \smax^\oplus\setminus\{\zero\}$, and 
$\morphismsys(-\elf^2)\in \smax^\ominus\setminus\{\zero\}$ 
for all $\elf\in \rfield^*$,
and the map $\vall: \rfield\to\tmax$ such that $\vall(\elf)=|\morphismsys(\elf)|$, is
a valuation over $\rfield$. 
If in addition, $\rfield$ is an ordered field, and if 
$\morphismsys$ satisfies \eqref{posval}
in which $\sval$ is replaced by $\morphismsys$,
then the valuation
$\vall$ is convex, and so $\rfield$ is an ordered valued field,
and $\morphismsys$ coincides with
the signed valuation $\sval$ defined in \Cref{def-sign}.
In particular, if $\rfield$ is real closed, then the order is unique and
satisfies $\elf\geq 0$ if and only if $\elf=a^2$ for some $a\in \rfield$.
Hence, using the property that
$\morphismsys(a^2)\in \smax^\oplus\setminus\{\zero\}$, we get that
$\morphismsys$ satisfies \eqref{posval}, and so the valuation
$\vall$ is convex, $\rfield$ is an ordered valued field, and
$\morphismsys$ coincides with
the signed valuation $\sval$ defined in \Cref{def-sign}.
\end{remark}

\subsection{Layered semirings}
In \cite{cramer-guterman}, we defined the notion of tropical extension
of semirings, also called {\em layered semiring} in \cite{AGRowen}.
We recall here a definition adapted to our notations, that we shall use in the following sections.

Let $\vgroup=(\vgroup,+,0,\leq)$ be a (totally) ordered abelian group,
and $\tmax=\tmax(\vgroup)$.

Assume that $(\semiring,\oplus,\zero,\odot,\unit)$ is a semiring
with a negation operator $\semiring\to \semiring$, $a\mapsto \ominus a$
(satisfying~\Cref{property-minus}, (1)--(3)).
We define the semiring  
$\skewproductstar{\semiring}{\tmax^*}$ as
the set $\semiring\times \tmax^*\cup\{(\zero,\botelt)\}$
endowed with the operations
\[
(a,g)\oplus(a',g')=\begin{cases}
(a\oplus a',g) & \text{if } g=g' \\
(a,g) & \text{if } g'< g \\
(a',g') & \text{if } g< g' 
\end{cases}
\qquad\text{and}\quad 
\begin{array}{l}
(a,g)\odot (a',g')=(a \odot a',g + g')  \\
\ominus (a,g)=(\ominus a,g)\\
|(a,g)|=g\enspace .
\end{array}
\]
This semiring has zero element $(\zero,\botelt)$ and unit $(\unit,0)$,
that shall also be denoted by $\zero$ and $\unit$ respectively.
The semiring $\semiring$ can be embedded in $\skewproductstar{\semiring}{\tmax^*}$ by the semiring morphism $a\in \semiring\mapsto (a,0)$, and so can be identified to
a subsemiring of $\skewproductstar{\semiring}{\tmax^*}$.
Moreover, $\tmax$ can be  embedded in $\skewproductstar{\semiring}{\tmax^*}$
as a multiplicative monoid by the morphism $g\in \tmax^*\mapsto (\unit,g)$
and $\botelt\mapsto (\zero,\botelt)$.
When $\semiring$ is idempotent, this morphism is also a semiring morphism,
and so $\tmax$ can be identified to a subsemiring of $\skewproductstar{\semiring}{\tmax^*}$.
In any case, we shall identify the elements of $\tmax$ with their images by this morphism, in particular, for any non-zero element $(a,g)$ of $\skewproductstar{\semiring}{\tmax^*}$, we shall see its absolute value $|(a,g)|$ as the element $(\unit,g)$ of $\skewproductstar{\semiring}{\tmax^*}$.

If $\semiring$ is also a semiring system with set of tangible elements $\tangible$, negation $\ominus$
and surpassing relation $\surpass$, then $\skewproductstar{\semiring}{\tmax^*}$ is a 
semiring system with set of tangible elements
$\tangible\times \tmax^*$, the above negation $\ominus$, and surpassing relation
\[
(a,g)\surpass (a',g')\quad \text{if }  \begin{cases}
\text{either } g< g' \; \text{and} \; \zero\surpass a' \enspace ,\\
\text{or }\;  g=g' \; \text{and} \; a\surpass a' \enspace .
\end{cases}
\]
In particular $(\skewproductstar{\semiring}{\tmax^*})^\vee
=\tangible\times \tmax^*\cup\{\zero\}$ satisfies unique negation, that is 
$(a,g),(a',g')\in\tangible\times \tmax^*\cup\{\zero\}$ and $\zero\surpass (a,g)\ominus (a',g')$ then $(a,g)=(a',g')$.

If $\semiring$ is zero-sum free and without zero divisors,
which implies that $\semiring^*$ is preserved by all the operations
$\oplus,\odot,\ominus$,
then we shall also use the notation $\skewproductstar{\semiring^*}{\tmax^*}$ 
for the subset $\semiring^*\times \tmax^*\cup\{(\zero,\botelt)\}$
endowed with all the above operations.
The semiring $\semiring$ is also embedded in $\skewproductstar{\semiring^*}{\tmax^*}$
by the morphism $a\in \semiring^*\mapsto (a,0),\; \zero\mapsto (\zero,\botelt)$,
and so can be identified to
a subsemiring of $\skewproductstar{\semiring^*}{\tmax^*}$.
Both embeddings of $\semiring$ (the present one and the above one)
send  $\tangible$ into the subgroup 
$\tangible\times \{0\}$ of $\tangible\times \tmax^*$, and they are morphisms of
semiring systems.
So we can also identify $\semiringvee$ to the subset $(\tangible\times \{0\})\cup\{(\zero,\botelt)\}$ of $(\skewproductstar{\semiring}{\tmax^*})^\vee$ and $(\skewproductstar{\semiring^*}{\tmax^*})^\vee$.

The semirings $\skewproductstar{\semiring}{\tmax^*}$ and
  $\skewproductstar{\semiring^*}{\tmax^*}$ 
are called {\em tropical extensions} of $\semiring$ by $\vgroup$
and
$\semiring$ is called the {\em layer}.

\begin{remark}\label{extendedsmax}
  The semiring system $\skewproductstar{\bmaxs^*}{\tmax^*}$
is isomorphic to $\smax(\vgroup)$ (see~\cite{cramer-guterman} or
\cite{AGRowen}).
This provides an alternative way to defined $\smax(\vgroup)$,
as a layered semiring constructed from $\bmaxs$ and $\tmax$.
\end{remark}

\section{Factorization of Polynomials over $\smax$}\label{sec-factsym}
{\em In this section, unless otherwise stated, we shall assume that $\vgroup$
is a non-trivial divisible ordered group.}
The following result implies the ``only if'' part of \Cref{theo-factored}
and characterizes $P^\sharp$.
\begin{proposition}\label{lemma-psharp}
Let $\vgroup$ be any ordered group, $P\in \smax[\Y] $, and $r_1,\ldots, r_n\in \smax$,
then $\widehat{P}$ can be factored as
\begin{equation}\label{phat}\widehat{P}(y)=  P_n\odot (y \ominus r_1) \odot \cdots \odot (y \ominus r_n)\end{equation}
if, and only if, $\mathcal{F}^{-1}(\widehat{P})$ contains
the element:
\begin{equation}\label{psharp}
 Q=  P_n (\Y \ominus r_1) \odot \cdots \odot (\Y \ominus r_n)\in \smax[\Y] \enspace.\end{equation}
Moreover, if $\vgroup$ is non-trivial and divisible, and the above conditions
hold, then $\widehat{P}$ is closed and
$P^{\sharp}= Q$ is  the maximum element of
$\mathcal{F}^{-1}(\widehat{P})$ for $\preceq$.
\end{proposition}
\begin{proof}
Let $\widehat{P}$ be factored as in~\eqref{phat} and 
 $Q \in  \smax[\Y]$ be the formal polynomial given in~\eqref{psharp},
then the polynomial function $\widehat{Q}$ coincides with $\widehat{P}$,
so that $Q\in \mathcal{F}^{-1}(\widehat{P})$.
Conversely, if $Q\in \mathcal{F}^{-1}(\widehat{P})$, then
 $\widehat{P}$ can be factored as in~\eqref{phat}.
This shows the first assertion of the lemma.

For the second one, we need to prove that  for all 
$P\in \mathcal{F}^{-1}(\widehat{Q})=\mathcal{F}^{-1}(\widehat{P})$,
\begin{equation}\label{aim}P \preceq Q.\end{equation} 
Due to the definition of partial order $\preceq$, we obtain that 
\begin{equation}\label{idempotent_or}
P_{n-k}\odot y^{\odot n-k} \preceq  \widehat{P}(y)=\widehat{Q}(y) \; \forall y\in \smax.
\end{equation}
We have 
\begin{equation}\label{formula}
Q = P_n \bigg( \bigoplus_{k=0}^n ((\ominus \unit)^k \odot R_k) \Y^{n-k} \bigg)\enspace ,
\end{equation}
with $R_0=\unit$ and 
\[ R_k=\bigoplus_{1\leq i_1 < \cdots < i_k\leq n}r_{i_1}\odot \cdots\odot  r_{i_k},\quad \text{for} \; k=1,\ldots , n.\]
W.l.o.g suppose that $|r_1|\geq \cdots\geq |r_n|$, so that 
$|r_1 \odot \cdots \odot r_k| \geq |r_{i_1} \odot \cdots \odot r_{i_k}|$
for all $1\leq i_1 < \cdots < i_k\leq n$. So we have 
\[R_k=\begin{cases}
(r_{1}\odot \cdots\odot  r_{k})^{\circ}&\text{if}\;  |r_{k+1}|=|r_k|\;\text{and}\; r_{l}\in\{\ominus r_k, r_k^\circ\}\; \text{for some} \; l \geq 1 ,\\
r_{1}\odot \cdots\odot  r_{k}&\text{otherwise.}
\end{cases}\]
In the following we will consider all the possible cases.
\bigskip

\noindent 
Case 1: Suppose that $|r_k| > |r_{k+1}|$ and choose $y\in\smax^\vee$ such that
$|r_{k+1}|< |y| < |r_k|$. Then by \Cref{formula} 
\[\widehat{Q}(y) = P_n\odot (\ominus \unit)^{\odot k} \odot r_1 \odot \cdots \odot r_k \odot y^{\odot n-k}=Q_{n-k} \odot y^{\odot n-k}.\]
So by \Cref{idempotent_or} 
 we deduce $P_{n-k} \preceq Q_{n-k}$, since $y\in \smax^\vee\setminus\{\zero\}$ 
so is invertible.
\bigskip

\noindent
Case 2: Suppose that $|r_k| = |r_{k+1}|$ and that there is no $l\geq 1$ such 
that $r_{l}\in\{\ominus r_k, r_k^\circ\}$.
In particular $r_k\not\in  \smax^\circ$ and so is invertible.
Then, $R_k=r_1\odot \cdots \odot r_k$.
Take $y= \ominus r_k$. We have
$y\ominus r_l=y$ if $l\geq k$ and $y\ominus r_l=\ominus r_l$ if $l< k$.
Then, using \Cref{idempotent_or}, and the factorization of $\widehat{P}$,
we obtain
\begin{eqnarray}
P_{n-k} \odot (\ominus r_k)^{\odot n-k} &\preceq& \widehat{P}(\ominus r_k)\nonumber\\
&=&P_n\odot (\ominus r_1) \odot \cdots \odot (\ominus  r_k) \odot (\ominus r_k)^{\odot n-k},\nonumber
\end{eqnarray}
which implies
\[P_{n-k} \preceq P_n\odot (\ominus \unit)^{\odot k} \odot r_1 \odot \cdots \odot r_k, \]
since $\ominus r_k$ is invertible.
Using \Cref{formula},  we deduce $P_{n-k} \preceq Q_{n-k}$.

\bigskip

\noindent
Case 3: Assume now that $|r_k| = |r_{k+1}|\neq \zeror$
and that there exists $l\geq 1$ such 
that $r_{l}\in\{\ominus r_k, r_k^\circ\}$. So 
$R_k=(r_{1}\odot \cdots\odot  r_{k})^{\circ}$. Then
\[Q_{n-k}=P_n\odot (r_1 \odot \cdots \odot r_k)^{\circ}.\]
Take again $y=\ominus r_k$. Applying the modulus
to \Cref{idempotent_or}, we obtain 
\begin{eqnarray*}
|P_{n-k} \odot (\ominus r_k)^{\odot n-k} |&\leq & %
|\widehat{Q}(\ominus r_k)|\\
&=&|P_n|\odot |r_1 \odot \cdots \odot r_k \odot (\ominus r_k)^{\odot n-k}| \\ 
&=& |Q_{n-k} \odot (\ominus r_k)^{\odot n-k}|,
\end{eqnarray*}
and so $|P_{n-k}| \leq |Q_{n-k} |$, since $|r_k|$ is invertible in $\tmax$.
Since $Q_{n-k}$ is balanced we can obtain $P_{n-k} \preceq Q_{n-k}$.
\bigskip

\noindent
Case 4: Assume now that $r_k=\zero$ and let $\ell$ be the smallest such
$k$ (so that $n-\ell$ is the lower degree of $Q$). Let $y$ be such that
$\zeror=|r_{\ell}|<|y|<|r_{\ell-1}|$. 
Then, by the same arguments as in Case 1, we have
$\widehat{Q}(y) = Q_{n-\ell+1} \odot y^{\odot n-\ell+1}$.
By \Cref{idempotent_or}, we get $P_{n-k} \odot y^{\odot n-k}\preceq \widehat{Q}(y)$.
We deduce $P_{n-k}\preceq Q_{n-\ell+1} \odot y^{\odot k-\ell+1}$.
Since $k\geq \ell$, letting $y$ going to $\zero$, we obtain
$P_{n-k}=\zero=Q_{n-k}\preceq Q_{n-k}$. 
\end{proof}
\begin{remark}\label{rem-nonunique}
For  a general closed (or factored) polynomial function $\widehat{P}\in \PF (\smax^{\vee})$,
the factorization of $\widehat{P}$ or equivalently $P^{\sharp}$ is not necessarily unique. So, the multiplicity cannot be defined by using this factorization. For instance
\[\Y^4 \oplus   \unit^{\circ}\Y^{3} \oplus  \unit^{\circ}\Y^2 \oplus \unit^{\circ} \Y \ominus \unit =(\Y \ominus \unit)^{3} (\Y \oplus \unit)= (\Y \oplus \unit)^{3} (\Y \ominus \unit)\]
is the maximum representative $P^\sharp$ of the closed polynomial function
$\widehat{P}$ with $P= \Y^4 \ominus \unit\in \smax^\vee[\Y]$, and 
$P^\sharp$ does not have a unique factorization.
Moreover, the same maximum $P^\sharp$ holds for
$P= \Y^4 \oplus   \Y^{3} \oplus\Y^2 \oplus \Y \ominus \unit$ and
for $P=  \Y^4 \ominus   \Y^{3} \oplus\Y^2 \ominus \Y \ominus \unit$.
\end{remark}
In the following, we shall first give a new sufficient condition
for a polynomial to be factored 
which generalizes \Cref{suf_closed} (recall that factored is equivalent to closed in the non-trivial case),
and deduce a condition for the factorization to be unique.
We shall need the following result which precises \Cref{abs_roots}
when the factorization exists.
\begin{proposition}\label{corner-modulus}
Let ${P} \in \smax^\vee[\Y]$, with $\vgroup$ divisible,
and assume that $\widehat{P}$ is factored as
\begin{equation}\label{assumpfact} \widehat{P}(y)=  P_n\odot (y \ominus r_1) \odot \cdots \odot (y \ominus r_n)\enspace .\end{equation}
Then, $|r_1|,\ldots, |r_n|$ are the corners of $|P|$ counted with multiplicities.
\end{proposition}
\begin{proof} Assume that $P$ satisfies \eqref{assumpfact}.
Using that the modulus map is a morphism, we obtain that 
the polynomial function $\widehat{|P|}$ over $\tmax$
associated to the formal polynomial $|P|$ (over $\tmax$) satisfies:
\[\widehat{|P|}|(|y|)=|\widehat{P}(y)|=|P_n|\odot
\tropprod_{1\leq j\leq n} (|y|\oplus |r_j|),\]
for all $y\in \smax$.
Applying this property to any $x\in\tmax$ such that $|y|=x$, 
and using the uniqueness of the factorization of polynomial 
functions in $\tmax$, we get the result.
\end{proof}
A formal polynomial $P$ over $\tmax$ having distinct corners 
is such that $P^\sharp$ satisfies the inequalities in \Cref{concave} strictly, 
so it cannot be the concave hull of some formal polynomial $P$ distinct 
from $P^\sharp$, and so $P=P^\sharp$ is necessarily factored, see 
\Cref{max-lemma,roots_poly}.
Therefore, the following result is a generalization of \Cref{suf_closed}.

\begin{theorem}[Sufficient condition for factorization]\label{suf_cond}
Let ${P} \in \smax^\vee[\Y]$ with $\vgroup$ a divisible ordered group.
A sufficient condition for $\widehat{P}$  to be factored is that the formal polynomial $|{P}|$ is factored (see \Cref{roots_poly}).
In that case, we have  $\widehat{P}(y)= P_n\odot (y \ominus r_1)\odot  \cdots\odot  (y \ominus r_n)$, with $r_i\in\smax^\vee$, $i=1,\ldots, n$, such that $r_i\odot P_{n-i+1}= \ominus P_{n-i}$ for all $i\leq n-\uval(P)$ and $r_i= \zero$ otherwise.
Moreover, $|r_1|\geq \cdots\geq |r_n|$ are the corners of  $|{P}|$,
counted with multiplicities.
If in addition $\vgroup$ is non-trivial, then $\widehat{P}$ is closed and 
$P^\sharp= P_n(\Y \ominus r_1)  \cdots  (\Y \ominus r_n)$.
\end{theorem}
\begin{proof} %
Assume that $|P|$ can be factored. Let $c_1\geq \cdots \geq c_n$
be the corners of $|P|$ and $k=\uval(|P|)=\uval(P)$.
Then, by \Cref{roots_poly}, $|P|=(|P|)^\sharp$ has full support,
and we have $c_{n-k+1}=\cdots = c_{n}=\zeror$ and 
\[c_i = |P_{n-i}| - |P_{n-i+1}| \quad \forall i\leq n-k\enspace.\]
Choose $r_{n-k+1}=\cdots = r_{n}=\zero$.
Let $i\leq n-k$.
Since the coefficients of $P$ are in $\smax^\vee$ and
$|P_{n-i+1}|$ and $|P_{n-i}|$ are different from $\zeror$,
$P_{n-i+1}$ and $P_{n-i}$ are invertible, so 
there exists unique $r_i\in \smax^\vee\setminus\{\zero\}$ such
 that $r_i\odot P_{n-i+1}= \ominus P_{n-i}$.
We have $|r_i|=c_i$, for all $i\leq n$.
Moreover, using the definition of the $r_\ell$, 
the following property can be obtained by induction on $i$: 
\begin{equation}\label{def-ci-rec}
 P_n\odot (\ominus r_1)\odot \cdots \odot (\ominus r_i)=P_{n-i}
\quad \forall 1\leq i\leq n-k\enspace .\end{equation}
Denote 
\[Q= P_n (\Y \ominus r_1)  \cdots  (\Y \ominus r_n),\]
and let us show that $\widehat{P}=\widehat{Q}$,
which will implies that $\widehat{P}$ can be factored 
(see \Cref{lemma-psharp}). Moreover, 
when $\vgroup$ is non-trivial, we will get that $\widehat{P}$ is closed 
with $P^\sharp=Q$, using \Cref{theo-factored} and \Cref{lemma-psharp}.
For all $y\in \smax$, we have $|\widehat{P}(y)|=\widehat{|P|}(|y|)$
and 
\[|\widehat{Q}(y)|=\widehat{|Q|}(|y|)=|P_n|\odot (|y| \oplus |r_1|) \odot \cdots \odot (|y|\oplus |r_n|).\]
Since $|P|$ is factored with corners $c_i=|r_i|$,
we deduce that  
\begin{equation}\label{equ_abs}|\widehat{P}(y)|=|\widehat{Q}(y)|.\end{equation}
To show the equality $\widehat{P}(y)=\widehat{Q}(y)$,
we will consider the following cases:

\begin{enumerate}
\item \label{cas2}
 $y= r_i$ for some $i=1,\ldots, n-k$;
\item \label{cas4}
 $y\in \smax^\circ$ and $|y|\geq |r_n|$;
\item \label{cas3} 
$y\in \smax^\vee\setminus\{r_1,\ldots, r_n\}$ or $|y|<|r_n|$.
\end{enumerate}
\bigskip
\noindent
Case~\ref{cas2}:
We have
$\widehat{P}(r_i)\succeq P_{n-i+1}\odot r_i^{\odot n-i+1}\oplus P_{ n-i}\odot r_i^{\odot n-i}= (P_{ n-i}\odot r_i^{\odot n-i})^\circ$. Moreover,  
\begin{eqnarray}\widehat{Q}(r_i)&=& P_n\odot (r_i \ominus r_1) \odot \cdots  \odot r_i^\circ \odot \cdots \odot (r_i \ominus r_n)\nonumber\\
&=& (P_n\odot r_1 \odot \cdots \odot r_{i-1}\odot r_i^{\odot n-i+1})^\circ\nonumber\\
&=& (P_{n-i}\odot r_i^{\odot n-i})^\circ,\nonumber
\end{eqnarray}
where the last equality uses \Cref{def-ci-rec}.
Therefore, $\widehat{Q}(r_i)\preceq \widehat{P}(r_i)$. So, by using \Cref{equ_abs} together with \Cref{pro_preceq}, we have $\widehat{Q}(r_i)\succeq \widehat{P}(r_i)$ and therefore $\widehat{Q}(r_i)= \widehat{P}(r_i)$.

\bigskip
\noindent
Case~\ref{cas4}: We proceed as in Case~\ref{cas2}. 
Since $|y|\geq |r_n|$, there exists $1\leq i\leq n$ such that
$|r_i|\leq |y|<|r_{i-1}|$, with the convention that $|r_0|=+\infty$.
Since we also have $y\in\smax^\circ$,  we get that $y\ominus r_j=y$
for $j\geq i$ and $y\ominus r_j=\ominus r_j$ for $j<i$. Then, 
\begin{eqnarray}\widehat{Q}(y)&=& P_n\odot (y \ominus r_1) \odot \cdots  \odot (y\ominus r_{i-1})\odot (y\ominus r_i) \odot \cdots \odot (y \ominus r_n)\nonumber\\
&=& P_n\odot (\ominus r_1 )\odot \cdots \odot (\ominus r_{i-1})\odot y^{\odot n-i+1}\nonumber\\
&=& P_{n-i+1}\odot y^{\odot n-i+1},\nonumber
\end{eqnarray}
where the last equality uses \Cref{def-ci-rec}.
Moreover, we have
$\widehat{P}(y)\succeq P_{n-i+1}\odot y^{\odot n-i+1}= \widehat{Q}(y)$.
Since $y\in \smax^\circ$, and $i\leq n$, we also get that $\widehat{Q}(y)\in
\smax^\circ$. 
So, by using \Cref{equ_abs} together with \Cref{pro_preceq} we have $\widehat{Q}(y)\succeq \widehat{P}(y)$ and therefore $\widehat{Q}(y)= \widehat{P}(y)$.
\bigskip

\noindent
Case~\ref{cas3}:
Let $y$ be such that $y\in \smax^\vee\setminus\{r_1,\ldots, r_n\}$ 
or $|y|<|r_n|$.
In all cases, there exists $1\leq i\leq n+1$ such that
$|r_i|\leq |y|<|r_{i-1}|$, with the convention that $|r_0|=+\infty$
and $|r_{n+1}|=\zeror=\botelt$.
Since $y\neq r_j$ for all $j$, we have that
$ y \ominus r_j$ is signed and it is equal to $ y$ for $j\geq i$
and $\ominus r_j$ if $j<i$.
So again
\begin{eqnarray}\widehat{Q}(y)&=& P_n\odot (y \ominus r_1) \odot \cdots  \odot (y\ominus r_{i-1})\odot (y\ominus r_i) \odot \cdots \odot (y \ominus r_n)\nonumber\\
&=& P_n\odot (\ominus r_1 )\odot \cdots \odot (\ominus r_{i-1})\odot y^{\odot n-i+1}\nonumber\\
&=& P_{n-i+1}\odot y^{\odot n-i+1},\nonumber
\end{eqnarray}
where the last equality uses \Cref{def-ci-rec}.
Since either $i=n+1$, or $y\in \smax^\vee$, we get that 
$\widehat{Q}(y)\in \smax^\vee$.

Moreover, we have
$\widehat{P}(y)\succeq P_{n-i+1}\odot y^{\odot n-i+1}= \widehat{Q}(y)$.
\Cref{not_balanced} below shows that $\widehat{P}(y)$ is not balanced.
Since the modulus of $\widehat{P}(y)$ and $\widehat{Q}(y)$
are the same and both are unbalanced,
the inequality $\widehat{P}(y)\succeq  \widehat{Q}(y)$
implies the equality.
\end{proof}

The previous proof used the following result, which shows in particular that
the only possible roots of ${P}$ are the $r_j$. 

\begin{lemma}\label{not_balanced}
Let ${P}$ and $r_1,\cdots, r_n$ be as in \Cref{suf_cond}.
Then, for all $y \in \smax$ such that either $y\in\smax^{\vee}\setminus\{r_1,\ldots, r_n\}$, or $|y|<|r_n|$, we have $\widehat{P}(y) \notin \smax^{\circ}$.
\end{lemma}
\begin{proof}
Denote by $\mv$ the lower degree of $P$.
Let us assume first that $|y|<|r_n|$, which implies that $\mv=0$.
Then, $|\widehat{P}(y)|=\widehat{|P|}(|y|)= |P_0|$ and
the maximum of $|P_{n-i}|\odot |y|^{\odot n-i}$ over $i=0,\ldots , n$ is
attained for $i=n$ only. So $\widehat{P}(y)=P_0\in\smax^\vee\setminus\{\zero\}$.

Let consider now $y\in\smax^{\vee}\setminus\{r_1,\ldots, r_n\}$,
with $|y|\geq |r_n|$, 
and suppose by contradiction that $\widehat{P}(y)\balance \zero$. 

Since 
\[\widehat{P}(y)=\bigoplus_{j=0}^{n-\mv} P_{n-j} \odot y^{\odot n-j},\]
 and all the summands are signed, using \Cref{prop-signed} in \Cref{prop-nabla}, we get that there exist $j<j'\leq n-\mv$ such that
\begin{equation}\label{cond3}
P_{n-j} \odot y^{\odot n-j}= \ominus P_{n-j'} \odot y^{\odot n-j'},\end{equation}
with maximal modulus. So 
\begin{equation}\label{abs_equ}|P_{n-j}| \odot (|y|)^{\odot n-j}=|P_{n-j'}| \odot (|y|)^{\odot n-j'}=\widehat{|P|}( |y|)\enspace.\end{equation}

Since $y\neq r_n$ and $|y|\geq |r_n|$, we deduce that $|y|>\zeror$.
 Then, using
 \Cref{abs_equ} and the concavity of $|P|$, we get the
equality of $|y|$ with the corners $c_{j+1}=\cdots=c_{j'}$.
Since $y\in\smax^{\vee}\setminus\{r_1,\ldots, r_n\}$,
this implies that $y=\ominus r_{j+1}=\cdots=\ominus r_{j'}$.
Using \Cref{def-ci-rec} for $i=j$ and $j'$, we get that 
\[P_{n-j}  \odot (\ominus r_{j+1})\odot\cdots\odot (\ominus r_{j'}) = P_{n-j'},\]
 and so 
 \[P_{n-j}  \odot y^{\odot j'-j}=  P_{n-j'}.\]
However, since  $y\neq \zero$ is invertible,
\Cref{cond3} implies
 that 
 \[P_{n-j}  \odot y^{\odot j'-j}= \ominus P_{n-j'},\]
leading to $\ominus P_{n-j'}=P_{n-j'}$, which contradicts the fact that
$ P_{n-j'}\in \smax^\vee\setminus \{\zero\}$.
\end{proof}

\begin{corollary}[Sufficient condition for unique factorization]\label{coro-uniquefact}
Let ${P} \in \smax^\vee[\Y]$ with $\vgroup$ divisible.
Assume that $|{P}|$ is factored (see \Cref{roots_poly}),
and let the $r_i$ be as in \Cref{suf_cond}.
If all the $r_i$ with same modulus are equal, 
or equivalently if for each corner $c\neq \zeror$ of $|{P}|$,
$c$ and $\ominus c$ are not both roots of $P$,
then the factorization of $\widehat{P}$ is unique (up to reordering).
\end{corollary}
In the above case, the multiplicity of a root of $P$ can be defined as
the multiplicity of its modulus as a corner of $|P|$.
\begin{proof}[Proof of {\Cref{coro-uniquefact}}]
Since $|P|$ is factored, \Cref{suf_cond} shows that $\widehat{P}$ is factored, using the $r_i$ defined as in \Cref{suf_cond},
which are such that $|r_1|\geq \cdots\geq  |r_n|$ are the corners of
$|P|$ counted with multiplicities.
By \eqref{fact-roots} (or by \Cref{not_balanced}), we have that 
$y$ is a root of $P$ if and only if $y\in  \{r_j:j=1, \ldots, n\}$.
Then, for each corner $c\neq \zeror$ of $|{P}|$, we have that 
$c=|r|$ for some root $r$ of $P$, that is $c$ or $\ominus c$ is a root of
$P$, so $c=r_i$ or $c=\ominus r_i$ for some $i$. 
If for $c\neq \zeror$, $c$ and $\ominus c$ are not both roots of $P$, then
all the $r_i$ with same modulus are equal.
Conversely, if all the $r_i$ with same modulus are equal, then
for each $r\in \smax^\vee\setminus\{\zeror\}$ (and thus for any corner $r\neq \zeror$ of $|{P}|$), we have that $r$ and $\ominus r$ are not both roots of $P$.

Assume that $\widehat{P}$ has another factorization:
\[\widehat{P}(y)=  P_n\odot (y \ominus r'_1) \odot \cdots \odot (y \ominus r'_n)\]
with $r'_i\in\smax^\vee$, $i=1,\ldots, n$, and assume that (up to some
reordering) that $|r'_1|\geq \cdots\geq |r'_n|$.
By \Cref{corner-modulus}, this implies that 
$|r'_1|,\ldots, |r'_n|$ are the corners of $|P|$
counted with multiplicities, 
hence we have $|r'_i|=|r_i|$ for all $i=1,\ldots, n$.

We also have, by \eqref{fact-roots}, that all the $r'_j$ are roots of $P$.
Hence $r'_i\in \{r_j:j=1, \ldots, n\}$, for all $i=1,\ldots, n$.
If all the $r_i$ with same modulus are equal, there is only one choice for 
all the $r'_i$ and so  $r'_i=r_i$ for all $i=1,\ldots, n$.
This shows that the factorization is unique (up to reordering).
\end{proof}

\begin{example}
Consider the polynomial $P=\Y^3\oplus 2  \Y^2\ominus 2 \Y \oplus 2$. Then, $|P|$ has a corner $2$ with multiplicity $1$ and $e=0$ with 
multiplicity $2$. We have $\widehat{P}(2)=6$ and 
$\widehat{P}(\ominus 2)=6^\circ$ so only $\ominus 2$ is a root of $P$ with modulus $2$. Similarly  $\widehat{P}(\unit)=2^\circ$ and 
$\widehat{P}(\ominus \unit)=2$ so only $\unit$ is a root of $P$ 
with modulus $0$.
We can check that $P= (\Y\oplus 2) (\Y\ominus \unit)^2$.
\end{example}

\section{Multiplicities of roots of polynomials over tropical extensions}\label{sec-multsym}
In view of \Cref{rem-nonunique}, the notion of multiplicity for a root of  a general polynomial over $\smax$ or $\smax^\vee$ cannot be defined by considering the factorizations of this polynomial, as in the case of $\tmax$. Moreover, for a polynomial with coefficients in $\smax$, it may happen that the set of roots is a continuous set of signed elements, for instance  a finite union of intervals.
Baker and Lorscheid defined in \cite{baker2018descartes} a general notion
of multiplicity for roots of polynomials over hyperfields, in terms
of a simple recursion. %
This notion can be applied in particular to the signed tropical hyperfield,
and thus to $\smax^\vee$.
Recently, 
this notion was extended by Gunn in \cite{gunn2} to polynomials over idylls.
Moreover, \cite[Th. A]{gunn2} gives 
a characterization of the multiplicities of
roots of polynomials over any ``split tropical extension'' over a 
``whole idyll'' using the notion of initial forms.
Then, using an embedding of hyperfields into the class of idylls, 
the characterization of \cite{gunn} of 
multiplicities over the signed tropical hyperfield with $\Gamma=\R$ 
is recovered and a characterization for higher rank tropical
hyperfields is obtained (higher rank tropical correspond to the case
when $\Gamma$ is a general ordered group).

As explained in \cite{AGRowen}, see also \Cref{sec-signval}, 
idylls can be realized as a subclass of semiring systems,
in which the set $\tangible$ is a group $G$, the semiring $\semiring$ is
the monoid semiring $\N[G]$ and the surpassing relation
is defined as $a\preceq b$ if $b=a+c$ with $c\in N_G$, where $N_G$ is
a $G$-sub-bimodule of $\N[G]$.
However, elements of idylls are defined as formal expressions,
these are elements of $\N[G]$, which means that the quotient with 
any possible congruence compatible with $N_G$ is not performed.
For instance, the hyperfield idyll corresponding to 
the hyperfield of signs is obtained by taking the group
$G=\{-1,1\}$, and $N_G$ the ideal of $\N[G]$ generated by all $k 1-1$,
with $k$ a positive integer. 
However, the equation $1+1=1$ is not satisfied in this hyperfield idyll, 
whereas it is satisfied in the hyperfield of signs.
Moreover, a ``split tropical extension'' is similar to what we called
a layered semiring system, although the simplifications 
$(a,g)\oplus (a,g')=(a,g)$ when $g>g'$ are not done.
Therefore, some special versions of the results that we shall state
and prove here %
 may be deduced from the ones of \cite{gunn2}, but at the price of
an identification by projection. 

In this section, we consider only
the formalism of semiring systems and of layered semiring systems.
We also propose extensions of the notion of multiplicity to the case
where the set of tangible elements $\tangible$ is not necessarily a group,
and show in \Cref{tho-mult-smax}
 a characterization of the multiplicities of a layered semiring
system.
Note that 
to state our result, we are using assumptions
which are not used in \cite{gunn2}, but are  necessary to obtain
the equivalence between two possible definitions of a root of a polynomial,
see \Cref{lem-equivmult} which is an extension of
\cite[Lemma A]{baker2018descartes}, see \Cref{rk:mult-smax}.

\subsection{Multiplicities over semiring systems}
Here, we shall use the standard notions of formal polynomial 
$P$, degree $\deg(P)$, lower degree $\uval(P)$, and associated polynomial function $\widehat{P}$,
which make sense in any semiring.
Moreover, we define the following notion of balance relation
which extends the one already introduced in $\smax$.
\begin{definition}
  Let $(\semiring,\tangible,\ominus,\surpass)$ be a semiring system.
We define the balance relation $\balance$ as follows, for $a_1,a_2\in \semiring$
\[ a_1\balance a_2\quad \text{iff}\quad 0\surpass a_1\ominus a_2\enspace .\]
Moreover, the balance relation is applied coefficient-wise to formal polynomials:
for $P,Q\in \semiring[\Y]$,
\[ P\balance Q\quad \text{iff}\quad P_i\balance Q_i\;\forall i\geq 0\enspace .\]
\end{definition}
Then, as for $\smax$, the unique negation property of the semiring system is equivalent to the reduction of balances:
\[ a_1\balance a_2\quad \text{and}\quad  a_1,a_2\in \semiringvee\quad 
\Longrightarrow a_1= a_2\enspace .\] 

\begin{definition}[Multiplicity of a root, compare with \cite{baker2018descartes}] \label{def-mult-BL}
  Let $(\semiring,\tangible,\ominus,\surpass)$ be a semiring system.
For a formal polynomial $P\in \semiringvee[\Y]$, where $\semiringvee=\tangible\cup\{\zero\}$,
and a scalar $a\in \semiringvee$, we say that $a$ is a root of $P$ if 
$\zero\surpass \widehat{P}(a)$, and
define the \new{multiplicity} 
of $a$, and denote it by $\mathrm{mult}_a(P)$, as follows.
If $a$ is not a root of $P$, set $\mathrm{mult}_a(P)=0$. 
If $a$ is a root of $P$, then 
\begin{equation}\label{mult}\mathrm{mult}_a(P)=1+\max\{\mathrm{mult}_a(Q)\mid Q\in \semiringvee[\Y],\; \lambda\in \tangible,\; \lambda P \balance (\Y \ominus a) Q\}\enspace .\end{equation}
\end{definition}
When $\semiring$ is a hyperfield system, that
is $\semiringvee$ is a hyperfield and $\surpass$ is the inclusion relation,
then $a\balance b$ is equivalent to $a\surpass b$, that is $a\in b$,
 for any $a\in  \semiringvee$ and $b\in \semiring$,
see  \cite{AGRowen}. For  $\semiring=\smax$, this is stated in \eqref{pro_preceq3} of \Cref{pro_preceq}.
Moreover, $\tangible$ is a group, hence $\lambda$ can be ommited,
by replacing $Q$ by $\lambda^{-1} Q$.
Then, the condition 
$ P \balance (\Y \ominus a)Q$  in \Cref{def-mult-BL} is equivalent to 
$P \surpass (\Y \ominus a) Q$ or $P \in (\Y \ominus a) Q$ which is used in \cite{baker2018descartes}.
Moreover, in that case, 
it is shown in  \cite[Lemma A]{baker2018descartes}  that $a$ is a root of $P\in \semiringvee[\Y]$ if and only if there exists $Q\in \semiringvee[\Y]$ such that $P \surpass (\Y \ominus a)Q$, which means if and only if the maximum in \eqref{mult} is well defined and is a non-negative integer. In particular,  $a$ is a root of $P$ if and only if $\mathrm{mult}_a(P)\geq 1$. We shall state it in
the framework of semiring systems in \Cref{lem-equivmult} below.

The results in \cite{baker2018descartes} show that the above definition of multiplicity for a  polynomial $P$ over $\tmax$ coincides with the one given by the factorization in \Cref{factor}.
The same does not hold in  $\smax$, because of the non-uniqueness of factorization.
Moreover, the multiplicity depends on the coefficients of the formal polynomial and not only on the corresponding polynomial function as shown in \Cref{examplebmax} below.

The following assumptions on semiring systems will be used to extend
\cite[Lemma A]{baker2018descartes}.
\begin{definition}\label{def-properties-systems}
Let $(\semiring,\tangible,\ominus,\surpass)$ be a (commutative) semiring system.
\begin{enumerate}
\item {\bf Weak tangible  balancing:}\label{wtb} We say that  $\semiring$ satisfies  \new{weak tangible  balancing} if for any
$a\in \semiring$, there exists $a' \in \semiringvee$ such that $a'\balance a$.

\item {\bf Tangible  balancing:}\label{tb} We say that $\semiring$ satisfies \new{tangible  balancing} if for any
pair $(a_1,a_2)$ of elements of $\semiring$ such that $a_1\balance a_2$, 
there exists $a \in \semiringvee$ such that $a_1\balance a$ and $a\balance a_2$.

\item{\bf Tangible balance elimination:}\label{tbe} We say that $\semiring$ satisfies \new{tangible balance elimination}  or \new{weak transitivity} if for any $a \in\semiringvee$, and  $a_1,a_2 \in \semiring$, 
$ a_1   \balance  a$ and $a \balance a_2$ imply $a_1 \balance a_2$.

\item {\bf Balance cancellation:}\label{bc} We say that $\semiring$ satisfies \new{balance cancellation} if for any $a\in\tangible$ and $a_1,a_2\in\semiring$ such that $a\odot a_1 \balance a \odot a_2$, we have $a_1\balance a_2$.

\item {\bf Balance inversion:}\label{bi} We say that $\semiring$ has \new{balance inversion} if for any $a,a'\in\semiringvee$ and $a_1\in\semiring$ such that $a\odot a_1 \balance a'$, 
there exists $a'' \in \semiringvee$ such that $a''\balance a_1$ and $a\odot a''=a'$.
\end{enumerate}
\end{definition}
Weak transitivity was stated for $\smax$ in \Cref{prop-nabla}. It was already used, together with reduction of balances (that is unique negation), in the context of Cramer theorem \cite{cramer-guterman}. The semiring $\smax$ also satisfies tangible balancing. If $\tangible$ is a group,
then balance inversion is deduced from the properties of a semiring system.
Balance cancellation is equivalent to the condition that $a\odot a_1 \balance \zero$ implies $a_1\balance \zero$ (by taking $a_1\ominus a_2$) and is implied by balance inversion (by taking $a'=\zero$).
It should allow to extend by localization the semiring system $\semiring$,
in order to embed it into semiring system such that the set of
tangible elements is a group,
however, we prefer below to state the results in the smaller
and more general semiring system which satisfies balance cancellation.
As said in \Cref{sec-signval}, any hyperfield gives rise to a system,
called a hyperfield system. The same holds for hyperrings. Moreover, 
a hyperring system satisfies necessarily tangible  balancing,
tangible balance elimination and balance inversion~\cite{AGRowen},
and hyperfield systems are such that $\tangible$ is a group.
Hyperring systems also satisfy the ``faithfully balanced'' property,
meaning that the application  $b\in \semiring\mapsto \{a\in\tangible\mid a\balance b\}$ is injective~\cite{AGRowen}, a property that we shall not use later.
Also, we shall not assume that $\tangible$ is a group,
and not assume balance inversion property, but only balance cancellation
which is weaker.
Moreover, even under all the above conditions, when the hyperfield is not doubly distributive, the multiplication in the semiring is not equivalent to the one of the hyperfield, except when multiplying an element of the hyperfield with a sum of elements of the hyperfield~\cite{AGRowen}. 
Here, although the same results can be written using the notations of hyperfields, we shall state all the results over semiring systems satisfying a minimal set of properties, namely tangible  balancing, tangible balance elimination and balance cancellation, while using balance relations instead of inclusions or surpassing relations. Indeed, the semiring properties allow to avoid the difficulties due to the lack of double distributivity and the balance relation is also more convenient to work with because of its symmetry.

Let us first state \cite[Lemma A]{baker2018descartes} with these notions and 
notations, and with a slightly more general assumption. Indeed 
\cite[Lemma A]{baker2018descartes} is stated for a hyperfield,
which is equivalent to adding the assumption that $\tangible$ is a group, 
which is weaken here in the balance cancellation property. 
This result will follow from \Cref{lem10} and \Cref{lem2}
which are proved under weaker assumptions.
Note that Gunn also related this property to the notion of pastures,
see \cite[Sec.\ 4.1]{gunn2}.

\begin{lemma}[Compare with \protect{\cite[Lemma A]{baker2018descartes}}]
\label{lem-equivmult}
 Let $(\semiring,\tangible,\ominus,\surpass)$ be a semiring system satisfying 
tangible balancing, tangible balance elimination and balance cancellation properties  (\crefrange{tb}{bc} of \Cref{def-properties-systems}).
Then,  $a\in\semiringvee$ is a root of  $P\in \semiringvee[\Y]$ if and only if there exist $\lambda\in\tangible$ and $Q\in \semiringvee[\Y]$ such that $\lambda P \balance (\Y \ominus a)Q$, that is $\mathrm{mult}_a(P)\geq 1$.
\end{lemma}

The ``if'' part of the first assertion of \Cref{lem-equivmult} can be proved under weaker assumptions, which are indeed similar to the ones used for Cramer theorem in \cite{cramer-guterman}. Moreover, we first state some properties which need no assumptions. 
\begin{lemma}\label{lem1} Let $(\semiring,\tangible,\ominus,\surpass)$ be a semiring system. 
Let  $a\in\semiringvee$, and $P,Q\in \semiringvee[\Y]$ such that $P \balance (\Y \ominus a)Q$.
Then, $\deg(Q)=\deg(P)-1$, and $Q_{n-1}=P_n$.
\end{lemma}
\begin{proof}
Let  $a\in\semiringvee$, and $P,Q\in \semiringvee[\Y]$ such that $P \balance (\Y \ominus a)Q$. Denote $n=\deg(P)$ the degree of $P$.
We shall denote by $(P_k)_{k\in {\N}}$ and $(Q_k)_{k\in\N}$
the sequences of coefficients of $P$ and $Q\in \semiringvee[\Y]$ respectively.

Consider first the case $a=\zero$.
Since $P\balance \Y Q$, we get in particular
that $Q_k\balance P_{k+1}=\zero$ for all $k\geq n-1$. Since
$Q_k\in \semiringvee$, using definition of
surpassing relation (in \Cref{defi-system}),
this implies that $Q_k=\zero$, so $\deg(Q)\leq\deg(P)-1$.
Since $Q_{n-1}\balance P_{n}\in \semiringvee\setminus\{\zero\}$, using 
unique negation, we get that $Q_{n-1}=P_{n}\in \semiringvee\setminus\{\zero\}$,
and $\deg(Q)=\deg(P)-1$.

Consider now the case where $a\in \semiringvee\setminus\{\zero\}=\tangible$.
The condition $P \balance (\Y \ominus a)Q$ is equivalent to
\begin{equation}\label{eq-PQ1}
 P_{k} \balance Q_{k-1} \ominus a \odot Q_{k},\;\text{for all}\; k\geq 1,\quad
\text{and}\quad P_{0} \balance \ominus a \odot Q_{0}\enspace .\end{equation}
Using \eqref{eq-PQ1} for $k \geq n+1$, we have
\[\zero=P_k \balance Q_{k-1} \ominus a\odot Q_k,\]
and since $Q_k \in \semiringvee$ for all $k$, we deduce from unique negation
\begin{equation}\label{equ_q}
Q_{k-1}=a\odot Q_k\quad \text{for all}\; k \geq n+1\enspace .
\end{equation} 
Since $Q_k=\zero$ for large $k$, \Cref{equ_q} shows that $Q_k=\zero$ for $k \geq n$ and therefore $\deg(Q)\leq n-1$.
Also, using \eqref{eq-PQ1} for $k=n$, we obtain
$P_n\balance Q_{n-1} \ominus a\odot Q_n=Q_{n-1}$, and since $P_n$ and $Q_{n-1}$
belong to $\semiringvee$, we get $P_n=Q_{n-1}$ from unique negation.
Now $\deg(P)=n$ implies 
$P_n\neq \zero$ and so $Q_{n-1}\neq \zero$ and $\deg(Q)=n-1$.
\end{proof}

\begin{lemma}\label{lem10} Let $(\semiring,\tangible,\ominus,\surpass)$ be a semiring system satisfying tangible balance elimination property (\cref{tbe} of \Cref{def-properties-systems}).
Let  $a\in\semiringvee$, and $P,Q\in \semiringvee[\Y]$ such that $P \balance (\Y \ominus a)Q$.
Then, $a$ is a root of $P$. 
\end{lemma}
\begin{proof}
Let us use the same notations as in previous proof.

Consider first the case $a=\zero$.
Since $P\balance \Y Q$, we get in particular
that  $P_0\balance \zero$, and since $\widehat{P}(\zero)=P_0$, 
we obtain that $\widehat{P}(\zero)\balance \zero$, so $\zero$ is a root of $P$.

Consider now the case where $a\in \semiringvee\setminus\{\zero\}=\tangible$.
Let us show by induction that 
 \begin{equation}\label{inductive}
 P_0 \oplus \cdots \oplus P_k\odot a^{\odot k} \balance \ominus  Q_k \odot a^{\odot k+1}\quad \forall k\geq 0\enspace .
 \end{equation} 
This is true for $k=0$, since $P_{0} \balance \ominus  a \odot Q_{0}$. 
Assume that \eqref{inductive} holds for $k\geq 0$.
By definition of balance relation, \eqref{eq-PQ1} for $k+1$ is equivalent to
$\ominus  Q_k \balance \ominus a \odot Q_{k+1} \ominus P_{k+1}$,
which implies $\ominus   Q_k\odot a^{\odot k +1} \balance \ominus  Q_{k+1}\odot a^{\odot k+2} \ominus  P_{k+1}\odot  a^{\odot k+1 }$.
 Applying  tangible balance elimination to this balance relation and \eqref{inductive} with $\ominus   Q_k \odot a^{\odot k +1} \in \semiringvee$, we get that
\[P_0 \oplus \cdots \oplus P_{k}\odot  a^{\odot k} \balance \ominus   Q_{k+1} \odot a^{\odot k+2}\ominus P_{k+1}\odot   a^{\odot k+1}\enspace,\]
which is equivalent to \eqref{inductive} for $k+1$, and shows the induction.
In particular, for $k=n$, we have 
\[\widehat{P}(a)=P_0 \oplus \cdots \oplus P_n \odot  a^{\odot n} \balance \ominus Q_n \odot  a^{\odot n+1} \enspace .\]
Since $\deg(Q)=n-1$, by \Cref{lem1}, we get that  $Q_n=\zero$, hence the 
previous equation shows that $a$ is a root of $P$.
This finishes the proof of the lemma.
\end{proof}

\begin{remark}\label{rem-balance-necessary}
The above assumption that the semiring satisfies 
tangible balance elimination property 
seems necessary to get the assertion of \Cref{lem10}.
Let  $P$ is of degree $3$ and
assume that $P\balance (\Y\ominus \unit) Q$ for some polynomial $Q$.
By \Cref{lem1}, $Q$ has degree $2$ and satisfies 
\[P_3 \balance Q_2,\; P_2\balance Q_1\ominus Q_2,\;  P_1\balance Q_0\ominus Q_1,\; P_0\balance \ominus Q_0\enspace .\]
Using unique negation, this is equivalent to $Q_2=P_3$, $Q_0= \ominus P_0$, together with
\[ Q_1\balance  P_3 \oplus P_2,\;  Q_1\balance \ominus P_1 \ominus P_0 \enspace .\]
Moreover, $\unit$ is a root of $P$ if and only if 
$ P_3 \oplus  P_2\balance \ominus P_1\ominus P_0$.
Therefore the existence of $Q$ implies that $\unit$ is a root,
if and only if tangible balance elimination property is
satified for the elements $a_1= P_3 \oplus P_2$ and  
$a_2= \ominus P_1 \ominus P_0$, that is for elements  $a_1$ and $a_2$ 
of hight $2$ (as defined in \cite{Rowen2}).
\end{remark}

\begin{corollary}\label{lem1-cor} Let $(\semiring,\tangible,\ominus,\surpass)$ be a semiring system satisfying tangible balance elimination and balance cancellation  properties (\cref{tbe,bc} of \Cref{def-properties-systems}).
Let  $a\in\semiringvee$, $\lambda\in\tangible$, and $P,Q\in \semiringvee[\Y]$ such that $\lambda P \balance (\Y \ominus a)Q$.
Then, $a$ is a root of $P$, $\deg(Q)=\deg(P)-1$, and $Q_{n-1}=\lambda P_n$.
\end{corollary}
\begin{proof}
The polynomial $\lambda P$ has same degree as $P$ since
$P\in \semiringvee[\Y]$ and $\deg(P)=n$ means that $P_m=\zero$ for $m\geq n+1$ and
$P_n\in \tangible$, which implies $(\lambda P)_n=\lambda \odot P_n\in \tangible$.  Applying \Cref{lem10} to $\lambda P$ and $Q$, we get that $a$ is a root of $\lambda P$, that is $\lambda \odot \widehat{P}(a)\balance \zero$. Applying balance cancellation property, we obtain that $\widehat{P}(a)\balance \zero$.
\end{proof}

We also get the following result which says that in \Cref{def-mult-BL}, 
one can restrict the set of 
polynomials $Q\in \semiringvee[\Y]$ to be in 
the subset of polynomials with degree $n-1$.
\begin{corollary}
\label{cor-mult-degree}
Let $(\semiring,\tangible,\ominus,\surpass)$ be a semiring system satisfying tangible balance elimination and balance cancellation  properties (\cref{tbe,bc} of \Cref{def-properties-systems}).
Let $P\in \semiringvee[\Y]$ be of degree $n$, let  $a\in\semiringvee$,
and let $1\leq m\leq n$.
Then, $\mathrm{mult}_a(P)\geq m$, if and only if there exist $\lambda\in\tangible$ and $Q\in \semiringvee[\Y]$ of degree $n-1$ such that $\lambda P \balance (\Y \ominus a)Q$ and $\mathrm{mult}_a(Q)\geq m-1$.
Moreover, when $\tangible$ is a group, one can choose $\lambda=\unit$.
\end{corollary}
\begin{proof}
The ``if'' part of the first assertion
follows from the definition of $\mathrm{mult}_a(P)$ in
 \Cref{def-mult-BL}, as we only put a restriction on $Q$. The ``only if'' part
follows from the same definition and from \Cref{lem1-cor}.
\end{proof}

The ``only if'' part of \Cref{lem-equivmult} can also be proved under weaker assumptions, as follows.
\begin{lemma}\label{lem2} Let $(\semiring,\tangible,\ominus,\surpass)$ be a semiring system satisfying tangible balancing property (\cref{tb} of \Cref{def-properties-systems}).
Let $P\in \semiringvee[\Y]$ and let $a\in\semiringvee$ be a root of $P$.
Then, there exists $\lambda\in \tangible$, and $Q\in \semiringvee[\Y]$  such that $\lambda P \balance (\Y \ominus a)Q$.
\end{lemma}
\begin{proof}
If $a=\zero$ is a root of $P$, then $P_0=\widehat{P}(\zero)\balance \zero$,
so $P\balance P_1 \Y\oplus\cdots \oplus P_n \Y^n=\Y Q$ for some $Q\in \semiringvee[\Y]$ (which has degree $n-1$ and is such that $Q_{n-1}=P_n$).

Suppose $a\neq \zero$. We first reduce us to the case $a=\unit$,
for which we shall prove the result with $\lambda=\unit$.
Indeed, consider the formal polynomial $P'(\Y)=P(a\Y)$.
If $P'\balance Q'(Y\ominus \unit)$ with $Q'\in \semiringvee[\Y]$,
then by \eqref{eq-PQ1}, we have
$P'_{k} \balance Q'_{k-1} \ominus Q'_{k}$, for all $k\geq 0$ 
(with $Q_{-1}=\zero$).
Multiplying these equations by $a^{\odot n-k}$, 
and using that $P'_k= P_k \odot a^{\odot k}$, we get
$a^{\odot n}\odot P_k\balance  a^{\odot n-k} (Q'_{k-1} \ominus Q'_{k})$.
Considering the polynomial $Q\in  \semiringvee[\Y]$ with degree $n-1$,
such that $Q_k= a^{\odot n-k-1} \odot Q'_k$ for $0\leq k\leq n-1$,
we obtain $a^{\odot n}\odot P_k\balance  Q_{k-1} \ominus a\odot Q_{k}$,
that is \eqref{eq-PQ1} with $P_k$ replaced by $a^{\odot n}\odot P_k$.
Hence $\lambda P\balance (Y\ominus a) Q$ for $\lambda= a^{\odot n}\in \tangible$.

Assume now that $a=\unit$.
Let us construct $Q$ of degree $n-1$ as follows. Start with $Q_{n-1}=P_n\neq \zero$ (and $Q_k=\zero$ for $k\geq n$).
Since $\unit$ is a root of $P$, we have
\[\widehat{P}(\unit)=P_0 \oplus \cdots \oplus P_n  \balance \zero \enspace .\]
Hence $P_0 \oplus \cdots \oplus P_{n-2} \balance  \ominus (P_{n-1}\oplus Q_{n-1})$.
Choose $Q_{n-2}\in \semiringvee$ such that
$P_0 \oplus \cdots \oplus P_{n-2} \balance \ominus  Q_{n-2}$ and $Q_{n-2}\balance P_{n-1}\oplus Q_{n-1}$.
This is possible using  tangible balancing property.

Similarly, by descending induction on $0\leq k\leq n-1$, one can construct 
$Q_k\in \semiringvee$ such that
\begin{equation}\label{cond-rec}
 P_0 \oplus \cdots \oplus P_{k}\balance \ominus Q_{k}\quad \text{and}\quad Q_{k}\balance P_{k+1}\oplus Q_{k+1} \enspace .\end{equation}
The right equations in \eqref{cond-rec} are also true for $k\geq n$, and then 
imply that  $Q_{k-1}\ominus Q_{k} \balance P_{k}$ for all $k\geq 1$, that is the first equation in \eqref{eq-PQ1}.
Moreover, for $k=0$, the left equation in \eqref{cond-rec}
implies $P_0\balance \ominus Q_{0}$ which is the right equation in \eqref{eq-PQ1}.
We have shown all the equations in \eqref{eq-PQ1}, so $P\balance Q(Y\ominus \unit)$ for some polynomial $Q$ of degree $n-1$, which is the result of the lemma with $\lambda=\unit$.
\end{proof}
\begin{remark}\label{rem-balancing-necessary}
The above assumption that the semiring satisfies tangible balance property
seems necessary to get the assertion of \Cref{lem2}.
Indeed, assume as in \Cref{rem-balance-necessary}, 
that $\tangible$ is a group, so that the factor $\lambda$ is not needed.
Consider $P$ of degree $3$  and assume that $\unit$ is a root of $P$
which means that $ P_3 \oplus  P_2\balance \ominus P_1\ominus P_0$.
Then, by the same arguments as in \Cref{rem-balance-necessary},
the existence of $Q$ such that $P\balance (\Y\ominus \unit) Q$ is equivalent to the tangible balance property
for the elements $a_1= P_3 \oplus P_2$ and  $a_2= \ominus P_1 \ominus P_0$,
that is for elements  $a_1$ and $a_2$ of hight $2$.
\end{remark}

\subsection{Characterization of multiplicities of polynomial roots}
The following result of \cite{baker2018descartes} characterizes
the  multiplicities of polynomial roots over the hyperfield of signs
and was used there to reprove the  Descartes' rule of signs.
We rewrite it using the equivalence between $\bmaxs^\vee$ and 
the hyperfield of signs.
\begin{theorem}[\protect{\cite[Th.\ C]{baker2018descartes}}]\label{baker31}
Let $P=\bigoplus_{i=0}^n P_i \Y^i\in \bmaxs^\vee[\Y]$, be a formal polynomial 
with degree $n$ (that is such that $P_n\neq \zero$).
Then, $\mult_{\unit}(P)$ is the number of sign changes in the (non-zero)
coefficients $P_i$ of the formal polynomial $P$:
\begin{align*}
\mult_{\unit}(P)= &\card\{ i\in \{0,\ldots,  n-1\}\mid \exists k\geq 1,
\; P_i=\ominus P_{i+k}\neq \zero,\\
&\qquad\qquad  P_{i+1}=\cdots =P_{i+k-1}=\zero\} \enspace .\end{align*}
Moreover, $\mult_{\ominus \unit}(P)$ is the number of sign changes in
the coefficients of the formal polynomial $P(\ominus \Y)$.
\end{theorem}

\begin{example}\label{examplebmax} 
Consider the polynomials of \Cref{rem-nonunique}:
$P_1= \Y^4 \ominus \unit$, 
$P_2= \Y^4 \oplus   \Y^{3} \oplus\Y^2 \oplus \Y \ominus \unit$ and
$P_3=  \Y^4 \ominus   \Y^{3} \oplus\Y^2 \ominus \Y \ominus \unit$.
They all have same canonical form $P^\sharp$, so same polynomial
function and two possible factorizations.
These formal polynomials have only $\zero$, $\unit$, and $\ominus \unit$
coefficients, so they can be considered as polynomials over the
symmetrized Boolean semiring $\bmaxs$, and the multiplicities of roots in
$\bmaxs$ and $\smax$ coincide.
Therefore, using \Cref{baker31}, we get that
 $\unit$ is a root of 
$P_1$, $P_2$, $P_3$ with respective multiplicities $1$, $1$ and $3$.
Similarly, $\ominus \unit$ is a root of 
$P_1$, $P_2$, $P_3$ with respective multiplicities $1$, $3$ and $1$.
\end{example}

In  \cite{gunn} and \cite{gunn2}, Gunn extended the above
result to the case of 
polynomials over the signed tropical hyperfield, 
and to the case of split tropical extension of a whole idyll, respectively.
The proof there %
is based on the notion of initial forms.
To prove our general result \Cref{tho-mult-smax}, 
we use the notion of saturated polynomial, as in
\cite{akian2016non}, which is indeed similar to the one of initial forms.
The word ``saturated'' is related to optimization and graphs.

\begin{theorem}[Multiplicities of roots in a tropical extension]\label{tho-mult-smax}
 Let $(\semiring,\tangible,\ominus,\surpass)$ be a semiring system satisfying 
tangible balancing, tangible balance elimination and balance cancellation properties (\crefrange{tb}{bc} of \Cref{def-properties-systems}).
Let $\semiring'$ be either $\semiring$ or $\semiring^*$ assuming that $\semiring'$ is stable by all the operations $\oplus,\odot,\ominus$.
Let $\vgroup$ be any non-trivial divisible ordered group.
Consider the tropical extension
$\extension:=\skewproductstar{\semiring'}{\tmax^*}$ 
of $\semiring$ by $\vgroup$ and let us use the identification of 
$(\tangible\times \{0\})\cup\{\zero\}\subset \extension^\vee$ with
$\semiringvee$ and also of $\tmax$ as a subset of $\extension^\vee$.
Then, $\extension$ is a semiring system satisfying 
tangible balancing, tangible balance elimination and balance cancellation
properties, and 
multiplicities of polynomials over $\extension$ can be computed using multiplicities over $\semiring$ as follows.

Let $P=\bigoplus_{i=0}^n P_i \Y^i\in \extension^\vee[\Y]$, be a non-zero formal 
polynomial, with degree $n\geq 0$ and lower degree $\mv\geq 0$.

For $r=\zero$, the multiplicity $\mult_{r}(P)$ of $r$ as a root of $P$,
in the sense of \Cref{def-mult-BL}, 
is equal to $\mv$.

For $r\in\extension^\vee\setminus\{\zero\}$, the following properties hold.

\begin{enumerate}
\item\label{tho-mult-smax1}  $|r|$ is invertible, and the normalization $\bar{r}$ of $r$, defined as
$\bar{r}  := |r|^{\odot -1}\odot r$ belongs to $\semiringvee\setminus\{\zero\}=\tangible$. 
\item\label{tho-mult-smax2} $ |\widehat{P}(r)|=\widehat{|P|}(|r|)\in \tmax^*$ is invertible and only depends on $|r|$. Let us define the normalized polynomial $\Pn^r$ as
\begin{align*}
&\Pn^r:=|\widehat{P}(r)|^{\odot - 1} P(r\Y)\in 
\extension^\vee[\Y]
\end{align*}
meaning $\Pn^r=\bigoplus_{i=0}^n \Pn^r_i \Y^i$ with $\Pn^r_i= |\widehat{P}(r)|^{\odot - 1} \odot r^{\odot i}\odot  P_i$.
\item\label{tho-mult-smax3} The saturation set $\sat(r,P)$, defined as
\begin{align*}
& \sat(r,P):= \{i\in \{0,\ldots, n\} \mid
|\widehat{P}(r)|= |P_i | \odot |r|^{\odot i}\}
\end{align*}
only depends on $|r|$, and coincides with $\sat(\unit, \Pn^r)$ and $\sat(\bar{r}, \Pn^{|r|})$. 
\item \label{tho-mult-smax4} The saturation polynomial $P^{\sat,r}$, defined as
\begin{align*}
& P^{\sat,r}(\Y):= \bigoplus_{i\in \sat(r,P)} \Pn^r_i \Y^i \in \extension^\vee[\Y]\enspace,\end{align*}
has coefficients in $\semiringvee$. 
\item \label{tho-mult-smax5} We have
\[ %
\mult_{r}(P)=\mult_{\bar{r}}(\Pn^{|r|})=\mult_{\bar{r}}(P^{\sat,|r|})\enspace.
\] %
\item \label{tho-mult-smax6} If $r$ is invertible, then $\bar{r}$ is invertible and
we have
\[ \mult_{r}(P)=\mult_{\unit}(\Pn^{r})=\mult_{\unit}(P^{\sat,r})\enspace.
\]
\item \label{tho-mult-smax4bis} If $|r|=\unit$ (or equivalently $r\in\semiringvee\setminus\{\zero\}$) and $P$ has coefficients in
$\semiringvee$, then considering $P$ as a formal polynomial over 
 $\semiringvee$ or over
$\extension^\vee$ does not change the multiplicity $\mult_{r}(P)$.
In particular $\mult_{\bar{r}}(P^{\sat,|r|})$ and $\mult_{\unit}(P^{\sat,r})$ above can be computed either in $\semiringvee$ or in $\extension^\vee$.
\end{enumerate}
\end{theorem}
\begin{remark} \label{rk:mult-smax}
Let us compare \Cref{tho-mult-smax}  to the recent result 
of Gunn \cite[Th.\ A and B]{gunn2} proved in the setting of ``whole idylls''.
As said before, idylls are particular cases of semiring systems 
in which the set of tangible elements is a group,  up to some identification by projection.
Moreover, split tropical extensions of idylls (as in \cite{gunn2})
are similar to tropical extensions of these idylls seen as
semiring systems, and the condition that an idyll is a ``whole idyll''
is similar to the weak tangible balancing condition
that there exists $a' \in \semiringvee$ such that $a'\balance a$,
for elements $a$ which are of the form $a=b+c$ with $b,c\in \semiringvee$
(called elements of height $2$ in \cite{Rowen2}).

In that case, the conclusion of \Cref{tho-mult-smax} is similar as
the one of \cite[Th.\ A]{gunn2}, and may certainly be deduced by applying
an appropriate projection.
Note that we stated our result using more restrictive assumptions, since
these assumptions are necessary to obtain the
equivalence between the two definitions of a root of a polynomial,
see \Cref{lem-equivmult}.
In \cite{gunn2}, Gunn chose to define a root 
 by the condition that the multiplicity is $\geq 1$ only, which would make 
most of the assumptions unnecessary to prove \Cref{tho-mult-smax}, as
can be checked in \Cref{prooftheo-mult} below.
Indeed, for this proof, we only need
the weak tangible balancing condition for elements of height $2$.
Therefore \Cref{tho-mult-smax} and \cite[Th.\ A and B]{gunn2}
are comparable up to some projection of a whole idyll onto the semiring system,
and also a localization of the semiring system (enlarging the semiring system 
and the set of tangible elements in order to get a group).
\end{remark}

As a corollary of \Cref{tho-mult-smax} and \cite[Th.\ C]{baker2018descartes} (see \Cref{baker31}),
we recover the characterization of the multiplicities 
of polynomials over $\smax$, already proved by Gunn \cite[Theorem A]{gunn}
when $\vgroup=\R$, in the setting of hyperfields.

\begin{corollary}[Multiplicities of roots in $\smax$ (see \protect{\cite[Theorem A]{gunn}} for $\vgroup=\R$)]\label{cor-mult-smax}
Assume that $\vgroup$ is a non-trivial divisible ordered group.
Let $P=\bigoplus_{i=0}^n P_i \Y^i\in \smax^\vee[\Y]$, be a non-zero formal 
polynomial, with degree $n\geq 0$ and lower degree $\mv\geq 0$,
and let $r\in\smax^\vee$. 
For $r=\zero$, the multiplicity $\mult_{r}(P)$ of $r$ as a root of $P$,
in the sense of \Cref{def-mult-BL}, 
is equal to $\mv$.

For $r\in\smax^\vee\setminus\{\zero\}$, the following properties hold.
\begin{enumerate}
\item  $r$ is invertible.
\item $ |\widehat{P}(r)|=\widehat{|P|}(|r|)\in \tmax^*$ is invertible and only depends on $|r|$. Let us define the normalized polynomial $\Pn^r$ as
\begin{align*}
&\Pn^r:=|\widehat{P}(r)|^{\odot - 1} P(r\Y)\in 
\extension^\vee[\Y]
\end{align*}
meaning $\Pn^r=\bigoplus_{i=0}^n \Pn^r_i \Y^i$ with $\Pn^r_i= |\widehat{P}(r)|^{\odot - 1} \odot r^{\odot i}\odot  P_i$.
\item The saturation set $\sat(r,P)$, defined as
\begin{align*}
& \sat(r,P):= \{i\in \{0,\ldots, n\} \mid
|\widehat{P}(r)|= |P_i | \odot |r|^{\odot i}\}
\end{align*}
only depends on $|r|$, and coincides with $\sat(\unit, \Pn^r)$.
\item The saturation polynomial $P^{\sat,r}$, defined as
\begin{align*}
& P^{\sat,r}(\Y):= \bigoplus_{i\in \sat(r,P)} \Pn^r_i \Y^i \in \extension^\vee[\Y]\enspace,\end{align*}
has coefficients in $\{\zero,\unit,\ominus\unit\}=\bmaxs^\vee\subset \smax^\vee$.
 Then, 
considering $P^{\sat,r}$ as a formal polynomial over $\bmaxs^\vee$  or over
$\smax^\vee$ does not change the multiplicity $\mult_{\unit}(P^{\sat,r})$.
\item\label{last-gunn} We have
\[ \mult_{r}(P)=\mult_{\unit}(\Pn^{r})=\mult_{\unit}(P^{\sat,r})\enspace,
\]
\end{enumerate}
and they are equal to the number of sign changes in the coefficients of $P^{\sat,r}$.
\end{corollary}
 \begin{example}
 Turning back to \Cref{ex_sym_plot}. To compute the multiplicity of the root $\rsmax=\ominus -1$ of 
 $P = \Y^5\oplus 4 \Y^3\oplus \Y \oplus 1$ we have
 $\sat(\ominus -1,P)=\{0,3\}$. Therefore

 \[P^{\sat,\ominus -1}(\Y)=(-1\odot (\ominus -1)^{\odot 3}\odot 4)\Y^{3}\oplus -1 \odot 0 \odot 1 =\ominus 4 \Y^3 \oplus \unit\enspace,\]
 which contains only one sign change and thus 
  $\mult_{\ominus -1}(P)=1$\enspace.
 \end{example}

\subsection{Properties of semiring extensions}
For the proof of \Cref{tho-mult-smax}, we shall need the following lemmas and properties, some of them generalizing the properties already stated for $\smax$, to the case of general semiring extensions.

\begin{property}\label{pro_preceq-ext}
Let $(\semiring,\tangible,\ominus,\surpass)$ be a semiring system, and
let $\semiring'$ be either $\semiring$ or $\semiring^*$ assuming that $\semiring'$ is stable by all the operations $\oplus,\odot,\ominus$.
Let $\vgroup$ be any non-trivial divisible ordered group.
Consider the tropical extension
$\extension:=\skewproductstar{\semiring'}{\tmax^*}$ 
of $\semiring$ by $\vgroup$.

 We have the following properties for all $b_1,b_2 \in \extension$:
\begin{enumerate}
\item \label{pro_preceq-ext1} $b_1 \surpass b_2$ implies $|b_1| \leq |b_2|$;
\item \label{pro_preceq-ext2}If $|b_1| \leq |b_2|$ and $\zero\surpass b_2$, then $b_1 \surpass b_2$;
\item \label{pro_preceq-ext3} If  $b_1\in\extension^\vee$ and $b_1\balance b_2$ then $|b_1| \leq |b_2|$.
\end{enumerate} 
Moreover, we have
\begin{enumerate}
 \setcounter{enumi}{3}
\item\label{prop-exten-1}
 If $\semiring$ satisfies tangible  balancing
(\cref{tb} of \Cref{def-properties-systems}), then so does
$\extension$.
\item\label{prop-exten-0}
 If $\semiring$ satisfies weak tangible  balancing
(\cref{tb} of \Cref{def-properties-systems}), then so does
$\extension$.
\item \label{prop-exten-2}
If $\semiring$ satisfies tangible balance elimination
(\cref{tbe} of \Cref{def-properties-systems}), then so does
$\extension$.
\item \label{prop-exten-3} If $\semiring$ satisfies balance cancellation
(\cref{bc} of \Cref{def-properties-systems}), 
then so does $\extension$.
\end{enumerate} 
\end{property}
\begin{proof}
The first three properties follow from the definition of addition and negation
in $\extension$. 
For the proof of the last properties, we remark that for
$(a_1,g_1),(a_2,g_2)\in \extension$, we have that $(a_1,g_1)\balance (a_2,g_2)$
if and only if $\zero\surpass (a_1,g_1)\ominus (a_2,g_2)$, which holds
if and only either $g_1>g_2$ and $\zero\surpass a_1$,
 or $g_2>g_1$ and $\zero\surpass a_2$, or 
$g_1=g_2\in \vgroup$ and $a_1\balance a_2$, or $(a_1,g_1)=(a_2,g_2)=(\zero,\botelt)$.

Let us show \eqref{prop-exten-1}.
If $\semiring$ satisfies tangible  balancing, then for all
$a_1,a_2\in \semiring$, $a_1\balance a_2$
implies that there exists $a\in \tangible\cup\{\zero\}$ such that
$a_1\balance a$ and $a_2\balance a$, and this applies to the particular
case $a_1=a_2$. 
Assume that  $(a_1,g_1)\balance (a_2,g_2)$, and let us show that there
exists $(a,g)\in \extension^\vee$ such that
 $(a_1,g_1)\balance (a,g)$ and $(a,g)\balance (a_2,g_2)$.
This is trivial for the case $(a_1,g_1)=(a_2,g_2)=(\zero,\botelt)$.
Assume first that $g_1=g_2\in \vgroup$, and let $a$ be as above,
that is  $a\in \tangible\cup\{\zero\}$, $a_1\balance a$ and $a_2\balance a$.
If $a\neq \zero$, then $a\in \tangible$, so $(a,g_1)\in\extension^\vee$,
and $(a_1,g_1)\balance (a,g_1)$ and $(a,g_1)\balance (a_2,g_2)$.
Otherwise, $a=\zero$, then $a_1\balance \zero$ and $a_2\balance \zero$,
so   $(a_1,g_1)\balance (\zero,\botelt)$ and
$ (\zero,\botelt)\balance (a_2,g_2)$, with $(\zero,\botelt)\in\extension^\vee$.
Assume now that $g_1>g_2$, then  $\zero\surpass a_1$,
and there exists $a\in\tangible\cup\{\zero\}$ such that $a\balance a_2$.
If $a\neq \zero$, then $(a,g_2)\in\extension^\vee$ and we have
 $(a_1,g_1)\balance (a,g_2)$ and $(a,g_2)\balance (a_2,g_2)$.
Otherwise, we have $(a_1,g_1)\balance (\zero,\botelt)$ and
$ (\zero,\botelt)\balance (a_2,g_2)$, with $(\zero,\botelt)\in\extension^\vee$. 
The same holds for the symmetrical case $g_2>g_1$.
Then, we have shown that $\extension$  satisfies tangible  balancing.

The proofs of \eqref{prop-exten-0}  \eqref{prop-exten-2} are similar.

Let us show \eqref{prop-exten-3}.
If $\semiring$ satisfies balance cancellation, then
for any $a\in\semiringvee\setminus\{\zero\}$ and $a_1\in\semiring$ such that $a\odot a_1 \balance \zero$, we have $a_1\balance \zero$ (in $\semiring$).
Consider $(a,g)\in \extension^\vee\setminus\{\zero\}$ and $(a_1,g_1)\in \extension$
such that $(a,g)\odot (a_1,g_1)\balance \zero$.
By definition this means that $a\in \tangible$ and 
$\zero\surpass (a,g)\odot (a_1,g_1)$.
Then, $\zero\surpass a\odot a_1$ or equivalently  
$a\odot a_1\balance \zero$. 
Since $\semiring$ satisfies balance cancellation, this implies that
 $a_1\balance \zero$, and by definition again, we get $(a_1,g_1)\balance \zero$.
\end{proof}

Let $(\semiring,\tangible,\ominus,\surpass)$, 
$\vgroup$ and $\extension:=\skewproductstar{\semiring'}{\tmax^*}$ be as
in \Cref{pro_preceq-ext}. 
We denote by $\ballext:= \{x\in \extension\mid |x|\leq \unit\}=
\{(a,g)\in \extension\mid g\leq 0\}$
the unit ball of $\extension$, 
$\ballext^\vee:=\ballext\cap\extension^\vee$ its intersection
with $\extension^\vee$, and define the map $\pbool: \ballext\to \semiring$,
such that
$\pbool((a,g))=a$ if $g=0$ and
$\pbool((a,g))=\zero$ if $g<0$. 
Note that one can compare the absolute value morphism to a valuation,
then $\ballext$, $\semiring$ and $\pbool$ can be compared 
respectively to the valuation ring,
the residue field and the residue map.
We shall also use the notation $\pbool$ entrywise or coefficientwise,
that is, for any formal polynomial $P$ with coefficients in $\ballext$,
 $\pbool (P)$ will be the formal
polynomial with coefficients equal to $\pbool (P_k)$.
We have the following property.

\begin{property}\label{prop-pbool}
With the above notations, 
$(\ballext,\ballext^\vee\setminus\{\zero\},\ominus, \surpass)$ is a subsemiring
system of $(\extension,\extension^\vee\setminus\{\zero\},\ominus,\surpass)$.
The map $\pbool$ is a  morphism of semiring systems
from $(\ballext,\ballext^\vee\setminus\{\zero\},\ominus, \surpass)$
to $(\semiring,\tangible,\ominus,\surpass)$, which is a projection.
\end{property}
In particular $a\balance b$ implies $\pbool(a)\balance \pbool(b)$,
for any $a,b\in\ballext$. We also have the following reverse implication.

\begin{property}\label{prop-pbool-balance}
With the above notations, 
assume that  $(\semiring,\tangible,\ominus,\surpass)$ satisfies 
weak tangible  balancing.
Then, 
for all $\tilde{a}\in \semiringvee$ and $b\in \ballext$, such that $\tilde{a}\balance \pbool(b)$, there exists $a\in \ballext^\vee$ such that $\pbool(a)=\tilde{a}$  and $a\balance b$. 
\end{property}
\begin{proof}
Using the identification of elements of $\semiring'$
as elements of $\extension$ with modulus $\unit$, by the map $a\mapsto (a,0)$,
we get that,
for all $b\in \ballext$, $\pbool(b)=b$ if $|b|=\unit$ and $\pbool(b)=\zero$ otherwise.

Let $\tilde{a}\in \semiringvee$  and $b\in \ballext$ be such that  $\tilde{a}\balance \pbool(b)$. 

If $|b|=\unit$, then $\pbool(b)=b$ and so $a=\tilde{a}$ satisfies trivially the conditions $\pbool(a)=\tilde{a}$  and $a\balance b$. 

If $|b|<\unit$, then  $\pbool(b)=\zero$, so $\tilde{a}\balance \zero$.
If $\tilde{a}\neq \zero$, then $|\tilde{a}|=\unit$, and so $a=\tilde{a}$ satisfies $\pbool(a)=a=\tilde{a}$ and $a\ominus b= a \balance \zero$, so
$a\balance b$.
If  $\tilde{a}=\zero$, then 
we need to find $a\in  \ballext^\vee$ such that $\pbool(a)=\zero$
and $a\balance b$. This holds for $a=\zero$, if $b\balance \zero$.
Otherwise, $b=(b',g)$ for some $b'\in \semiring$ such that $b'\notbalance \zero$ and 
$g\in \Gamma$ and $g<0$. By weak tangible balancing, there exists 
$a'\in \semiringvee$ such that $b'\balance a'$.
Since $b'\notbalance \zero$, we get that $a'\neq \zero$, so
$(a',g)\in\extension^\vee$. Since $g<0$, we also have $(a',g)\in \ballext$
and  $\pbool((a',g))=\zero$. So $a=(a',g)$ satisfies all the properties.
\end{proof}

\begin{lemma}\label{lemma-pbool}
For all $P\in\extension^\vee[\Y]$ and $r\in\extension^\vee$ such that  $|r|=\unit$. If $|\widehat{P}(r)|=\unit$, then we have $P=\Pn^\unit$,
$P\in \ballext^\vee[\Y]$, and $P^{\sat,\unit}=\pbool(P)\neq \zero$. Moreover,  $\Pn^r=P(r\Y)\in \ballext^\vee[\Y]$,
 and $P^{\sat,r}=\pbool(\Pn^r)$.
Conversely, if $P\in \ballext^\vee[\Y]$ and $\pbool(P)\neq \zero$,
then $|\widehat{P}(r)|=\unit$, and all the other properties hold.
\end{lemma}
\begin{proof} Assume that $|\widehat{P}(r)|=\unit$.
Since the absolute value is a morphism, we get 
$|\widehat{P}(r)|= \widehat{|P|}(|r|)$, for all $r\in \extension$,
hence $|\widehat{P}(\unit)|=\unit$.
By definition of  $\Pn^\unit$, we have $\Pn^\unit=P$.
Similarly, $\Pn^r=P(r\Y)$.
Since the addition in $\tmax$ is
the maximization, we get that $|\widehat{P}(\unit)|=\max_{k\geq 0} |P_k|$ so
$P_k\in \ballext$ for all $k\geq 0$,
so $P\in\ballext^\vee[\Y]$. Then, $\sat(\unit,P)=\sat(r,P)$
is the set of indices 
$k$ such that $|P_k|= \unit$. So for all $k\geq 0$, we have,
by  definition, $P^{\sat,\unit}_k=\zero$
if $|P_k|\neq \unit$ and  $P^{\sat,\unit}_k=P_k$ if  $|P_k|= \unit$,
that is $P^{\sat,\unit}_k=\pbool(P_k)$.
Similarly, $P^{\sat,r}=P^{\sat,\unit}(r\Y)$,  
$\Pn^r_k=P_k\odot r^{\odot k}\in \ballext$ and
$P^{\sat,r}_k=\pbool(\Pn^r_k)$, for all $k\geq 0$.

Conversely, if we assume only that $|r|=\unit$ and $P\in \ballext^\vee[\Y]$, 
then $|\widehat{P}(r)|\leq\unit$, and $\pbool(r)=r$.
By the morphism properties of the modulus and of $\pbool$ map, 
we get that $|\widehat{\pbool(P)}(r)|=|\pbool(\widehat{P}(r))|$.
So if $\pbool(P)\neq \zero$, $\widehat{\pbool(P)}(r)\neq \zero$ and
so has modulus $\unit$, which implies that 
$|\pbool(\widehat{P}(r))|=\unit$ and so $|\widehat{P}(r)|=\unit$.
\end{proof}

\begin{lemma}\label{degQ}
Let $P,Q\in\extension^\vee[\Y]$ and $r\in\extension^\vee$ such that $|r|=\unit$ and $|\widehat{P}(r)|=\unit$.
If $P\balance Q(\Y\ominus r)$,  then $|\widehat{Q}(r)|=\unit$.
\end{lemma}
\begin{proof}
Recall that the condition $P\balance Q(\Y\ominus r)$ is equivalent to
$P_{k} \balance Q_{k-1} \ominus r\odot Q_{k}$ for all $k\geq 1$ and
$P_{0} \balance \ominus  r \odot Q_{0}$, see \eqref{eq-PQ1}. Let $n=\deg(P)$.
Since $|\widehat{P}(r)|=\unit$, then by \Cref{lemma-pbool}, we get that
$|P_k|\leq \unit$ for all $k=0,\ldots, n$.
Since $P_0 \balance \ominus r \odot Q_0$, we have $P_0 = \ominus r \odot Q_0$,
by unique negation. So  $|Q_0| =|P_0|\leq \unit$.
Also, by \Cref{lem1}, we have $\deg(Q)=n-1$, and 
$P_n=Q_{n-1}$. So, %
$|Q_{n-1}| =|P_n|\leq \unit$.

For all $k=0, \ldots, n-2$, we have
$P_{k+1} \balance Q_{k} \ominus r\odot Q_{k+1}$, 
and since $P_{k+1}\in \extension^\vee$, using \Cref{pro_preceq-ext3} of
\Cref{pro_preceq-ext}, 
we obtain that $|P_{k+1}|\leq | Q_{k} \ominus r\odot Q_{k+1}|=|Q_k| \oplus |Q_{k+1}|$.
Together with $|P_0| =|Q_0|$ and $|P_n|=|Q_{n-1}|$, this leads to 
\begin{equation}\label{firstineq}
 \bigoplus_{k=0}^{n}|P_k| \leq \bigoplus_{k=0}^{n-1}|Q_k| \enspace.\end{equation}

Similarly, for all $k=0, \ldots, n-2$, we have
$Q_{k} \balance P_{k+1} \oplus  r\odot Q_{k+1}$,
 and so
$| Q_{k} |\leq |P_{k+1} \oplus r\odot Q_{k+1}|=|P_{k+1} | \oplus |Q_{k+1}|$.
By induction, we obtain 
that, for all $k=0, \ldots, n-1$, $| Q_{k} |\leq 
\bigoplus_{\ell=k+1}^{n}|P_k|\leq \unit$.

Since $\bigoplus_{k=0}^{n} |P_k|=|\widehat{P}(r)|=\unit$, using \eqref{firstineq}, we deduce:
\[ \unit=\bigoplus_{k=0}^{n}|P_k| \leq \bigoplus_{k=0}^{n-1}|Q_k| 
\leq \unit \enspace,\]
which implies the equality.
Then, $|\widehat{Q}(r)|=\widehat{|Q|}(|r|)=\bigoplus_{k=0}^{n-1}|Q_k|=\unit$.
\end{proof}

\begin{corollary}\label{cor-res}
Let $P\in\ballext^\vee[\Y]$, $Q\in \extension^\vee[\Y]$ and $r\in\extension^\vee$ such that  $|r|=\unit$, $\pbool(P)\neq \zero$, and $P \balance (\Y \ominus r)Q$.
Then, $Q\in\ballext^\vee[\Y]$, $\pbool(Q)\neq \zero$ and $\pbool(P) \balance (\Y \ominus r)\pbool(Q)$.
\end{corollary}
\begin{proof}
By \Cref{lemma-pbool}, we get that $|\widehat{P}(r)|=\unit$.
Then, by \Cref{degQ}, $|\widehat{Q}(r)|=\unit$, and
applying \Cref{lemma-pbool} again, we get that $Q\in \ballext^\vee[\Y]$.
Using \Cref{prop-pbool} and that $\pbool(r)=r$,
we deduce that  $\pbool(P) \balance (\Y \ominus r)\pbool(Q)$.
\end{proof}

The following result is showing a reverse implication.
\begin{lemma}\label{lem-existres}
Let  $P\in \ballext^\vee[\Y]$ be of degree $n$, $r\in\extension^\vee$ such that $|r|=\unit$, and $\widetilde{Q}\in\semiringvee[\Y]$ be of degree $\leq n-1$, and such that
 \begin{equation}\label{Q_property}\pbool (P) \balance (\Y \ominus r) \widetilde{Q}.\end{equation} 
 Then there exists $Q\in \ballext^\vee[\Y]$ of degree $n-1$,
such that $P\balance (\Y \ominus r) {Q}$ 
and $\pbool (Q)=\widetilde{Q}$.
 \end{lemma}
 \begin{proof}
Let $P$ and $\widetilde{Q}$ be as in the lemma.
\Cref{Q_property} means that 
\begin{equation}\label{assump-tilde}
 \widetilde{Q}_{k-1}\balance \pbool( P_k ) \oplus r\odot \widetilde{Q}_k
\quad \forall\;  0\leq k\leq n\enspace .\end{equation}
Let us construct $Q$ satisfying the conditions of the lemma,
which means that $Q\in \ballext^\vee[\Y]$, that 
it satisfies $Q_n=\zero$ and that, for all $k\leq n$, we have  
\begin{subequations}\label{induc-deftilde}
\begin{align}
& Q_{k-1}\balance P_k \oplus r \odot Q_k\label{induc-deftilde1}\end{align}
and 
\begin{align}
&\pbool(Q_{k-1})=\widetilde{Q}_{k-1}\enspace .\label{induc-deftilde2}
\end{align}\end{subequations}

We shall construct $Q_{k-1}\in \ballext^\vee$ satisfying \eqref{induc-deftilde} for $k$, by descending induction on $k=n,\ldots , 1$, assuming $Q_n=\zero$.

For $k=n$, %
we take $ Q_{n-1} = P_n$, then $Q_{n-1}\in \ballext^\vee$,
 and since $Q_{n}=\zero$, 
\eqref{induc-deftilde1}  holds for $k$.
Moreover, since $\deg(\widetilde{Q}) %
\leq n-1$,
we get that  $\widetilde{Q}_{n}=0$, and using  \Cref{assump-tilde},
we obtain that $\widetilde{Q}_{n-1} =  \pbool( P_n)$. %
Then, $\pbool({Q}_{n-1}) =\pbool( P_n)=\widetilde{Q}_{n-1}$,
which gives \eqref{induc-deftilde2} for $k=n$.

Assume now that \eqref{induc-deftilde} is true for $k+1\leq n$, 
and let us show that there exists $Q_{k-1}\in \ballext^\vee$
 satisfying \eqref{induc-deftilde} for $k$.
Using  \eqref{induc-deftilde2} for $k+1$, we get
$\pbool(Q_{k})=\widetilde{Q}_{k}$, so,
using the morphism property of $\pbool$,  and that $|r|=\unit$, we get that
the right hand side of  \eqref{assump-tilde} satisfies 
$\pbool( P_k ) \oplus r\odot \widetilde{Q}_k=\pbool( P_k \oplus r\odot \widetilde{Q}_k)$.
Applying \Cref{prop-pbool-balance} to $\tilde{a}= \widetilde{Q}_{k-1}$,
that is the left hand side of \eqref{assump-tilde}, and to
$b= P_k \oplus r\odot \widetilde{Q}_k$, 
we deduce that there exists $a\in \ballext^\vee$ such that
$a\balance b$ and $\pbool(a)=\tilde{a}$. Taking $Q_{k-1}=a$, we get that 
$Q_{k-1}\in \ballext^\vee$ and satisfies 
 \eqref{induc-deftilde} for $k$.

This shows the existence of $Q\in \ballext^\vee[\Y]$,
satisfying $Q_n=\zero$ and \eqref{induc-deftilde} for all $1\leq k\leq n$,
which finishes the proof of the lemma.
\end{proof}
\subsection{Proof of \Cref{tho-mult-smax}}\label{prooftheo-mult}
We are now able to prove \Cref{tho-mult-smax}.

Let $r=\zero$. 
If $P,Q\in\extension^\vee[\Y]$ and $P \balance  Q \Y$, then $P_{i+1}\balance Q_i$ 
for $i\geq 0$, and $P_0\balance \zero$.
By unique negation, this implies the equalities,
and so $P=Q\Y$, and $Q$ is unique. Since, by \Cref{mult} we have 
$\mult_{\zero}(P)=1+\mult_{\zero}(Q)$, we deduce that $\mult_{\zero}(P)$
is the lower degree of $P$.

\noindent
{\bf Proof of \Cref{tho-mult-smax1} and \Cref{tho-mult-smax2}:}
 Now, let $r\in\extension^\vee\setminus\{\zero\}$, then $|r|\neq \zero$, so
$|\widehat{P}(r)|=\widehat{|P|}(|r|)\neq \zero$ (since $P$ is not the zero polynomial). Therefore $|r|$ and $|\widehat{P}(r)|$ are invertible in $\extension^\vee$. So we can consider the normalization $\bar{r}=|r|^{\odot -1}\odot r$ of $r$ as in  \Cref{tho-mult-smax}. Since $|\bar{r}|=\unit$, and $\bar{r}\in \extension^\vee$, it can be seen as an element of $\tangible$. This shows \Cref{tho-mult-smax1} and \Cref{tho-mult-smax2} of \Cref{tho-mult-smax}.

\noindent
{\bf Proof of \Cref{tho-mult-smax3}:} By definition of the saturation polynomial, $\sat(r,P)$ only depends on $|r|$ and of the formal polynomial $|P|$. Moreover, given any $\gamma,\mu\in\extension^\vee$ such that $r=\gamma\odot \mu$,  $\sat(r,P)$ coincides with $\sat(\gamma, \Pn^{\mu})$, so in particular with $\sat(\unit, \Pn^r)$ and 
$\sat(\bar{r}, \Pn^{|r|})$, which shows  \Cref{tho-mult-smax3} of \Cref{tho-mult-smax}.

\noindent
{\bf Proof of \Cref{tho-mult-smax4}:} By definition of $\sat(r,P)$, the coefficients of the saturation polynomial $P^{\sat,r}$ are either $\zero$ or with modulus equal to $\unit$. Since they are also in $\extension^\vee$, they are identified with elements of $\semiringvee$, which shows \Cref{tho-mult-smax4} of \Cref{tho-mult-smax}.

\noindent
{\bf Proof of \Cref{tho-mult-smax5} and \Cref{tho-mult-smax6}:} To show \Cref{tho-mult-smax5} and \Cref{tho-mult-smax6} of \Cref{tho-mult-smax}, we consider $\gamma,\mu\in\extension^\vee$ such that $r=\gamma\odot \mu$ with $|\gamma|=\unit$ and $\mu$ invertible,
and shall show that 
\begin{equation}\label{equality_mult}
\mult_{r}(P)=\mult_{\gamma}(\Pn^{\mu})=\mult_{\gamma}(P^{\sat,\mu})\enspace.
\end{equation}
Indeed, applying \Cref{equality_mult} to the case $\gamma=\bar{r}$ and $\mu=|r|$ will give  \Cref{tho-mult-smax5}, and applying it to the case $\gamma=\unit$ and $\mu=r$ will give  \Cref{tho-mult-smax6}.

Since $|r|=|\mu|$, we have  $|\widehat{P}(r)|=|\widehat{P}(\mu)|$,
so $\Pn^\mu= |\widehat{P}(r)|^{\odot - 1} P(\mu\Y)$ is well defined.
Moreover, since $\mu$ is invertible, then 
the multiplicity of $r$ as a root of $P$, $\mult_r(P)$, is the same as
the multiplicity of $\gamma$ as a root of $\Pn^\mu$, $\mult_\gamma(\Pn^\mu)$,
that is the first equality in \eqref{equality_mult}.
Indeed, if $\lambda P\balance Q(\Y\ominus r)$ for some $\lambda\in \extension^\vee\setminus\{\zero\}$ and $Q\in \extension^\vee[\Y]$, then
$\lambda \Pn^\mu \balance Q'(\Y\ominus \gamma)$ with
$Q'=  |\widehat{P}(r)|^{\odot - 1} \mu Q(\mu\Y)$;
and the converse is also true since $\mu$ is invertible.

We have $|\gamma|=\unit$ and $|\widehat{\Pn^\mu}(\gamma)|=|\widehat{P}(r)|^{\odot - 1} |\widehat{P}(\mu\odot \gamma)|=\unit$.
Moreover, $P^{\sat,\mu}$ only depends  on $\Pn^\mu$, and is 
equal to $(\Pn^\mu)^{\sat,\unit}$. Therefore, the second equality 
in \eqref{equality_mult} will follow from the following property
\begin{equation}\label{multsatmod1old}
\mult_{r}(P)=\mult_{r}(P^{\sat,\unit})\quad\text{for}\; r\in\extension^\vee,\; |r|=\unit,\; P\in \extension^\vee[\Y],\; |\widehat{P}(r)|=\unit\enspace .
\end{equation}
Moreover, from \Cref{lemma-pbool}, we get that 
\eqref{multsatmod1old} is equivalent to %
\begin{equation}\label{multsatmod1}
\mult_{r}(P)=\mult_{r}(\pbool(P))\quad\text{for}\; r\in\extension^\vee,\;|r|=\unit,\; P\in \ballext^\vee[\Y],\; \pbool(P)\neq \zero\enspace .
\end{equation}

Let us first note that if $P$ is as in \eqref{multsatmod1}, then 
$\pbool(P)\in\semiringvee[\Y]\setminus\{\zero\}$
and conversely if $P\in\semiringvee[\Y]\setminus\{\zero\}$, then $P=\pbool(P)$
and $P$ satisfies the conditions of \eqref{multsatmod1}.

Let us fix $r$ as in \eqref{multsatmod1}, that is 
$r\in\extension^\vee$ with $|r|=\unit$, which is equivalent to
the condition that $r\in \semiringvee\setminus\{\zero\}=\tangible$.
We shall show 
 \eqref{multsatmod1} %
for all formal polynomials of degree $n\geq 0$, by induction on $n$.

If $n=0$, this property are trivial.
Now assume by induction that the property holds
for all formal polynomials of degree $\leq n-1$.
Let $P\in \ballext^\vee[\Y]$ be with degree $n$ and such that $\pbool(P)\neq \zero$, or equivalently $|\widehat{P}(r)|=\unit$,  by \Cref{lemma-pbool}.

Let us first show that $r$ is a root of $P$ if, and only if, it is a root of $\pbool(P)$. Indeed, since $|\widehat{P}(r)|=\unit$, and $|r|=\unit$, we get that
$\widehat{P}(r)=\pbool(\widehat{P}(r))$ and $r=\pbool(r)$.
Since $\pbool$ is a morphism,
we deduce that $\pbool(\widehat{P}(r))=\widehat{\pbool(P)}(\pbool(r))=
\widehat{\pbool(P)}(r)$ and so $\widehat{P}(r)=\widehat{\pbool(P)}(r)$,
which implies that $r$ is a root of $P$ if and only if it is a root of $\pbool(P)$.

Assume first that $r$ is not a root of $P$, then it is not a root of 
 $\pbool(P)$,
and so $\mult_r(P)=\mult_r(\pbool(P))=0$, which shows \eqref{multsatmod1}.

Now assume that $r$ is a root of $P$, and so it is also a root of
$\pbool(P)$. By definition of multiplicites, there exist 
$\lambda\in \extension^\vee\setminus\{\zero\}$ and 
$Q\in \extension^\vee[\Y]$, such that $\lambda P \balance (\Y \ominus r)Q$
and $\mult_r(P)=1+\mult_r(Q)$.
Since $\lambda\neq \zero$, dividing the balance equation by $|\lambda|$, we get
$\bar{\lambda} P \balance (\Y \ominus r) Q'$
for $\bar{\lambda}=|\lambda|^{\odot -1}\odot \lambda$ and
$Q'= |\lambda|^{\odot -1} Q$.
Since $|\lambda|$ is invertible, the multiplicity of $Q'$ is the same as the one of $Q$, and so we can assume from the begining 
that $|\lambda|=\unit$.
In that case, $\lambda\in \semiringvee\setminus \{\zero\}$ and 
$\pbool(\lambda)=\lambda$.
Then, $\lambda P\in \ballext^\vee[\Y]$ and $\pbool(\lambda P)=\lambda \pbool(P)\neq \zero$.
Hence, by \Cref{lem1}, 
$\deg(Q)=n-1$ and by \Cref{cor-res}, $Q\in\ballext^\vee[\Y]$,
$\pbool(Q)\neq \zero$ 
and $\lambda \pbool(P) \balance (\Y \ominus r)\pbool(Q)$.
Therefore, $\mult_r(\pbool(P))\geq 1+\mult_r(\pbool(Q))$, by definition
of multiplicities.
Using the inductive hypothesis, we obtain
 $\mult_{r}(Q)= \mult_{r}(\pbool(Q))$
and thus  $\mult_r(\pbool(P))\geq 1+\mult_r(Q)=\mult_r(P)$.
Let us show the reverse inequality, that is $\mult_{r} (\pbool(P))\enspace\leq  \mult_r(P)$.
By definition of multiplicities, there exist $\lambda \in \extension^\vee\setminus\{\zero\}$ and $\widetilde{Q}\in \extension^\vee[\Y]$
such that $\lambda \pbool(P) \balance (\Y \ominus r) \widetilde{Q}$,
and $\mult_r(\pbool(P))=1+\mult_r(\widetilde{Q})$.
Dividing by $|\lambda|$ as above, we are reduced to the case 
in which $|\lambda|=\unit$, and so $\lambda \in \semiringvee\setminus\{\zero\}$
and $\pbool(\lambda P)=\lambda \pbool(P)\neq \zero$.
Applying \Cref{cor-res}, we obtain that $\widetilde{Q}\in\ballext^\vee[\Y]$, $\pbool(\widetilde{Q})\neq \zero$ and $\pbool(\lambda P) \balance (\Y \ominus r)\pbool(\widetilde{Q})$.
By \Cref{lem1}, $\deg(\pbool(\widetilde{Q}))=\deg(\pbool(\lambda P))-1\leq n-1$.
Then, applying \Cref{lem-existres} to $\lambda P$ and $\pbool(\widetilde{Q})$ (instead of $P$ and $\widetilde{Q}$), we obtain $Q\in \ballext^\vee[\Y]$ of degree $n-1$,
such that $\lambda P\balance (\Y \ominus r) {Q}$ 
and $\pbool (Q)=\pbool(\widetilde{Q})$.
Now applying the induction assumption, we get that 
$\mult_r(Q)=\mult_r(\pbool (Q))=\mult_r(\pbool(\widetilde{Q}))=\mult_r(\widetilde{Q})$. For the last equality, we used that $\deg(\widetilde{Q})=\deg(\pbool(\lambda P))-1\leq n-1$ by \Cref{lem1}.
This implies $\mult_r(\pbool(P))=1+\mult_r(Q)$.
Since $\lambda P\balance (\Y \ominus r) {Q}$ and $\lambda \in \extension^\vee\setminus\{\zero\}$, we get by definition of multiplicities that 
$\mult_r(P)\geq 1+\mult_r(Q)=\mult_r(\pbool(P))$, which 
shows the reverse inequality and thus the equality in
\eqref{multsatmod1}.  This finishes the proof of the induction,
and thus of  \Cref{tho-mult-smax5} and \Cref{tho-mult-smax6}.

{\bf Proof of \Cref{tho-mult-smax4bis}:}
Let $r$ be as in \Cref{tho-mult-smax4bis}, that is
$r\in\extension^\vee$ and $|r|=\unit$. Then, $r$ can be seen equivalently
as an element of $\semiringvee\setminus\{\zero\}=\tangible$.
Any formal polynomial $P$ with coefficients in $\semiringvee$
can be seen also as a formal polynomial with coefficients in $\extension^\vee$.
Let us denote by $\mult_{r,\semiring}(P)$ the multiplicity of
$r$ as a root of $P$ computed in $\semiring$, and keep the notation
$\mult_{r}(P)$ for the multiplicity of $r$ as a root of $P$ computed in 
$\extension$. 
Then \Cref{tho-mult-smax4bis} of \Cref{tho-mult-smax} can be rewritten as: 
\begin{equation}\label{eqmult}
\mult_{r}(P)=\mult_{r,\semiring}(P)\quad\text{for} \; 
r \in \semiringvee\setminus\{\zero\}, \; P\in \semiringvee[\Y]\setminus\{\zero\}
\enspace .
\end{equation}
We shall prove \eqref{eqmult} 
for all formal polynomials of degree $n\geq 0$, by induction on $n$.

If $n=0$, this property is trivial.
Now assume by induction that the property holds
for all formal polynomials of degree $\leq n-1$.
Let $P\in \semiringvee[\Y]$ with degree $n$.

For $r\in \semiringvee\setminus\{\zero\}$, being a root of $P$ in $\semiring$ or
in $\extension$ is the same. So,
when $r$ is not a root of $P$ (in $\semiring$ or in $\extension$),
we have  $\mult_r(P)=\mult_{r,\semiring}(P)$.

Assume now that $r$ is a root of $P$ (in $\semiring$ or in $\extension$).
By definition of multiplicites, there exist 
$\lambda\in \semiringvee\setminus\{\zero\}$ and 
$Q\in \semiringvee[\Y]$, such that $\lambda P \balance (\Y \ominus r)Q$
and $\mult_{r,\semiring}(P)=1+\mult_{r,\semiring}(Q)$.
By \Cref{lem1}, $\deg(Q)=n-1$, so by the induction assumption, 
we have $\mult_{r,\semiring}(Q)=\mult_r(Q)$.
Using the definition of multiplicities and that $\semiringvee\subset \extension^\vee$, we obtain $\mult_{r,\semiring}(P)\leq \mult_r(P)$.

To show the reverse inequality, consider 
$\lambda\in \extension^\vee\setminus\{\zero\}$ and 
$Q\in \extension^\vee[\Y]$, such that $\lambda P \balance (\Y \ominus r)Q$
and $\mult_r(P)=1+\mult_r(Q)$.
Using the same arguments as in the proof of \eqref{multsatmod1}, we can reduce us to
the case where $|\lambda|=\unit$, and so $\lambda\in\semiringvee\setminus \{\zero\}$. Then, $\lambda P\in \semiringvee[\Y]\setminus\{\zero\}$.
Moreover, $Q\in\ballext^\vee[\Y]$, $\pbool(Q)\neq \zero$ 
and $\lambda P=\lambda \pbool(P) \balance (\Y \ominus r)\pbool(Q)$.
Now, applying \eqref{multsatmod1} (which we already proved)  to $Q$,
we obtain that  $\mult_r(\pbool(Q))=\mult_r(Q)$, and
so $\mult_r(P)=1+\mult_r(\pbool(Q))$ with $\pbool(Q)\in \semiringvee[\Y]\setminus\{\zero\}$.
This means that we can reduce us to the case where $Q=\pbool(Q)\in \semiringvee[\Y]\setminus\{\zero\}$.
Since now $\lambda\in\semiringvee\setminus \{\zero\}$,
$Q\in \semiringvee[\Y]\setminus\{\zero\}$, and
$\lambda P \balance (\Y \ominus r)Q$,
by definition of multiplicities, we get that 
$1+\mult_{r,\semiring}(Q)\leq \mult_{r,\semiring}(P)$.
Since $\lambda P$ has degree $n$, we get by \Cref{lem1}, that
$Q$ has degree $n-1$. Hence, using
the induction hypothesis, we obtain $\mult_{r,\semiring}(Q)=\mult_r(Q)$,
and so $\mult_r(P)=1+\mult_r(Q)\leq \mult_{r,\semiring}(P)$.
This finishes the proof of the equality in \eqref{eqmult}, and so of 
\Cref{tho-mult-smax4bis}.
\qed 

\section{From multiplicities to factorization in $\smax$}
\label{sec-multtact}
{\em In this section, unless otherwise stated, 
we shall assume that $\vgroup$
is a divisible ordered group.}
Recall that when $\vgroup$ is trivial, $\smax$ is equal to $\bmaxs$,
and is thus 
equivalent to the hyperfield of signs studied in \cite{baker2018descartes}.

\begin{proposition}[Bounds on multiplicities of signed roots of polynomials]\label{atmostn}
Let ${P} \in \smax^\vee[\Y]$ of degree $n$.
Then, for any root $r\in\smax^\vee$ of $P$, its multiplicity
(in the sense of \Cref{def-mult-BL}) is bounded above by the multiplicity 
$\mult_{|r|}(|P|)$ of $|r|$ as a corner of $|P|$.
More precisely, we have $\mult_{r}(P)+\mult_{\ominus r}(P)\leq \mult_{|r|}(|P|)$.
Therefore, $P$ has at most $n$ roots counted with multiplicity.
\end{proposition}

The proof uses the following result, which was stated in 
\cite[Remark 1.13]{baker2018descartes} (and follows from 
\cite[Th.\ C]{baker2018descartes}, see also \Cref{baker31}).
\begin{lemma}[\protect{\cite[Remark 1.13]{baker2018descartes}}]\label{lemma-bmax}
Let $Q\in \bmaxs^\vee[\Y]$ of degree $n$.
Then, $Q$ has at most $n$ roots in $\bmaxs$ counted with multiplicity, that is
$\mult_{\zero}(Q)+\mult_{\unit}(Q)+\mult_{\ominus \unit}(Q)\leq n$.
\end{lemma}

\begin{proof}[Proof of \Cref{atmostn}]
The result follows from \Cref{lemma-bmax}, when $\vgroup$ is trivial.
Assume now that $\vgroup$ is non-trivial.
Let $P=\bigoplus_{i=0}^n P_i \Y^i\in \smax^\vee[\Y]$, be a non-zero formal 
polynomial, and let  $n\geq 0$ be its degree.
Then, by \Cref{abs_roots}, any root $r\in\smax^\vee$ of $P$ is such that
$|r|$ is a corner of $|P|$.

If $r=\zero$, then, by \Cref{cor-mult-smax},
 the multiplicity $\mult_{r}(P)$ of $r$ 
as a root of $P$ is equal to the lower degree $\uval(P)$ of $P$,
which is also the  multiplicity of
$\zero$ as a corner of $|P|$.

If $r\neq \zero$, then again by \Cref{cor-mult-smax},
$\mult_r(P)$ is equal to $\mult_{\unit}(P^{\sat,r})$.
Similarly, $\mult_{\ominus r}(P)$ is equal to $\mult_{\unit}(P^{\sat,\ominus r})$.
Since $P^{\sat,\ominus r}=P^{\sat,r}(\ominus \Y)$,
we get that $\mult_{\ominus r}(P)$ is equal to $\mult_{\ominus \unit}(P^{\sat,r})$.
Using \Cref{lemma-bmax}, and the fact that the multiplicity of 
$\zero$ as a root of a polynomial over $\smax$ or $\bmaxs$
 is equal to the lower degree of this polynomial, we deduce that 
$\mult_{\unit}(P^{\sat,r})+ \mult_{\ominus \unit}(P^{\sat,r})$ is bounded
above by $\deg(P^{\sat,r})-\uval(P^{\sat,r})$.
One can also check, from the definition of $P^{\sat,r}$,
 that $\deg(P^{\sat,r})-\uval(P^{\sat,r})$
is equal to the multiplicity of $|r|$ as a corner of $|P|$.
This shows the inequality $\mult_{r}(P)+\mult_{\ominus r}(P)\leq \mult_{|r|}(|P|)$.

Summing the inequalities for all $r\neq \zero$ and 
the equality $\mult_{\zero}(P)=\mult_{\zero}(|P|)$ for $r=\zero$, we 
obtain that the sum of the multiplicities of the roots of $P$ is
upper bounded by the sum of multiplicities of the corners of $|P|$, that is $n$.
\end{proof}

Recall that, the factorization of a polynomial function
$\widehat{P}$ over $\smax$ may not be unique, so a factorization
does not determine the multiplicities of the roots. 
Moreover, a factorization does not imply
that there are $n$ roots counted with multiplicities.
However, the following result shows the converse implication.

\begin{theorem}[Another sufficient condition for factorization]
\label{prop-mult-fact} 
Let ${P} \in \smax^\vee[\Y]$ of degree $n$, and
having $n$ roots $r_1,\ldots , r_n$, counted with multiplicity.
Then $\widehat{P}$ is factored as 
\begin{equation} \label{factorization-eq}
 \widehat{P}(y)=P_n\odot (y\ominus r_1)\odot \cdot \odot (y\ominus r_n) \enspace .\end{equation}
\end{theorem}
To prove this result, let us show the following result for
 polynomials over $\bmaxs$.
\begin{lemma}\label{lemma-max-bmax}
Let $Q\in \bmaxs^\vee[\Y]$ of degree $n$, and lower degree $m$,
and having $n$ roots in $\bmaxs$ 
counted with multiplicities, that is
$\mult_{\unit}(Q)+\mult_{\ominus \unit}(Q)=n-m$.
Then, for all successive degrees $k<\ell$ of non-zero monomials of $Q$
(that is $m\leq k<\ell\leq n$  and $Q_i=\zero$ for
$k<i<\ell$), we have either $\ell-k=1$, or $\ell-k=2$ and $Q_\ell=\ominus Q_k$.
Conversely, if the latter property holds, then $Q$ has  $n$ roots in $\bmaxs$ 
counted with multiplicities.

Let us consider the polynomial $R$ with same
lower degree and degree as $Q$ and such that, for all
$m\leq i\leq n$,  $R_i=Q_i$ if $Q_i\neq \zero$ and $R_i=\unit$ otherwise.
We have that $|R|$ is factored (see \Cref{roots_poly}),
$\widehat{R}=\widehat{Q}$, and $R$ has same multiplicities of roots as $Q$:
 $\mult_{r}(R)=\mult_{r}(Q)$, for all $r\in\bmaxs^\vee$.
Moreover, these multiplicities are the same as the ones in
the sequence $r_1,\ldots, r_n$ obtained in \Cref{suf_cond} for the
polynomial $R$.
In particular $\widehat{Q}$ is factored as in 
\begin{equation} \label{factorization-eq-lem}
 \widehat{Q}(y)=Q_n\odot (y\ominus r_1)\odot \cdot \odot (y\ominus r_n) \enspace .\end{equation}
\end{lemma}
\begin{proof}
Let $Q\in \bmaxs^\vee[\Y]$ of degree $n$, and lower degree $m$,
and having $n$ roots in $\bmaxs$ 
counted with multiplicities, that is
$\mult_{\unit}(Q)+\mult_{\ominus \unit}(Q)=n-m$. 
For any two degrees $k<\ell$ of monomials of $Q$, with $m\leq k<\ell\leq n$,
denote by $\sigma(k,\ell)$ the number of sign changes in the non-zero coefficients $Q_i$ of $Q$ with indices $i$ such that $k\leq i\leq \ell$ and $\sigma'(k,\ell)$ be the number of sign changes in the non-zero coefficients $Q_i(\ominus \unit)^i$ of $Q(\ominus \Y)$ such that $k\leq i\leq \ell$.
Let $k$ and $\ell$ be two successive degrees of non-zero monomials of $Q$ such that $m\leq k<\ell\leq n$. We have 
\[\sigma(k,l) + \sigma'(k,l)=
\begin{cases}
1& \text{if} \; \ell-k \; \text{is odd};\\
2&  \text{if} \;Q_k = \ominus Q_\ell \;\text{and}\; \ell-k \; \text{is even};\\
0& \text{if} \;Q_k = Q_\ell \;\text{and}\; \ell-k \; \text{is even}.
\end{cases}
\]
Then, $\sigma(k,\ell) + \sigma'(k,\ell) \leq \ell-k$ and we have equality if and only if either $\ell-k=1$, or $\ell-k=2$ and $Q_\ell=\ominus Q_k$.
Moreover, $\sigma$ and $\sigma'$ are additive, for instance
$\sigma(k,\ell)=\sigma(k,k')+\sigma(k',\ell)$ when $m\leq k<k'<\ell\leq n$.
Therefore, $\sigma(m,n) + \sigma'(m,n) \leq n-m$ and  we have the equality if and only if we have, for all two successive degrees of non-zero monomials of $Q$, $k<\ell$,
 either $\ell-k=1$, or $\ell-k=2$ and $Q_\ell=\ominus Q_k$.
This shows the first assertion of the lemma and its converse implication.

Now let $R \in \bmaxs^\vee[\Y]$ be as in the lemma.
Then the coefficients of $|R|$ with indices between $m$ and $n$ 
are all equal to $\unit$, so that $|R|$ is factored (the coefficient map is
constant so concave).
Also, if $R\neq Q$, then there are successive degrees of non-zero monomials
 of $Q$, $k<\ell$, such that $\ell-k=2$ and $Q_\ell=\ominus Q_k$,
which implies that both $\unit$ and $\ominus \unit$ are roots of $Q$:
$\widehat{Q}(\unit)=\widehat{Q}(\ominus\unit)=\unit^\circ$.
Since $\widehat{Q}(\unit)\leq \widehat{R}(\unit)\leq\widehat{Q}(\unit)\oplus \unit^\circ$, we get that $\widehat{Q}(\unit)= \widehat{R}(\unit)$.
Similarly $\widehat{Q}(y)= \widehat{R}(y)$, for all $y\in\bmaxs$ and
also for all $y\in\smax$.

Let us consider now the sequence $r_1,\ldots, r_n$ obtained in \Cref{suf_cond} for the polynomial $R$. %
Since $|R|$ is factored, %
we get that $\widehat{Q}=\widehat{R}$ is factored as
 in \eqref{factorization-eq-lem},
with the elements $r_1,\ldots, r_n$ of $\bmaxs^\vee$ defined in \Cref{suf_cond}.
Let us show that the multiplicities of the roots $r\in\bmaxs^\vee$ of $Q$ and $R$ are the same as the ones in the sequence $r_1,\ldots, r_n$.

By definition, $r_i=\zero$ for $i=n-m+1,\ldots, n$,
and $r_i \odot R_{n-i+1}=\ominus R_{n-i}$ for $i=1,\ldots, n-m$.
So the multiplicity of $r=\zero$ is the same for the polynomials $Q$ and $R$
and for the sequence  $r_1,\ldots, r_n$.

Let us consider now the roots $r=\unit$ and $\ominus \unit$.
For $i=1,\ldots, n-m$, we have
\[r_i=\begin{cases}
\unit & \text{if}\; R_{n-i+1}=\ominus R_{n-i};\\
\ominus \unit& \text{if} \; R_{n-i+1}= R_{n-i}.
\end{cases}
\]
So, the multiplicity of $r=\unit$ in the sequence $r_1,\ldots, r_n$
is equal to the number of sign changes in the non-zero coefficients of the
polynomial $R$, and so by \cite[Th.\ C]{baker2018descartes} (see \Cref{baker31}), it is equal to $\mult_{\unit}(R)$.
Similarly,  the multiplicity of $r=\ominus\unit$ in the sequence $r_1,\ldots, r_n$ is equal to the number of sign changes in the non-zero coefficients of the
polynomial $R(\ominus \Y)$, and so it is equal to $\mult_{\ominus \unit}(R)$.

Define $\sigma$ and $\sigma'$ as above for both $Q$ and $R$.
Given two successive degrees $k<\ell$ of non-zero monomials of $Q$, 
we have either $\ell-k=1$ or $\ell-k=2$ and $Q_{\ell}=\ominus Q_k$.
In the first case, the coefficients of $Q$ and $R$ with
indices between $k$ and $\ell$ are the same, so
$\sigma(k,l)$ is the same for $Q$ and $R$.
In the second case, we have  $R_{k+1}=\unit$, 
so $R_{k+1}$ is either equal to $Q_k$ or $Q_{\ell}$,
and $\sigma(k,\ell)=1$ for both $Q$ and $R$.
Similarly $\sigma'(k,\ell)$ are the same for both $Q$ and $R$.
Summing the $\sigma(k,\ell)$ for all successive degrees $k<\ell$,
 we get that $\sigma(m,n)$ and $\sigma'(m,n)$
are the same for $Q$ and $R$, which by \cite[Th.\ C]{baker2018descartes} (see \Cref{baker31}) implies
that the multiplicities of the roots $\unit$ and $\ominus \unit$ are the
same for $Q$ and $R$.
\end{proof}

\begin{proof}[Proof of \Cref{prop-mult-fact}]
We can assume that $\vgroup$ is non-trivial,
since otherwise the result follows from \Cref{lemma-max-bmax}.
Let ${P} \in \smax^\vee[\Y]$ of degree $n$, lower degree $m$, and
having $n$ roots $r_1,\ldots , r_n$, counted with multiplicity.

By \Cref{atmostn}, we have $\mult_{\zero}(P)=\mult_{\zero}(|P|)$,
and $\mult_{r}(P)+\mult_{\ominus r}(P)\leq \mult_{r}(|P|)$, for all $r\in \tmax$,
seen also as an element of $\smax^\oplus$.
Since, by assumption, the sum of all the multiplicities of roots of $P$ 
is equal to $n$,  which is also the sum of all the multiplicities of corners of $|P|$, we get the equality
$\mult_{c}(P)+\mult_{\ominus c}(P)=\mult_{c}(|P|)$, for all 
corners $c\neq \zeror $ of $|P|$.
By \Cref{cor-mult-smax}, for each corner $c=|r_i|$ of $|P|$, we have 
$\mult_{c}(P)=\mult_{\unit}(P^{\sat,c})$,
and $\mult_{\ominus c}(P)=\mult_{\unit}(P^{\sat,\ominus c})=\mult_{\ominus \unit}(P^{\sat,c})$ (since $P^{\sat,\ominus c}=P^{\sat,c}(\ominus \Y)$).
So $P^{\sat,c}$ satisfies the assumption of \Cref{lemma-max-bmax}.

Given a corner $c$ of $|P|$, let us add monomials in $P$ so that 
$P^{\sat,c}$ is changed as in \Cref{lemma-max-bmax}.
Then, the new polynomial $R$ is such that $R^{\sat,c}$ has same multiplicities
of roots in $\bmaxs^\vee$ as  $P^{\sat,c}$,
and so,  by \Cref{cor-mult-smax}, $R$ and $P$ have same multiplicities of
roots $r$ such that $|r|=c$.
Let $[k,\ell]$ be the convex hull of the saturation set $\sat(c,P)$,
then $k$ and $\ell$ are the lower degree and degree of $P^{\sat,c}$,
so the coefficients of $R$ and $P$ outside $[k,\ell]$ are the same.
This implies that,  for all the roots $r$ such that $|r|\neq c$,
the saturated polynomials of $R$ and $P$ are the same,
and, by \Cref{cor-mult-smax}, the multiplicities of the root $r$ are the same
for $R$ and $P$.
The polynomial $R$ also satisfies $\widehat{R^{\sat,c}}=\widehat{P^{\sat,c}}$,
which implies also by domination properties, that $\widehat{R}=\widehat{P}$.
Moreover, $|R^{\sat,c}|$  is factored,
 which means that the support of $|R^{\sat,c}|$  is full, and thus
the coefficient map of $|R|$ is concave in the 
interval $[k,\ell]$. 

Applying this transformation for all corners of $|P|$, we obtain
a polynomial $R$ such that $\widehat{R}=\widehat{P}$, $|R|$ is concave 
in all the interval $[m,n]$, and the multiplicities of the roots of 
$R$ and $P$  are the same.
Since  $|R|$ is concave, we get  that 
$|R|$ is factored (as a formal polynomial over $\tmax$), 
hence, by \Cref{suf_cond}, $\widehat{R}$ is factored as 
$\widehat{R}(y)=\tropprod_{1\leq i\leq n} (y\ominus r_i)$ with
$r_i$ such that 
$r_i\odot R_{n-i+1}= \ominus R_{n-i}$ for all $i\leq n-m$.
Using the last assertion in \Cref{lemma-max-bmax}, we get that these roots 
$r_1,\ldots, r_n$ coincide with the roots $r_1,\ldots , r_n$,
up to a permutation.
\end{proof}

We also have the following result which is a $\smax$ version of
\Cref{lemma-max-bmax} and is obtained using this lemma together
with \Cref{cor-mult-smax}.

\begin{theorem}\label{prop-sum-mult-n} 
Let ${P} \in \smax^\vee[\Y]$ of degree $n$, and
having $n$ roots $r_1,\ldots , r_n$, counted with multiplicity.
Then, for all non-zero corners $c\in \tmax$ of $|P|$, and
for all successive degrees $k<\ell$ of non-zero monomials of $P^{\sat,c}$,
we have either $\ell-k=1$, or $\ell-k=2$ and $P^{\sat,c}_\ell=\ominus 
P^{\sat,c}_k$, or equivalently $P_\ell\odot P_k\in \smax^\ominus$.
Conversely, if the latter property holds, then $P$ has  $n$ roots in $\smax$ 
counted with multiplicities.
\end{theorem}

\begin{remark}
Recall that a formal polynomial over $\tmax$ is factored if and
only if its coefficient map is concave. Therefore,
if $Q\in\bmaxs[\Y]$, the polynomial $|Q|$ is factored 
if and only if $|Q_k|=\unit$ for all $k$ between lower degree and degree of
$Q$. This implies that the condition that $|Q|$ is factored 
implies the condition in \Cref{lemma-max-bmax}.
Similarly, for a polynomial $P$ over $\smax$, the condition
that $|P|$ is factored implies the condition in \Cref{prop-mult-fact}.
Hence, \Cref{prop-mult-fact} is stronger than \Cref{suf_cond}. 
Note however that our proof of \Cref{prop-mult-fact} uses \Cref{suf_cond}.
\end{remark}

As suggested in \Cref{sec-factsym}, using \Cref{coro-uniquefact}, one may
define another notion of multiplicity of roots in case of unique factorization.
The following result shows that both notions coincide in the
framework of  \Cref{coro-uniquefact}.

\begin{theorem}[Multiplicities and unique factorization]\label{coro2-uniquefact}
Let ${P} \in \smax^\vee[\Y]$.
Assume that $|{P}|$ is factored (see \Cref{roots_poly}),
and let the $r_i$ be as in \Cref{suf_cond}.
If all the $r_i$ with same modulus are equal,
or equivalently if for each corner $c\neq \zeror$ of $|{P}|$,
$c$ and $\ominus c$ are not both roots of $P$,
then the multiplicity of a root $r$ of $P$ coincides with the 
number of occurences of $r$ in the unique factorization of $\widehat{P}$.
\end{theorem}
\begin{proof}
By \Cref{suf_cond}, $\widehat{P}$ is factored and by \Cref{coro-uniquefact} the factorization of  $\widehat{P}$ is unique, and
given by the $r_i$ of \Cref{suf_cond}, which are the only possible roots of 
$P$. 
In particular any root $r$ of $P$ belongs to $\{r_i,i=1,\ldots, n\}$,
and is such that $|r|$ is
a corner of $|P|$, and if $r\neq \zero$, then 
$r$ and $\ominus r$ are not both roots of $P$.
Moreover, from  \Cref{suf_cond}, 
 $|r_1|\geq \ldots \geq |r_n|$ are the corners of $|P|$ counted
with multiplicities, hence the number of occurences of $r$ in the unique 
factorization of $\widehat{P}$ is equal to the multiplicity of $|r|$ 
as a corner of $|P|$.

If $r=\zero$, the multiplicity of $r$ as a root of $P$ is the lower degree of
$P$, which is also the multiplicity of $r$ as a corner of $|P|$,
and thus the number of occurences of $r$ in the unique 
factorization of $\widehat{P}$.

Now, let $r\neq \zero$ be a root of $P$. 
We can reduce us to the case where $r=\unit$ and $|P(\unit)|=\unit$
by considering the polynomial $\Pn^r$ which satisfies the 
same properties as $P$. 
Since $|r_1|\geq \ldots \geq |r_n|$, we have 
$\unit=r_{k+1}=\cdots =r_{\ell}$ and $|r_i|>\unit$ for $i\leq k$ and
$|r_i|<\unit$ for $i>\ell$.
Using the definition of the $r_i$ in \Cref{suf_cond}, we deduce 
$P_{n-i+1}=\ominus P_{n-i}$ for $i= \{k+1,\ldots, \ell\}$,
$|P_{n-i+1}|<|P_{n-i}|$ for $i\leq k$, and
 $|P_{n-i+1}|>|P_{n-i}|$ for $i>\ell$.
This implies that $\sat(\unit, P)=\{n-\ell,\ldots, n-k\}$,
and $P^{\sat,\unit}= \bigoplus_{i=n-\ell}^{n-k} P_{i} \Y^{i}$.
By \Cref{cor-mult-smax}, the multiplicity of a root $r$ of $P$ is
equal to the number of sign changes in the coefficients of
$P^{\sat,r}$, which is here equal to $\ell-k$. This is also the number of 
occurences of $r=\unit$ in the sequence $(r_1,\ldots, r_n)$, 
and the multiplicity of the corner $\unit$ of $|P|$.
\end{proof}

\section{Application to signed valuations of roots over ordered non-Archimedean valued fields}\label{sec-valroots}

Let us consider now an ordered field $\rfield$ with a 
map $\sval:\rfield\to \smax=\smax(\vgroup)$,
 which is a morphism of semiring systems 
as in \Cref{def-morphism-systems},  for some ordered group $\vgroup$.
As said in \Cref{sec-signval},  $\smax$ is a hyperfield system associated to
a hyperfield structure on  $\smax^\vee$, and since $\sval$ is a morphism
of systems, then it is also a homomorphism of hyperfields from the field 
$\rfield$ to the hyperfield $\smax^\vee$. Baker and Lorscheid proved a general
inequality relating the multiplicities of the roots of a
polynomial over a field with the multiplicities of the roots
of the homomorphic image of the polynomial over a hyperfield.
When applied to the special case of the hyperfield $\smax^\vee$,
we arrive at the following result, which holds for any ordered group $\vgroup$.

\begin{theorem}[Corollary of \protect{\cite[Prop.~B]{baker2018descartes}}]\label{most_roots}
Let $\rfield$ be an ordered field with a morphism
$\sval:\rfield\to \smax^\vee$, as in \Cref{def-morphism-systems}.
Let $\bp=\sum_{k=0}^{n} \bp_k \Y^k\in \rfield[\Y]$ be a formal polynomial over $\rfield$,
in one variable $\Y$. Denote $P=\sval(\bp):=\bigoplus_{k=0}^{n} \sval(\bp_k) \Y^k \in \smax^{\vee}[\Y]$. Use \Cref{def-mult-BL} for the multiplicities in the semiring systems $\smax$ and $\rfield$
 (these are the usual ones in the field $\rfield$).
\begin{enumerate}
\item We have %
\begin{equation}\label{inq-mult}
  \mult_{\rsmax}(P)\geq \sum_{\elf\in\rfield\; \sval(\elf)={\rsmax}} \mult_{\elf}(\bp)\qquad 
\forall {\rsmax}\in\smax^\vee\enspace. \end{equation}
\item \label{eq-mult} Suppose in addition that $\bp$ is factored into linear factors : $\bp=(\Y-{\elf}_1)\cdots (\Y-{\elf}_n)$ and
that $\sum_{{\rsmax}\in \smax^\vee} \mult_{\rsmax}(P)\leq n$, then
the inequalities in \eqref{inq-mult}
are equalities, which means that $\sval({\elf}_1),\ldots, \sval({\elf}_n)$ are the roots of $\sval(\bp)$ counted with multiplicities.
\end{enumerate}
\end{theorem}
The following results will generally require $\vgroup$ to be divisible. 
This means that they cannot be applied to $\Z$, which is the valuation group of
the field of formal or convergent Laurent series.
The first one is deduced easily from \eqref{eq-mult} in 
\Cref{most_roots} together with \Cref{atmostn} and \Cref{prop-sum-mult-n}.
\begin{corollary}\label{most_roots2}
Let $\rfield$ be an ordered field with a morphism
$\sval:\rfield\to \smax^\vee$ as in \Cref{def-morphism-systems}, and
assume that $\vgroup$ is a divisible ordered group.
If $\bp\in\rfield[\Y]$ is a formal polynomial over $\rfield$ of degree $n$,
with exactly $n$ roots counted with multiplicities, ${\elf}_1, \ldots, {\elf}_n$, 
or equivalently $\bp=(\Y-{\elf}_1)\cdots (\Y-{\elf}_n)$,
then $P=\sval(\bp)$ has the $n$ roots in $\smax$,
$\sval({\elf}_i)$, $i=1, \ldots, n$, counted with multiplicity.

In particular, for all non-zero corners $c\in \tmax$ of $|P|$, and
for all successive degrees $k<\ell$ of non-zero monomials of $P^{\sat,c}$,
we have either $\ell-k=1$, or $\ell-k=2$ and $P^{\sat,c}_\ell=\ominus 
P^{\sat,c}_k$, or equivalently $P_\ell\odot P_k\in \smax^\ominus$.
\end{corollary}

Let us give some examples when $\rfield=\puiseuxseries{\R}$, the field of (formal) Puiseux series with real coefficients (see \Cref{ex-Puiseux}).

\begin{example}\label{ex-factored}
 As an example, let us consider the polynomial
 \begin{eqnarray*}
 \bp&=&(\Y+t^5)(\Y-t)^3(\Y-t^{-2})\\
&=& \Y^5 +(t^5-3t-t^{-2})\Y^4-(3t^6+t^3-3t^2-3t^{-1})\Y^3+(3t^7+3t^4-t^3-3)\Y^2\\
&&\quad  -(t^8+3t^5-t)\Y+t^6\enspace,
 \end{eqnarray*}
 with coefficients in $\R\{\{t\}\}$. Clearly, $\elf_1=-t^{5}, \elf_2=t$ and $\elf_3=t^{-2}$ are the roots of $\bp$ with multiplicities $1, 3$ and $1$, respectively. Now using signed valuation we have
  $P=\sval(\bp)=\Y^5\oplus 5\Y^4 \ominus 6\Y^3\oplus 7\Y^2\ominus 8\Y\oplus 6$
   which has the roots
     $\rsmax_1= \ominus 5$, $\rsmax_2=1$ and $\rsmax_3 = -2$.
  Here, we have equalities in \eqref{inq-mult} since $\bp$ is factored.
Moreover, \Cref{most_roots2} shows that
$\mult_{\ominus 5}(P)=1, \; \mult_{1}(P)=3$ and $\mult_{-2}(P)=1$.
We can recover also this by using \Cref{cor-mult-smax}.
\end{example}
\begin{example}
Consider now
\begin{eqnarray*}
\bp&=&(\Y+t^5)(\Y-t)(\Y-t^{-2})(\Y^2+t^2)\\
&=& \Y^5 +(t^5-t-t^{-2})\Y^4-(t^6+t^3-t^2-t^{-1})\Y^3+(t^7+t^4-t^3-1)\Y^2\\
&&\quad -(t^8+t^5-t)\Y+t^6\enspace , 
\end{eqnarray*}
which cannot be factored in linear factors in real Puiseux series.
We have that $P=\sval(\bp)$ is the same as in \Cref{ex-factored}, 
however, we do not have equality in \eqref{inq-mult} for root $r_2=1$,
since $\mult_{1}(P)=3$ whereas the only root $\elf$ of $\bp$ such that
$\sval(\elf)=r_2=1$ has multiplicity $1$.
\end{example}

Using \Cref{prop-mult-fact}, we also obtain the following result.
\begin{corollary}\label{fact_roots}
Let $\rfield$ be an ordered field with a morphism
$\sval:\rfield\to \smax^\vee$ as in \Cref{def-morphism-systems}, and
assume that $\vgroup$ is a divisible ordered group.
Let $\bp\in\rfield[\Y]$ be a formal polynomial over $\rfield$ of degree $n$,
with exactly $n$ roots counted with multiplicities, ${\elf}_1, \ldots, {\elf}_n$, 
or equivalently assume that $\bp$ is factored as $\bp=(\Y-{\elf}_1)\cdots (\Y-{\elf}_n)$,
and let $P=\sval(\bp)$. Then, $\widehat{P}$ is factored as
$\widehat{P}(y)=(\Y\ominus\sval({\elf}_1))\cdots (\Y\ominus\sval({\elf}_n))$.
\end{corollary}

\begin{remark}
Note that in \Cref{fact_roots}, $P=\sval(\bp)$ may not be factored as a formal polynomial.
Indeed, consider $\rfield=\puiseuxseries{\R}$, %
and the polynomial $\bp= \Y^2-1$ with roots $1$ and $-1$. Then, $P=\Y^2\ominus \unit$, which has $\unit$ and $\ominus \unit$ as roots, but
$P\neq (\Y\ominus \unit)(\Y\oplus \unit)$ (as a formal polynomial).
\end{remark}

\begin{remark}
Note that part of the result of \Cref{fact_roots} can be obtained easily 
using the same techniques as in \cite{akian2016non},
at least when $\vgroup=\R$.
Indeed, using
$\bp_{n-k}=(-1)^{k}\sum_{i_1<\cdots < i_k} {\elf}_{i_1}\cdots {\elf}_{i_k}$,
we deduce that 
$\sval(\bp_{n-k})\preceq^\circ (\ominus\unit)^{\odot k}\bigoplus_{i_1<\cdots < i_k}
\sval({\elf}_{i_1})\odot \cdots\odot \sval({\elf}_{i_k})$,
so $P\preceq^\circ \tropprod_{1\leq i\leq n} (\Y\ominus \sval({\elf}_i))$.
Then,
$\widehat{P}(y)\preceq^{\circ} \tropprod_{1\leq i\leq n} (y\ominus \sval({\elf}_i))$,
with equality for all $y$ such that the right hand side is in $\smax^\vee$,
in particular for  $y\in \smax^\vee\setminus\{\sval({\elf}_i)\mid i=1, \ldots, n\}$.
If $\vgroup=\R$, then, by continuity, we conclude that the modulus of 
both sides are equal for
$y\in\smax$. %
However, it is not so easy to obtain the equality everywhere
without using the results leading to \Cref{fact_roots}.
\end{remark}

We now suppose that $\rfield$ is a real closed field with a non-trivial convex non-Archimedean valuation. As said in the preliminaries, this implies that the value group $\vgroup$ is divisible. %
Moreover, by \Cref{rem-morphism-systems}, a morphism 
$\sval$ as in \Cref{most_roots} (that is as in \Cref{def-morphism-systems})
is the signed valuation associated, as in \Cref{def-sign},
to the non-trivial convex non-Archimedean valuation $\vall$ such that
$\vall(\elf)=|\sval(\elf)|$, for all $\elf\in \rfield$.
 Then, \Cref{most_roots} can be refined:
\Cref{th-descartes} below shows that the bound of~\Cref{most_roots} for the sum of multiplicities of the roots with a prescribed signed valuation is exact
 modulo 2 (part 1), and that it is tight (part 2).  In~\cite[Theorem B]{gunn},
 Gunn proved a similar result to the first part of \Cref{th-descartes}
in the case where $\vgroup=\R$ 
(rank $1$ valued fields). Our result holds for any real closed field with 
a non-trivial convex valuation, which implies that the 
non-trivial divisible ordered group $\vgroup$ is arbitrary.
As said in the introduction, the question of the 
tightness of the bounds of~\Cref{most_roots} was raised by Baker and Lorscheid \cite[Remark 1.14]{baker2018descartes} for general morphisms of hyperfields, and is already true for when the group $\vgroup$ is trivial.
Gunn gave a partial answer  in  \cite[Th.~5.4]{gunn}, where
a weak version of the second part of \Cref{th-descartes} is shown in the
case  where $\vgroup=\R$. Namely, this result is equivalent to 
the property that for any polynomial $P$ over $\smax^\vee$, 
there exists a polynomial $\bp$ over the real closed field  $\rfield$ of Hahn series
with signed valuation equal to $P$ and
satisfying equalities in \eqref{inq-mult}, %
with the sum over $\elf\in\rfield$ such that $\sval(\elf)={\rsmax}$
enlarged into a sum over
all elements $\elf$ of the algebraically closed complex field of Hahn series,
that have a real dominant term with a signed valuation equal to $\rsmax$.
Then, the tightness result shown in the second part of \Cref{th-descartes}
refines the one of~\cite[Th.~5.4]{gunn} in the case where $\vgroup=\R$, answering to a question raised there.
Moreover,  as for the first part, 
 our result holds for any real closed field with 
a non-trivial convex valuation, and for an arbitrary
non-trivial divisible ordered group $\vgroup$.
\begin{theorem}[Descartes' rule for signed valuations]
  \label{th-descartes}
  Let $\rfield$ be a real closed field with a non-trivial convex non-Archimedean
  valuation, and let $\smax=\smax(\vgroup)$ be constructed
  from the value group $\vgroup$ of $\rfield$, and let
 $\sval:\rfield\to \smax^\vee$ denote the signed valuation
  map, as in \Cref{def-sign}.
(Equivalently, let $\rfield$ be a real closed field and let 
$\sval:\rfield\to \smax^\vee$, as in \Cref{def-morphism-systems},
with $\smax=\smax(\vgroup)$ for a non-trivial group $\vgroup$.)

  Let
  \[ P=\bigoplus_{k=0}^n P_k\Y^k\in \smax^\vee[\Y]
  \enspace.
  \]
Then,
\begin{enumerate}
\item \label{th-descartes1}For all polynomials $\bp\in \rfield[\Y]$
  such that $\sval(\bp)=P$,
  we have 
\begin{equation}\label{inq-mult2}
  \mult_{\rsmax}(P)\geq \sum_{\elf\in\rfield\; \sval(\elf)={\rsmax}} \mult_{\elf}(\bp)\qquad 
\forall {\rsmax}\in\smax^\vee\enspace, \end{equation}
and the two terms of the inequalities~\eqref{inq-mult2} coincide
  modulo $2$.
\item \label{th-descartes2}There exists $\bp\in \rfield[\Y]$ such that $\sval(\bp)=P$
and, for all $ {\rsmax}\in\smax^\vee$, the inequalities in~\eqref{inq-mult2} are equalities.
\end{enumerate}
\end{theorem}

The proof of \Cref{th-descartes} is organized in two parts. First, we establish
the result in a special case, assuming that
$\rfield:= \puiseuxseries{\R}$ is the field of (formal) Puiseux series
 (see \Cref{ex-Puiseux}).
Then, we will deduce \Cref{th-descartes} in full generality from this special
case, by applying a model theory argument.

\begin{proof}[Proof of \Cref{th-descartes}, Part I (case $\rfield= \puiseuxseries{\R}$)]

The inequalities~\eqref{inq-mult2} follow from \Cref{most_roots} (and are true 
actually in any ordered field $\rfield$).

Proof of \eqref{th-descartes1}:
Let us show the ``modulo 2'' part.
When ${\rsmax}$ is not a root of $P$, we have $\mult_{\rsmax}(P)=0$, so the equality holds
 in~\eqref{inq-mult2}. Moreover, if ${\rsmax}=\zero$ is a root of $P$, then
$\mult_{\rsmax}(P)$ is the lower degree of $P$ (see \Cref{cor-mult-smax}), which is also the lower degree of 
$\bp$ (since $\elf=0$ iff $\sval(\elf)=\zero$ for all $\elf\in \rfield$),
and thus the multiplicity of the root $0$ of $\bp$. This shows the
equality in~\eqref{inq-mult2} for the case ${\rsmax}=\zero$.

We now consider the case in which ${\rsmax}\in \smax^\vee\setminus\{\zero\}$ is a root of $P$. Then, $|{\rsmax}|$ is a corner of $|P|=\bigoplus_{k=0} |P_k|\Y^k\in \tmax[Y]$,
see \Cref{abs_roots}.
We denote by $m$ the multiplicity of the corner $|{\rsmax}|$ of $|P|$. 
For all $s=\sum_{\alpha\in \Q} s_\alpha t^\alpha 
\in \rfield$, or $\rfield[\sqrt{-1}]$,
we denote by $\vall(s)$ the valuation of $s$,
that is the maximal element of its
support, and by $\dc(s)$ the dominant coefficient of $s$, that is 
$s_{\vall(s)}$.
The field $\rfield[\sqrt{-1}]$ is known to be
algebraically closed~\cite{ribenboim}.
  Moreover, the Newton-Puiseux theorem entails that
  $\bp$ has precisely $m$ roots (counted with multiplicities) ${\elf}_1,\dots,{\elf}_m$ of valuation $|{\rsmax}|$ over $\rfield[\sqrt{-1}]$, and that the dominant coefficients of these roots, $\dc({\elf}_1),\ldots, \dc({\elf}_m)$, are precisely the non-zero roots of the following polynomial with  real coefficients,
  \[
  Q(\Y) = \sum_{k\in \sat({\rsmax},P) } \dc(\bp_k)\Y^k\enspace,
  \]
  where $\sat({\rsmax},P)$  is defined as in \Cref{cor-mult-smax}.
  (Note that the polynomial is called the {\em initial form} of
  $\bp$ --with respect to the weight $|{\rsmax}|$-- in the tropical
  geometry literature.)
  More precisely, the characterization of the exponents and dominant coefficients of the roots of $\bp$ which we just stated follows
  from the analysis of the Newton-Puiseux algorithm (see Theorem~3.2
  in~\cite{TEISSIER2007} or Theorem~6 in~\cite{markwig2009field}),
  or Theorem~6.2 in~\cite{akian2016non}).

Assume first that $b$ is positive, that is 
${\rsmax}\in \smax^{\oplus}\setminus\{\zero\}$.
By \Cref{cor-mult-smax}, then the multiplicity $\mult_{\rsmax}(P)$
of the root ${\rsmax}$ of $P$
is precisely the number of sign changes of the polynomial
$P^{\sat,{\rsmax}}= P^{\sat,|{\rsmax}|}$. The supports of $P^{\sat,{\rsmax}}$ and $Q$ are 
both equal to $\sat({\rsmax},P)$ and since ${\rsmax}$ is positive,
the sign of the coefficients are the same:
$\sign(\dc(\bp_k))=\sign(\bp_k)=P^{\sat,|{\rsmax}|}_k$,
since $P_k=\sign(\bp_k)\vall(\bp_k)$.
So $\mult_{\rsmax}(P)$ is the number of sign changes of the polynomial $Q$.

By the Descartes' rule of sign for real polynomials,
the number of real positive roots of $Q$ is equal modulo $2$
to the number of sign changes of $Q$.
The positive roots of $Q$ can be partitionned in two lists:
(i) the ones which are the dominant coefficients of a root
of $\bp$ in $\rfield$; (ii) the ones which are the dominant
coefficients of a root $\elf$ of $\bp$ in $(\rfield[\sqrt{-1}])\setminus \rfield$
(i.e., of a root $\elf$ that has a non-zero imaginary part).
The roots of type (ii) come by pairs, since if $\elf$ is a root of $\bp$,
its conjugate $\bar{\elf}$ is also a root of $\bp$. The roots of $Q$ of type (i)
are precisely the ones that are the dominant coefficients of
roots of $\bp$ in $\rfield$ which are positive and 
with valuation $|{\rsmax}|$, or equivalently which have a signed valuation equal 
to ${\rsmax}$. It follows
that the number of roots of $\bp$ in $\rfield$ with signed valuation
${\rsmax}$ is equal modulo $2$ to the number of real positive roots of $Q$,
which coincides with the multiplicity $\mult_{\rsmax}(P)$ of the root ${\rsmax}$ of $P$.

By replacing $\bp$ by $\bp(-\Y)$, we arrive at the same
conclusion for the multiplicity of the root ${\rsmax}$ of $P$,
when ${\rsmax}$ is negative that is ${\rsmax}\in\smax^{\ominus}\setminus\{\zero\}$.
This shows the ``modulo $2$'' statement.

Proof of \eqref{th-descartes2}:
We now show that the 
inequalities in~\eqref{inq-mult2} are equalities, for some
$\bp\in \rfield[\Y]$ such that $\sval(\bp)=P$.
 We use the following
  technique known as Viro's method in tropical geometry~\cite{itenberg_viro}.
Consider a formal polynomial $P=\bigoplus_{k=0}^n P_k\Y^k\in \smax^\vee[\Y]$
with degree $n$.
  We introduce a positive real parameter $u$ (distinct from the indeterminate $t$
  of Puiseux series), take any strictly concave map $\omega$
  from $\{0,\dots,n\}$ to $\Q$, and consider the polynomial
  \[
  \bp := \sum_{k=0}^n \epsilon_k  u^{\omega_k} t^{|P_k|}\Y^k \in \R[\Y]
  \]
  where, $t^{-\infty}=0$, and for all $k\geq 0$,  $\epsilon_k= 1$ if $P_k$ is a positive element
  of $\smax$, $\epsilon_k= -1$ if $P_k$ is a negative element
  of $\smax$, and $\epsilon_k=\pm 1$ if $P_k=\zero$,
so that $\sval (\bp)=P$.
For ${\rsmax}=\zero$, we already know that the 
inequality in~\eqref{inq-mult2} is an equality.
Now consider a root ${\rsmax}\in \smax^\vee\setminus\{\zero\}$ of $P$,
and let $m$ be the multiplicity of the corner $|{\rsmax}|$ of $|P|$.
If ${\rsmax}$ is positive, then the associated polynomial $Q$ defined above is equal to
  \[
  Q(\Y) := \sum_{k\in \sat({\rsmax},P) } \epsilon_k u^{\omega_k}\Y^k  \enspace ,
  \]
and its degree minus its lower degree is equal to $m$.
  We shall see $u$ as a large parameter, hence, the coefficients
  of $Q$ will be interpreted as elements of the field
  $\R\{\{u\}\}_{\textrm{cvg}}$ of Puiseux series
  that are absolutely convergent for small enough
  values of the variable $u$. Using this interpretation,
the signed valuation of $Q$ is equal to the polynomial
 $R(\Y) := \oplus_{k\in \sat({\rsmax},P) } R_k\Y^k \in \smax^\vee[\Y]$,
with $R_k =\omega_k$ if $\epsilon_k=1$ and $R_k =\ominus \omega_k$ if $\epsilon_k=-1$.
Since $\omega$ is strictly concave, each corner $c$ of $|R|$ is such that
$\sat(c,R)$ is composed of two indices $k<\ell$  only, with $\ell-k$ 
equal to the multiplicity of $c$ as a corner of $|R|$.
Moreover, for any corner $c$ of $|R|$, $c$ is a root of $R$ 
if and only if $R_k$ and $R_\ell$ have opposite signs and then, the multiplicity is $1$ (the
multiplicity of $c$ as a root of $R$ is equal to the number of sign changes
of the saturation polynomial $R^{\sat,c}=R_k\Y^k\oplus R_\ell Y^\ell$),
and similarly $\ominus c$ is a root of $R$ if and only if 
$R_k$ and $R_\ell(\ominus \unit)^{\ell-k}$ have opposite signs.
So the multiplicity of any root of $R$ is at most $1$.

By applying the Newton-Puiseux
  theorem to the polynomial $Q\in\C\{\{u\}\}_{\textrm{cvg}}$,
we get that %
the polynomial $Q$ has exactly $m$ non-zero roots 
(counted with multiplicities) in $\C\{\{u\}\}_{\textrm{cvg}}$,
and applying \eqref{th-descartes1} of the theorem, 
we have that for all $\rsmax\in \smax^\vee$, 
 $ \mult_{\rsmax}(R)\geq \sum_{\elf}  \mult_{\elf}(Q)$, where the sum is over all
$\elf\in\R\{\{u\}\}_{\textrm{cvg}}$ with a signed valuation equal to $\rsmax$,
and the two terms coincide modulo 2.
Since the roots of $R$ have multiplicity $1$, we get that the equality
holds for all $\rsmax\in \smax^\vee$.
Then, the real roots of $Q$ in 
$\R\{\{u\}\}_{\textrm{cvg}}$ have their signed valuations equal to 
non-zero roots of the polynomial $R$, and these are all simple roots.

Specializing this result for a sufficiently large value of $u$, we get that 
all real roots of $Q$ (seen as an element of $\R[\Y]$) are simple roots.
Moreover, the number of positive roots of $Q$ is equal to the number of positive roots of $R$, that is the number of sign changes of $R$, so of $Q$.
By the same arguments as in the proof of the first item,
 this number is equal to $\mult_{\rsmax}(P)$.
Moreover, for any root $r$ of $\bp$ in $\rfield$
that has a signed valuation equal to ${\rsmax}$, its dominant coefficient $\dc(r)$ 
is a positive real root of $Q$. Since the real roots of $Q$ are simple, 
we get that the number of such roots $r$ of $\bp$ is exactly equal to 
the number of positive roots of $Q$, that is $\mult_{\rsmax}(P)$.
This shows the equality in \eqref{inq-mult2} for ${\rsmax}$.

  The same conclusion applies to all signed roots ${\rsmax}$ of $P$.
  Specializing $u$ to a suitably large value, we get
  a lift $\bp$ of $P$ which achieves the equality in~\eqref{inq-mult2},
  for all roots of ${\rsmax}$ of $P$.

\end{proof}
\begin{proof}[Proof of \Cref{th-descartes}, Part II (derivation of the general case)]
The inequality \eqref{inq-mult2} in
Point \eqref{th-descartes1} was already proved in full generality.
It suffices to show that the ``modulo 2'' part
of Point \eqref{th-descartes1} and 
Point \eqref{th-descartes2} of \Cref{th-descartes}
can be expressed as statements in the (first order) theory $\thVrcf$  of real ordered fields with angular component that are real closed
and have a non-trivial valuation. Then, since
this theory is complete (\Cref{theorem:qe_vrcf}), the validity of \Cref{th-descartes} in its full generality will follow from the validity
  of the same statements when $\rfield=\R\{\{t\}\}$, which we
  just established. 
  To see this, we must fix the degree $n$, so that the statements of \Cref{th-descartes}   can be considered as statements involving the variables
  ${\rsmax},P_0,\dots,P_n\in \smax^\vee$ and $\elf,\bp_0,\dots,\bp_n\in \rfield$. Then, we observe that every variable $a$ ranging in $\smax^\vee$ can be replaced by a pair of variables $(\epsilon_a,\mu_a)$
  ranging in $\{\unit,\ominus \unit,\zero\}\times (\Gamma\cup\{\bot\})$, interpreting $\epsilon_a$ as the sign of $a$ and $\mu_a$
  as its modulus. In that way, \Cref{th-descartes} can be rewritten as an equivalent statement involving the variables encoding the signs and moduli of ${\rsmax},P_0,\dots,P_n$, as well as the variables $\elf,\bp_0,\dots, \bp_n$. In the transformed statement,   we do not need to use any more the symbol ``$\sval$'',
  and use the symbols ``$\vall$'' and ``$>$'' instead to express
  the same property.
  Hence,
  modulo the abuse of terminology implied by this rewriting,
  it does makes sense to consider the
membership of this statement to the theory $\thVrcf$.

  We observe that for a fixed integer $m$, 
  the statement $\mult_{\elf}(\bp)=m$ with ${\elf}\in \rfield$,
can be expressed in the theory $\thVrcf$.
Indeed, in a field $\mult_{\elf}(\bp)=m$ holds if and only if
there exists a polynomial $\mathbf{Q}$ with degree $n-m$ such that
  \[ \bp= \mathbf{Q} (\Y-{\elf})^m  \;\text{and}\; \mathbf{Q}({\elf})\neq 0 \enspace,
  \]
  and the latter can be rewritten as a statement involving
$\elf,\bp_0,\ldots, \bp_n,\mathbf{Q}_0,\ldots ,\mathbf{Q}_{n-m}$
and belonging to $\thVrcf$.

 Moreover, again for a fixed integer $m$, the statement $\mult_{\rsmax}(P)=m$ can be expressed in the theory $\thVrcf$.
Indeed, using \Cref{cor-mult-degree}, this statement is equivalent to 
the properties that there exists $m$ polynomials $Q^{(1)}, \ldots, Q^{(m)}\in\smax^\vee[\Y]$
of degree $n-1, \ldots, n-m$
such that $Q^{(i)} \balance (\Y \ominus {\rsmax}) Q^{(i+1)}$
for $i=0,\ldots , m-1$ with $Q^{(0)}=P$, 
and there does not exist $m+1$ polynomials $Q^{(1)}, \ldots, Q^{(m+1)}\in\smax^\vee[\Y]$ with same properties.
 Then, we observe that every equation $Q \balance (\Y \ominus {\rsmax}) Q'$
with $Q,Q'\in \smax^\vee[\Y]$ with degree $j,j-1$, with $j\in\{1,\ldots, n\}$
 can be written as a system of $j+1$ equations, namely 
$Q_\ell\balance Q'_{\ell-1}\ominus \rsmax\odot Q'_{\ell}$, 
for $\ell=1,\ldots, j-1$, together with $Q_j= Q'_{j-1}$ and 
$Q_0= \ominus \rsmax\odot Q'_{0}$.
Moreover, these equations are of the form
$a\balance b\oplus c$ in $\smax$ for variables $a,b,c$ ranging in $\smax^\vee$,
and they can be replaced by an expression involving the variables 
$(\epsilon_a,\mu_a)$, $(\epsilon_b,\mu_b)$ and $(\epsilon_c,\mu_c)$
all  ranging in $\{\unit,\ominus \unit,\zero\}\times (\Gamma\cup\{\bot\})$,
as above.

  Observe that the equality $m=\sum_{\elf\in\rfield\; \sval(\elf)={\rsmax}} \mult_{\elf}(\bp)$ in~\eqref{inq-mult2} can be rewritten
  as the disjunction $\formF_m(\bp,{\rsmax}):=
\bigvee_{1\leq k\leq m,m_1+\dots +m_k= m}
  \formF_{k,m_1,\dots,m_k}(\bp,{\rsmax})$,
  in which $\formF_{k,m_1,\dots,m_k}(\bp,{\rsmax})$
stands for the following statement:
``$\bp$ has precisely $k$ distinct roots with signed valuation ${\rsmax}$, and the multiplicities of these roots, up to a reordering, are given by $m_1,\dots,m_k$''. The latter statement can be expressed in $\thVrcf$ by introducing
$k+1$ additional variables ${\elf}_1,\dots,{\elf}_{k+1}$ ranging in $\rfield$, so that it can be replaced by 
``$\bp$ has $k$ distinct roots ${\elf}_1,\dots,{\elf}_{k}$
with signed valuation ${\rsmax}$ and respective multiplicities $m_1,\dots,m_k$, and ``$\bp$ does not have $k+1$ distinct roots ${\elf}_1,\dots,{\elf}_{k+1}$ with signed valuation ${\rsmax}$''.
Then  the equality in ~\eqref{inq-mult2} is equivalent to the disjunction 
$\bigvee_{0\leq m \leq n}(\formF_m(\bp,{\rsmax})\wedge (\mult_{\rsmax}(P)=m))$.
It follows that the statement of \eqref{th-descartes2} of \Cref{th-descartes}
can be expressed in $\thVrcf$.

Similarly, the second part of statement \eqref{th-descartes1} can also be expressed in $\thVrcf$, by considering this time the disjonctions
$\formF'_m(\bp,{\rsmax}):\bigvee_{1\leq k\leq m,1\leq \ell\leq m, m_1+\dots +m_k+2 \ell= m}
  \formF_{k,m_1,\dots,m_k}(\bp,{\rsmax})$. 

\end{proof}

\begin{remark}
  Observe that, by~\Cref{most_roots}, the inequality~\eqref{inq-mult2}
  is valid for any ordered valued field with the signed valuation of
\Cref{def-sign}, since the approach
  of~\cite{baker2018descartes}
  via morphisms of hyperfields applies to this general
  setting. In contrast,
  the proof that the inequality is ``tight'' (i.e., attained
  by some lift and exact modulo $2$) requires the field
  to be {\em real closed}.
  \end{remark}

\appendix
\section{Direct proof of \Cref{max-lemma}}
We give here a proof of \Cref{max-lemma} for a general divisible group, without using the model theory argument.
Although $\R$ is replaced by $\vgroup$, one can still see the maps
$- P\mapsto \widehat{P}$ and $\widehat{P}\mapsto -  P^\bifench$ 
as Fenchel-Moreau conjugacies. They operate 
on functions defined on $\N$ and $\vgroup$ respectively,
with the kernel $c(k,y)=ky$. So by
the usual properties of Fenchel-Moreau conjugacies, if $P$ is non-zero,
meaning that the function $k\in \N\mapsto - P_k$ has a nonempty support, 
we have that $P\leq P^\bifench$ and $\widehat{P^\bifench}=\widehat{P}$.
So $P^\bifench$ is the maximum of the formal polynomials $Q$ such that $\widehat{Q}=\widehat{P}$ for the (coefficientwise) order $\leq$.

Let $P$ has lower degree $m$ and degree $n$.
Since $-P$ has a nonempty finite support with rational convex hull equal to 
$[m,n]\cap \Q$, we deduce that $k\in \N\mapsto -P^\bifench_k$ is convex 
with support equal to $\{m,\ldots, n\}=[m,n]\cap \N$.

Let us show that if  $k\in \N\mapsto P_k$ is concave,
we have $P^\bifench=P$. The proof is elementary and comes from
the fact that the functions are all polyhedral.
If $m=n$, this is trivial.
Otherwise, since $P$ is concave, we have in particular
$P_{k}-P_{k-1}\geq P_{k+1}-P_k$ for all $k$ such that $m<k<n$ if any.
Consider $\ell\in \{m,\ldots,n\}$ and take 
$y\in \vgroup$ such that $y=P_\ell-P_{\ell+1}$ if $\ell<n$ and
$y=P_{\ell-1}- P_\ell$, if $\ell=n$.
We have $P_{k-1}-P_{k}\leq y$ for all $k\leq \ell$,
and $P_{k-1}-P_{k}\geq y$ for all $k\geq \ell+1$.
This implies that $\widehat{P}(y)= P_{\ell}+\ell y$, and 
so $P^\bifench_\ell\leq \widehat{P}(y)-\ell y= P_{\ell}$.
Since this holds for all $\ell\in [m,n]$, we get that 
$P^\bifench\leq P$, and since the other inequality is always true, we obtain
the equality $P^\bifench= P$, for all $P$ such that
 $k\in \N\mapsto P_k$ is concave. 
The above proof also shows that in that case 
the infimum in the right hand 
side of \eqref{p-sharp2} is attained by some $y\in \vgroup$,
when $k\in \{m,\ldots,n\}$.
Moreover, since  $\widehat{P^\bifench}= \widehat{P}$, 
and  $k\in \N\mapsto P^\bifench_k$ is concave, the latter property also holds
for any $P$.

Using the definition of $P^{\sharp}$, we see also that 
$P^{\sharp}\geq P$, that it is concave on $\N$ with a support included in $\{m,\ldots, n\}$, and that $P^{\sharp}= P$ when $P$ is concave.
In particular we have $P^\bifench= {P}^{\sharp}$ when  $P$ is concave.

It remains to show that $P^\bifench= {P}^{\sharp}$ for a general non-zero 
formal polynomial $P$.
We already showed that
$P^\bifench\geq P$ and ${P}^{\sharp}\geq P$. 
Applying the order-preserving 
maps $P\mapsto {P}^{\sharp}$ and $P\mapsto P^\bifench$ to these inequalities,
we obtain
\begin{equation}\label{ineq-sharp}
(P^\bifench)^{\sharp}\geq {P}^{\sharp} \quad \text{and} \quad 
({P}^{\sharp})^\bifench\geq P^\bifench\enspace.\end{equation}
Since we also have that $P^\bifench$ and $P^\sharp$ are both concave, 
the above results show that 
$(P^\bifench)^{\sharp}=P^\bifench$, and  $({P}^{\sharp})^\bifench={P}^{\sharp}$.
Then, from \eqref{ineq-sharp}, we get
$P^\bifench= (P^\bifench)^{\sharp}\geq {P}^{\sharp}$,
and ${P}^{\sharp}=({P}^{\sharp})^\bifench\geq P^\bifench$, hence
$P^\bifench= {P}^{\sharp}$.
%


\end{document}